\documentclass[12 pt]{amsart}
\usepackage{amsmath}
\usepackage{amssymb}
\usepackage{amscd}
\usepackage{amsmath}
\usepackage{enumerate}
\usepackage{mathtools}
\usepackage{dashrule}
\usepackage{amssymb}
\usepackage{bm}
\usepackage{tikz}
\usepackage{hyperref}
\hypersetup{linktocpage}
\usepackage[T1]{fontenc}
\usepackage{ucs}
\usepackage{lmodern}
\usepackage{mathtools}
\usepackage[all]{xy}
\usepackage{ucs}
\usepackage{comment}
\usepackage[letterpaper, top = 2cm, bottom = 2cm,right = 2cm, left = 2cm]{geometry}

\numberwithin{equation}{section}

\theoremstyle{theorem}
\newtheorem{theorem}{Theorem}
\newtheorem{lemma}{Lemma}
\newtheorem{prop}{Proposition}
\newtheorem{cor}{Corollary}

\theoremstyle{definition}
\newtheorem{definition}{Definition}
\newtheorem{remark}{Remark}

\numberwithin{equation}{section}

\begin{document}


\title{Tensor algebras and Decorated representations}


\author{Raymundo Bautista and Daniel L\'{o}pez-Aguayo}
 \address{Centro de Ciencias Matem\'{a}ticas, Universidad Nacional Aut\'{o}noma de M\'{e}xico, Apartado Postal 61-3 (Xangari), Morelia, Michoac\'{a}n, Mexico 58089}
 \thanks{Supported by CONACyT Ph.D scholarship no. 358378}
 \email{raymundo@matmor.unam.mx}
\email{dlopez@matmor.unam.mx}

\maketitle

\begin{abstract} In \cite{2} we gave a generalization of the theory of quivers with potentials introduced by Derksen-Weyman-Zelevinsky, via completed tensor algebras over $S$-bimodules where $S$ is a finite dimensional basic semisimple algebra. In this paper we show how to extend this construction to the level of decorated representations and we prove that mutation of decorated representations is an involution. Moreover, we prove that there exists a nearly Morita equivalence between the Jacobian algebras which are related via mutation. This generalizes the construction given by Buan-Iyama-Reiten-Smith in \cite{4}.
\end{abstract}

\tableofcontents

\section{Introduction}

In this paper we extend the construction given in \cite{2} to the level of decorated representations of algebras with potentials realized via completed tensor algebras. Instead of working with a quiver, we consider the algebra of formal power series $\mathcal{F}_{S}(M)$ where $S=\displaystyle \prod_{i=1}^{n} D_{i}$ is a finite direct product of division rings containing the base field $F$ in its center and $M$ is a finite dimensional $S$-bimodule. In section \ref{sec3} we define a decorated representation of an algebra with potential $(\mathcal{F}_{S}(M),P)$ and in section \ref{sec4} we show how to associate a decorated representation to the premutated algebra with potential $(\mathcal{F}_{S}(\mu_{k}M),\mu_{k}P)$ and we prove this is indeed a decorated representation. In contrast to \cite{8} we do not assume that the basis is semi-multiplicative, but rather impose some conditions on the dual basis associated to the division algebra $Se_{k}$. In section \ref{sec5} we define mutation of a decorated representation, and we show it is a well-defined transformation on the set of right-equivalence classes  of decorated representations of algebras with potentials. A crucial result of this section is that mutation of decorated representations is an involution. Finally, in section \ref{sec6} we construct a functor to prove that there exists an Nearly Morita equivalence between the Jacobian algebras which are related via mutation.This construction gives a generalization of the one given in \cite{4}. 

\section{Preliminaries}

\begin{definition} Let $F$ be a field and let $D_{1},\hdots,D_{n}$ be division rings, containing $F$ in its center, and each of them is finite-dimensional over $F$. Let  $S=\displaystyle \prod_{i=1}^{n} D_{i}$ and let $M$ be an $S$-bimodule of finite dimension over $F$. Define the \emph{algebra of formal power series} over $M$ as the set: \\

\begin{center}
$\mathcal{F}_{S}(M):=\left\{\displaystyle \sum_{i=0}^{\infty} a(i): a(i) \in M^{\otimes i}\right\}$
\end{center}
\end{definition} 
where $M^{0}=S$. 

Note that $\mathcal{F}_{S}(M)$ is an associative unital $F$-algebra where the product is the one obtained by extending the product of the tensor algebra $T_{S}(M)=\displaystyle \bigoplus_{i=0}^{\infty} M^{\otimes i}$ to $\mathcal{F}_{S}(M)$.

Let $\{e_{1},\ldots,e_{n}\}$ be a complete set of primitive orthogonal idempotents of $S$.

\begin{definition} An element $m \in M$ is \emph{legible} if $m=e_{i}me_{j}$ for some idempotents $e_{i},e_{j}$ of $S$.
\end{definition}

\begin{definition} Let $Z=\displaystyle \sum_{i=1}^{n} Fe_{i} \subseteq S$. We say that $M$ is $Z$-\emph{freely generated} by a $Z$-subbimodule $M_{0}$ of $M$ if the map $\mu_{M}: S \otimes_{Z} M_{0} \otimes_{Z} S \rightarrow M$ given by $\mu(s_{1} \otimes m \otimes s_{2})=s_{1}ms_{2}$ is an isomorphism of $S$-bimodules. In this case we say that $M$ is an $S$-bimodule which is $Z$-\emph{free} or $Z$-freely generated.
\end{definition}

\begin{definition} Let $\mathcal{C}$ be a non-empty subset of $M$. We say that $\mathcal{C}$ is a \emph{right $S$-local basis} of $M$ if every element of $\mathcal{C}$ is legible and if for each pair of idempotents $e_{i},e_{j}$ of $S$ we have that $\mathcal{C} \cap e_{i}Me_{j}$ is a $Se_{j}=D_{j}$-basis of $e_{i}Me_{j}$.
\end{definition}

\begin{definition} Let $\mathcal{D}$ be a non-empty subset of $M$. We say that $\mathcal{D}$ is a \emph{left $S$-local basis} of $M$ if every element of $\mathcal{D}$ is legible and if for each pair of idempotents $e_{i},e_{j}$ of $S$ we have that $\mathcal{D} \cap e_{i}Me_{j}$ is a $Se_{i}=D_{i}$-basis of $e_{i}Me_{j}$.
\end{definition}

A right $S$-local basis $\mathcal{C}$ induces a dual basis $\{u,u^{\ast}\}_{u \in \mathcal{C}}$ of $M$ where $u^{\ast}: M_{S} \rightarrow S_{S}$ is the morphism of right $S$-modules defined by $u^{\ast}(v)=0$ if $v \in \mathcal{C} \setminus \{u\}$ and $u^{\ast}(u)=e_{j}$ if $u=e_{i}ue_{j}$. Similarly, a left $S$-local basis $\mathcal{D}$ of $M$ induces a dual basis $\{v,^{\ast}v\}_{v \in \mathcal{D}}$ where $^{\ast} v: \ _{S}{M} \rightarrow _{S}{S}$ is the morphism of left $S$-modules defined by $^{\ast}v(u)=0$ if $u \in \mathcal{D} \setminus \{v\}$ and $^{\ast}v(v)=e_{i}$ if $v=e_{i}ve_{j}$. \\

Let $L$ be a $Z$-local basis of $S$ and let $T$ be a $Z$-local basis of $M_{0}$. \\

Throughout this paper we will use the following notation $T_{k}= T \cap Me_{k}$ and $_{k}T= T \cap e_{k}M$. \\

We will also assume that for every integer $i$ in $[1,n]$ and for each $F$-basis $L(i)$ of $D_{i}$, the following equalities hold for each $f,w,z \in L(i)$:

\begin{equation} \label{2.1}
f^{\ast}(w^{-1}z)=w^{\ast}(zf^{-1})
\end{equation}
\begin{equation} \label{2.2}
f^{\ast}(zw)=(w^{-1})^{\ast}(f^{-1}z)
\end{equation}
\begin{equation}\label{2.3}
z^{\ast}(wf)=(w^{-1})^{\ast}(fz^{-1})
\end{equation}

\vspace{0.1in}

Note that \ref{2.1} readily implies (1) of \cite[p.29]{2} by taking $w=e_{i}$, $z=s$ and $f=t$. \\
Observe that in \ref{2.1} and \ref{2.2} one can replace $z \in L(i)$ by $z \in D_{i}$, because both expressions are linear in $z$. Similarly, one can replace in \ref{2.3} $f \in L(i)$ by $f \in D_{i}$.

\begin{remark}
If $L(i)$ is a semi-multiplicative basis of $D(i)$ then $L(i)$ satisfies \ref{2.1}, \ref{2.2} and \ref{2.3}. 
\end{remark}
\begin{proof}
Indeed, suppose that $f^{\ast}(w^{-1}z) \neq 0$ then $w^{-1}z=cf$ for some uniquely determined $c \in F^{\ast}$; thus $f^{\ast}(w^{-1}c)=c$. On the other hand, $w^{\ast}(zf^{-1})=w^{\ast}(z(z^{-1}wc))=w^{\ast}(wc)=c$ and the equality follows. A similar argument shows that \ref{2.2} and \ref{2.3} also hold. 
\end{proof}

\begin{remark}
Suppose that $L_{1}$ is an $F$-basis for the field extension $F_{1}/F$ and $L_{2}$ is an $F_{1}$-basis for the field extension $F_{2}/F_{1}$. If $L_{1}$ and $L_{2}$ satisfy \ref{2.1}, \ref{2.2} and \ref{2.3} then the $F$-basis $\{xy: x \in L_{1},y \in L_{2}\}$ of $F_{2}/F$ also satisfies \ref{2.1}, \ref{2.2} and \ref{2.3}.
\end{remark}

\begin{proof} Suppose that both $L_{1}$ and $L_{2}$ satisfy \ref{2.1}. Let $f=f_{1}f_{2}$, $w=w_{1}w_{2}$, $z=z_{1}z_{2}$ where $f_{1},w_{1},z_{1} \in L_{1}$ and $f_{2},w_{2},z_{2} \in L_{2}$. Then:
\begin{align*}
f^{\ast}(w^{-1}z)&=f_{1}^{\ast}f_{2}^{\ast}(w_{1}^{-1}w_{2}^{-1}z_{2}z_{1}) \\
&=f_{1}^{\ast}(w_{1}^{-1}z_{1}f_{2}^{\ast}(w_{2}^{-1}z_{2}))\\
&=f_{1}^{\ast}(w_{1}^{-1}z_{1}w_{2}^{\ast}(z_{2}f_{2}^{-1}))\\
&=w_{1}^{\ast}(z_{1}f_{1}^{-1}w_{2}^{\ast}(z_{2}f_{2}^{-1}))\\
&=(w_{1}w_{2})^{\ast}(z_{1}z_{2}(f_{1}f_{2})^{-1}) \\
&=w^{\ast}(zf^{-1})
\end{align*}
as claimed.The equalities \ref{2.2} and \ref{2.3} are established in an analogous way.
\end{proof}

\begin{remark} \label{rem3} If the basis $L(i)$ satisfies \ref{2.1}, then for each $f,z \in L(i)$ we have: \\
\begin{center}
$\displaystyle \sum_{w \in L(i)} f^{\ast}(w^{-1}z)w=zf^{-1}$
\end{center}
\end{remark}

\begin{proof}
\begin{center}
$\displaystyle \sum_{w \in L(i)} f^{\ast}(w^{-1}z)w=\displaystyle \sum_{w \in L(i)} w^{\ast}(zf^{-1})w=zf^{-1}.$
\end{center}
\end{proof}

\begin{remark} If the basis $L(i)$ satisfies \ref{2.2}, then for each $r,v \in L(i)$ we have: \\
\begin{center}
\begin{equation} \label{2.4}
\displaystyle \sum_{t \in L(i)} r^{\ast}(vt)t^{-1}=r^{-1}v
\end{equation}
\end{center}
\end{remark}

\begin{proof}
\begin{center}
$\displaystyle \sum_{t \in L(i)} r^{\ast}(vt)t^{-1}=\displaystyle \sum_{t \in L(i)} (t^{-1})^{\ast}(r^{-1}v)t^{-1}=r^{-1}v.$
\end{center}
\end{proof}

\begin{definition} Given an $S$-bimodule $N$ we define the \emph{cyclic part} of $N$ as $N_{cyc}:=\displaystyle \sum_{j=1}^{n} e_{j}Ne_{j}$.
\end{definition}

\begin{definition} A \emph{potential}  $P$ is an element of $\mathcal{F}_{S}(M)_{cyc}$. \\
\end{definition}

For each legible element $a$ of $e_{i}Me_{j}$, we let $\sigma(a)=i$ and $\tau(a)=j$. Recall that each $L(i)=L \cap e_{i}S$ is an $F$-basis of $D_{i}$. We will assume that each basis $L(i)$ satisfies that  $\operatorname{char}(F) \nmid \operatorname{card}L(i).$

\begin{definition} Let $P$ be a potential in $\mathcal{F}_{S}(M)$, then $R(P)$ is the closure of the two-sided ideal of $\mathcal{F}_{S}(M)$ generated by all the elements $X_{a^{\ast}}(P):=\displaystyle \sum_{s \in L(\sigma(a))} \delta_{(sa)^{\ast}}(P)s$ where $a \in T$. 
\end{definition}

In \cite[p.19]{2} it is shown that the Jacobian ideal (in the sense of \cite{6}) contains properly the ideal $R(P)$.

Let $k$ be an integer in $[1,n]$. Using the $S$-bimodule $M$, we define a new $S$-bimodule $\mu_{k}M=\widetilde{M}$ as:
\vspace{0.1in}
\begin{center}
$\widetilde{M}:=\bar{e_{k}}M\bar{e_{k}} \oplus Me_{k}M \oplus (e_{k}M)^{\ast} \oplus ^{\ast}(Me_{k})$
\end{center}

\vspace{0.1in}

where $\bar{e_{k}}=1-e_{k}$, $(e_{k}M)^{\ast}=\operatorname{Hom}_{S}((e_{k}M)_{S},S_{S})$ and $^{\ast}(Me_{k})=\operatorname{Hom}_{S}(_{S}(Me_{k}),_{S}S)$. One can show (see \cite[Lemma 8.7]{2}) that $\mu_{k}M$ is $Z$-freely generated by the following $Z$-subbimodule: \\
\begin{center}
$\bar{e_{k}}M_{0}\bar{e_{k}} \oplus M_{0}e_{k}Se_{k}M_{0} \oplus e_{k}(_{0}N) \oplus N_{0}e_{k}$
\end{center}

\vspace{0.1in}

where $N_{0}=\{h \in M^{\ast} | h(M_{0}) \in Z, h(tM_{0})=0, t \in L^{'} \}$, $_{0}N=\{h \in ^{\ast}M | h(M_{0}) \in Z, h(M_{0}t)=0, t \in L^{'} \}$ and $L'=L \setminus \{e_{1},\hdots,e_{n}\}$.

 \begin{definition} An algebra with potential is a pair $(\mathcal{F}_{S}(M),P)$ where $P$ is a potential in $\mathcal{F}_{S}(M)$ and $M_{cyc}=0$.
\end{definition}

Throughout this paper we will assume that $M$ is $Z$-freely generated by $M_{0}$.

\section{Decorated representations} \label{sec3}

\begin{definition} Let $(\mathcal{F}_{S}(M),P)$ be an algebra with potential. A decorated representation of $(\mathcal{F}_{S}(M),P)$ is a pair $\mathcal{N}=(N,V)$ where $N$ is a finite dimensional left $\mathcal{F}_{S}(M)$-module annihilated by $R(P)$ and $V$ is a finite dimensional left $S$-module.
\end{definition}

Equivalently, $N$ is a finite dimensional left module over the quotient algebra $\mathcal{F}_{S}(M)/R(P)$. For $u \in \mathcal{F}_{S}(M)$ we let $u_{N}=u: N \rightarrow N$ denote the multiplication operator $u(n)=un$. 

Let $(\mathcal{F}_{S}(M),P)$ and $(\mathcal{F}_{S}(M'),P')$ be algebras with potential. Let $\mathcal{N}=(N,V)$ and $\mathcal{N'}=(N',V')$ be decorated representations of $(\mathcal{F}_{S}(M),P)$ and $(\mathcal{F}_{S}(M'),P')$, respectively. A \emph{right-equivalence} between $\mathcal{N}$ and $\mathcal{N'}$ is a triple $(\varphi,\psi,\eta)$ where:

\begin{itemize}
\item $\varphi: \mathcal{F}_{S}(M) \rightarrow \mathcal{F}_{S}(M')$ is a right-equivalence between $(\mathcal{F}_{S}(M),P)$ and $(\mathcal{F}_{S}(M'),P')$. 
\item $\psi: N \rightarrow N'$ is an isomorphism of $F$-vector spaces such that $\psi \circ u_{N}=\varphi(u)_{N'} \circ \psi$ for each $u \in \mathcal{F}_{S}(M)$.
\item $\eta: V \rightarrow V'$ is an isomorphism of left $S$-modules.
\end{itemize}

\begin{remark} Suppose that $M^{\otimes n}=0$ for $n \gg 0$ and that $\mathcal{F}_{S}(M)_{cyc}=\{0\}$, then a decorated representation can be identified with a left module over the tensor algebra $T_{S}(M)$. In the case that the underlying semisimple algebra $S$ happens to be a finite direct product of copies of the base field $F$, then $T_{S}(M)$ can be identified with a path algebra, so in this case a decorated representation is a representation of a quiver in the classical sense.
\end{remark}

Let $M_{1}$, $M_{2}$ be $Z$-freely generated $S$-bimodules and let $T_{1}$ and $T_{2}$ be $Z$-free generating sets of $M_{1}$ and $M_{2}$, respectively. Let $P_{1}$ and $P_{2}$ be potentials in $\mathcal{F}_{S}(M_{1})$ and $\mathcal{F}_{S}(M_{2})$ respectively, and consider the potential $P_{1}+P_{2} \in \mathcal{F}_{S}(M_{1} \oplus M_{2})$. Let $\mathcal{N}=(N,V)$ be a decorated representation of the algebra with potential $(\mathcal{F}_{S}(M_{1} \oplus M_{2}),P_{1}+P_{2})$. We have an injective morphism of topological algebras $\mathcal{F}_{S}(M_{1}) \hookrightarrow \mathcal{F}_{S}(M_{1} \oplus M_{2})$ and thus, by restriction of scalars, $N$ is a left $\mathcal{F}_{S}(M_{1})$-module. We will denote this module by $N|_{\mathcal{F}_{S}(M_{1})}$. Let us show that $R(P_{1})$ annihilates $N|_{\mathcal{F}_{S}(M_{1})}$. Let $a \in T_{1}$, then $X_{a^{\ast}}(P_{2})=\displaystyle \sum_{s \in L(\sigma(a))} \delta_{(sa)^{\ast}}(P_{2})s=0$. Thus $X_{a^{\ast}}(P_{1}+P_{2})=X_{a^{\ast}}(P_{1}) \in R(P_{1}+P_{2})$. It follows that $\mathcal{N}|_{\mathcal{F}_{S}(M_{1})}:=(N|_{\mathcal{F}_{S}(M_{1})},V)$ is a decorated representation of the algebra with potential $(\mathcal{F}_{S}(M_{1}),P_{1})$.

\begin{prop} Let $M_{1}$ and $M_{2}$ be $Z$-freely generated $S$-bimodules and let $P,P'$ be reduced potentials in $\mathcal{F}_{S}(M_{1})$ and $W$ be a trivial potential in $\mathcal{F}_{S}(M_{2})$. Let $\mathcal{N}$ and $\mathcal{N'}$ be decorated representations of $\mathcal{F}_{S}(M_{1} \oplus M_{2})$ with respect the potentials $P+W$ and $P'+W$. If $\mathcal{N}$ is right-equivalent to $\mathcal{N'}$ then $\mathcal{N}|_{\mathcal{F}_{S}(M_{1})}$ is right-equivalent to $\mathcal{N'}|_{\mathcal{F}_{S}(M_{1})}$.
\end{prop}

\begin{proof} Let $(\phi,\psi,\eta): \mathcal{N} \rightarrow \mathcal{N'}$ be a right-equivalence of decorated representations. Then
\begin{enumerate}[(a)]
\item $\phi$ is an algebra automorphism of $\mathcal{F}_{S}(M_{1} \oplus M_{2})$ with $\phi_{|S}=id_{S}$ such that $\phi(P+W)$ is cyclically equivalent to $P'+W$.
\item $\psi: N \rightarrow N'$ is an isomorphism of $F$-vector spaces such that for $n \in N$ and $u \in \mathcal{F}_{S}(M_{1} \oplus M_{2})$ we have $\psi(un)=\phi(u)\psi(n)$.
\item $\eta: V \rightarrow V'$ is an isomorphism of left $S$-modules.
\end{enumerate}
Let $L$ be the closure of the two-sided ideal of $\mathcal{F}_{S}(M_{1} \oplus M_{2})$ generated by $M_{2}$, then as in \cite[Proposition 6.6]{2}:
\begin{enumerate}
\item $\mathcal{F}_{S}(M_{1} \oplus M_{2})=\mathcal{F}_{S}(M_{1}) \oplus L$
\item $R(P+W)=R(P) \oplus L$
\item $R(P'+W)=R(P') \oplus L$
\end{enumerate}
where $R(P)$ (respectively $R(P')$) is the closure of the two-sided ideal in $\mathcal{F}(M_{1})$ generated by all the elements of the form $X_{a^{\ast}}(P)$ (respectively $X_{a^{\ast}}(P'))$ where $a \in T_{1}$, a $Z$-free generating set of $M_{1}$.
Let $p: \mathcal{F}_{S}(M_{1} \oplus M_{2}) \twoheadrightarrow \mathcal{F}_{S}(M_{1})$ be the canonical projection induced by the decomposition $(1)$. As in \cite[Proposition 6.6]{2} there exists an algebra isomorphism
\begin{center}
$\rho=p \circ \phi_{|\mathcal{F}_{S}(M_{1})}: \mathcal{F}_{S}(M_{1}) \rightarrow \mathcal{F}_{S}(M_{1})$
\end{center}

such that $P'-\rho(P)$ is cyclically equivalent to an element of $R(\rho(P))^{2}$. By \cite[Proposition 6.5]{2} there exists an algebra automorphism $\lambda$ of $\mathcal{F}_{S}(M_{1})$ such that $\lambda \rho(P)$ is cyclically equivalent to $P'$ and $\lambda(u)-u \in R(\rho(P))$ for all $u \in \mathcal{F}_{S}(M_{1})$. By definition of decorated representation, $R(P+W)N=0$ and $R(P'+W)N'=0$. Since $L$ is contained in both $R(P+W)$ and $R(P'+W)$ then $LN=0$ and $LN'=0$. Then for $u \in \mathcal{F}_{S}(M_{1} \oplus M_{2})$ and $n \in N$
\begin{center}
$\psi(un)=\phi(u)\psi(n)$
\end{center}
Note that $\phi(u)=p\phi(u)+u'$ where $u' \in L$, then $\phi(u)\psi(n)=p\phi(u)\psi(n)=\rho(u)\psi(n)$, thus
\begin{center}
$\psi(un)=\rho(u)\psi(n)$
\end{center}
Since $P'-\rho(P)$ is cyclically equivalent to an element of $R(\rho(P))^{2}$, then there exists $z \in [\mathcal{F}_{S}(M_{1}),\mathcal{F}_{S}(M_{1})]$ such that $P'+z-\rho(P) \in R(\rho(P))^{2}$. Therefore by \cite[Proposition 6.4]{2} we obtain
\begin{center}
$R(P')=R(P'+z)=R(\rho(P))$
\end{center}
Now consider the automorphism $\lambda \rho$ of $\mathcal{F}_{S}(M_{1})$, this map has the property that $\lambda \rho(P)$ is cyclically equivalent to $P'$; also for $n \in N$ and $u \in \mathcal{F}_{S}(M_{1})$ we have
\begin{center}
$\psi(un)=\rho(u)\psi(n)$
\end{center}
and $\lambda \rho(u)=\rho(u)+w$ where $w \in R(\rho(P))=R(P')$. Therefore
\begin{center}
$\psi(un)=\lambda \rho(u)\psi(n)$
\end{center}
This proves that $(\lambda \rho, \psi, \eta)$ is a right-equivalence between $\mathcal{N}|_{\mathcal{F}_{S}(M_{1})}$ and $\mathcal{N'}|_{\mathcal{F}_{S}(M_{1})}$, as claimed.
\end{proof}

In what follows, we will use the following notation: for an $S$-bimodule $B$, define: \\

\begin{center}
$B_{\hat{k},\hat{k}}=\bar{e_{k}}B\bar{e_{k}}$
\end{center}

\vspace{0.1in}

Now consider the algebra isomorphism $\rho: \mathcal{F}_{S}(M)_{\hat{k},\hat{k}} \rightarrow \mathcal{F}_{S}((\mu_{k}M)_{\hat{k},\hat{k}})$ defined in \cite[Lemma 9.2]{2}. Let $P$ be a reduced potential in $\mathcal{F}_{S}(M)_{\hat{k},\hat{k}}$. Suppose first that:

\begin{center}
(A) $P=\displaystyle \sum_{u=1}^{N} f_{u}\gamma_{u}$
\end{center}

where $f_{u} \in F$ and $\gamma_{u}=x_{1} \hdots x_{n(u)}$ with $x_{i} \in \hat{T}$ as in \cite[Definition 26]{2}. Let $b \in T_{k}$ be fixed and let $N_{b}$ be the set of all $u \in [1,N]$ such that for some $x_{i}$, $a(x_{i})=b$. For each $u \in N_{b}$, let $\mathcal{C}(u)$ be the subset of all cyclic permutations $c$ of the set $\{1,\hdots,n(u)\}$ such that $x_{c(1)}=s_{c}b$. Then for each $c \in C(u)$ define $\gamma_{u}^{c}=x_{c(1)}x_{c(2)} \hdots x_{c(n(u))}$. Thus $\gamma_{u}^{c}=s_{c}br_{c}a_{c}z_{c}$ where $z_{c}=x_{3} \hdots x_{c(n(u))}$. Therefore \\

\begin{center}
$X_{b^{\ast}}(P)=\displaystyle \sum_{u \in N_{b}} \displaystyle \sum_{c \in \mathcal{C}(u)} f_{u}r_{c}a_{c}z_{c}s_{c}$
\end{center}

On the other hand:
\begin{align*}
X_{[bra]^{\ast}}(\rho(P))&=\displaystyle \sum_{u \in N_{b}} \displaystyle \sum_{c \in \mathcal{C}(u),r_{c}=r,a_{c}=a}f_{u}\rho(z_{c})s_{c} \\
&=\rho \left(\displaystyle \sum_{u \in N(b)} \displaystyle \sum_{c \in \mathcal{C}(u),r_{c}=r,a_{c}=a} f_{u}z_{c}s_{c} \right)
\end{align*}

Define $Y_{[bra]}(P):=\displaystyle \sum_{u \in N_{b}} \displaystyle \sum_{c \in \mathcal{C}(u),r_{c}=r,a_{c}=a}f_{u}z_{c}s_{c}$. Then \\

\begin{center}
$X_{[bra]^{\ast}}(\rho(P))=\rho(Y_{[bra]}(P))$
\end{center}

\medskip

Note that if $P$ is a potential in $\mathcal{F}_{S}(M)^{\geq n+3}$ then $Y_{[bra]}(P) \in \mathcal{F}_{S}(M)^{\geq n}$, thus if $(P_{n})_{n \geq 1}$ is a Cauchy sequence in $\mathcal{F}_{S}(M)$ then $(Y_{[bra]}(P_{n}))_{n \geq 1}$ is Cauchy as well. 
Now let $P$ be an arbitrary potential in $\mathcal{F}_{S}(M)$. We have \\

\begin{center}
$P=\displaystyle \lim_{n \to \infty} P_{n}$
\end{center}

where each $P_{n}$ is of the form given by (A). Define: \\

\begin{center}
$w=\displaystyle \lim_{n \to \infty} Y_{[bra]}(P_{n})$
\end{center}

Then
\begin{align*}
X_{[bra]^{\ast}}(\rho(P))&=\displaystyle \lim_{n \to \infty} X_{[bra]^{\ast}}(\rho(P_{n})) \\
&=\displaystyle \lim_{n \to \infty} \rho(Y_{[bra]}(P_{n})) \\
&=\rho(\displaystyle \lim_{n \to \infty} Y_{[bra]}(P_{n})) \\
&=\rho(w)
\end{align*}

Thus we let $Y_{[bra]}(P):=w$. Then $X_{[bra]^{\ast}}(\rho(P))=\rho(Y_{[bra]}(P))$. In \cite[p.60]{2} the following equalities are established for each potential $P$ in $T_{S}(M)$:

\begin{equation} \label{3.1}
\rho(b'X_{b^{\ast}}(P))=\displaystyle \sum_{r \in L(k),a \in _{k}T} [b'ra]X_{[bra]^{\ast}}(\rho(P)) 
\end{equation}

\begin{equation} \label{3.2}
\rho(X_{a^{\ast}}(P)a')=\displaystyle \sum_{b \in T_{k}, r \in L(k)} X_{[bra]^{\ast}}(\rho(P))[bra']
\end{equation}

By continuity, the above formulas remain valid for every potential $P \in \mathcal{F}_{S}(M)$. Using \ref{3.1} yields:

\begin{align*}
\rho(b'X_{b^{\ast}}(P))&=\displaystyle \sum_{r \in L(k),a \in _{k}T} \rho(b'ra) \rho(Y_{[bra]}(P)) \\
&=\rho \left( \displaystyle \sum_{r \in L(k), a \in _{k}T} b'raY_{[bra]}(P)\right)
\end{align*}

Since $\rho$ is injective, then: 

\begin{equation} \label{3.3}
b'X_{b^{\ast}}(P)=\displaystyle \sum_{r \in L(k), a \in _{k}T} b'raY_{[bra]}(P)
\end{equation}

Similarly:

\begin{equation} \label{3.4}
X_{a^{\ast}}(P)a'=\displaystyle \sum_{b \in T_{k},r \in L(k)} Y_{[bra]}(P)bra'
\end{equation}

For each $\psi \in M^{\ast}$ and for each positive integer $n$ we have an $F$-linear map $\psi_{\ast}: M^{\otimes n} \rightarrow M^{\otimes (n-1)}$ given by $\psi_{\ast}(m_{1} \otimes m_{2} \otimes \hdots \otimes m_{n})=\psi(m_{1})m_{2} \otimes \hdots \otimes m_{n}$. This map induces an $F$-linear map $\psi_{\ast}: \mathcal{F}_{S}(M) \rightarrow \mathcal{F}_{S}(M)$. Similarly, if $\eta \in ^{\ast}M$ then we obtain an $F$-linear map $_{\ast}\eta: M^{\otimes n} \rightarrow M^{\otimes(n-1)}$ given by $_{\ast}\eta(m_{1} \otimes \hdots \otimes m_{n-1} \otimes m_{n})=m_{1} \otimes \hdots \otimes m_{n-1}\eta(m_{n})$. 
Now suppose that $b' \in T_{k}$, then $b'e_{k}=b'$ and thus $e_{k}(b')^{\ast}=(b')^{\ast}$. Since $X^{P}$ is a morphism of $S$-bimodules \cite[Proposition 7.6]{2} then $e_{k}X_{b^{\ast}}(P)=X_{b^{\ast}}(P)$. Applying the map $(b')^{\ast}: \mathcal{F}_{S}(M) \rightarrow \mathcal{F}_{S}(M)$ to the left-hand side of \ref{3.3} yields \\

\begin{equation} \label{3.5}
X_{b^{\ast}}(P)=e_{k}X_{b^{\ast}}(P)=\displaystyle \sum_{r \in L(k),a \in _{k}T} e_{k}raY_{[bra]}(P)
\end{equation}

Therefore
\begin{equation} \label{3.6}
X_{b^{\ast}}(P)=\displaystyle \sum_{r \in L(k), a \in _{k}T} raY_{[bra]}(P)
\end{equation}

Let $H$ denote the set of all non-zero elements of the form $as$ where $a \in _{k}T$ and $s \in L$. Note that $H$ is a left $S$-local basis of $_{S}M$; for $x \in H$ we denote by $^{\ast}x \in ^{\ast}M$ the map given by $^{\ast}x(y)=0$ if $y \in H \setminus \{x\}$ and $x^{\ast}(x)=e_{i}$ where $e_{i}x=x$. Applying the map $^{\ast}(a')$ to \ref{3.4} we obtain: \\

\begin{equation} \label{3.7}
X_{a^{\ast}}(P)=\displaystyle \sum_{b \in T_{k}, r \in L(k)} Y_{[bra]}(P)br
\end{equation}

Let $\mathcal{N}=(N,V)$ be a decorated representation of the algebra with potential $(\mathcal{F}_{S}(M),P)$ and suppose that $k$ satisfies $(Me_{k} \otimes_{S} M)_{cyc}=\{0\}$. Define:

\begin{align*}
N_{in}=\displaystyle \bigoplus_{a \in _{k}T} D_{k} \otimes_{F} N_{\tau(a)} \\
N_{out}=\displaystyle \bigoplus_{b \in T_{k}} D_{k} \otimes_{F} N_{\sigma(b)}
\end{align*}

For each $a$ in $_{k}T$ and $r \in L(k)$ consider the projection map $\pi_{a}': N_{in} \rightarrow D_{k} \otimes_{F} N_{\tau(a)}$ and the map $\pi_{ra}': D_{k} \otimes_{F} N_{\tau(a)} \rightarrow N_{\tau(a)}$ given by $\pi_{ra}'(d \otimes n)=r^{\ast}(d)n$. Let $\xi_{a}': D_{k} \otimes_{F} N_{\tau(a)} \rightarrow N_{in}$ denote the inclusion map and define $\xi_{ra}': N_{\tau(a)} \rightarrow D_{k} \otimes_{F} N_{\tau(a)}$ as the map given by $\xi_{ra}'(n)=r \otimes n$. \\

Then for $r,r_{1} \in L(k)$ we have the following equalities

\begin{equation} \label{3.8}
\pi_{r_{1}a}' \xi_{ra}' = \delta_{r,r_{1}}id_{N_{\tau(a)}}
\end{equation}

\begin{equation} \label{3.9}
\pi_{ra}' \xi_{ra}'= id_{N_{\tau(a)}}
\end{equation}

For $a \in _{k}T$ and $r \in L(k)$ we define the following $F$-linear maps:

\begin{align*}
\pi_{ra}&=\pi_{ra}' \pi_{a}': N_{in} \rightarrow N_{\tau(a)} \\
\xi_{ra}&=\xi_{a}' \xi_{ra}': N_{\tau(a)} \rightarrow N_{in}
\end{align*}

then we have the following equalities \\
\begin{equation} \label{3.10}
\pi_{r_{1}a_{1}}\xi_{ra}=\delta_{r_{1}a_{1},ra}id_{N_{\tau(a)}}
\end{equation}

\begin{equation} \label{3.11}
\displaystyle \sum_{r \in L(k), a \in _{k}T} \xi_{ra}\pi_{ra}=id_{N_{in}}
\end{equation}

Similarly, for each $r \in L(k)$ and $b \in T_{k}$ we have the canonical projection $\pi_{b}': N_{out} \rightarrow D_{k} \otimes_{F} N_{\sigma(b)}$ and $\pi_{br}': D_{k} \otimes_{F} N_{\sigma(b)} \rightarrow N_{\sigma(b)}$ denotes the map given by $\pi_{br}'(d \otimes n)=(r^{-1})^{\ast}(d)n$. \\

We define $\xi_{br}': N_{\sigma(b)} \rightarrow D_{k} \otimes_{F} N_{\sigma(b)}$ as the map given by $\xi_{br}'(n)=r^{-1} \otimes n$ for every $n \in N_{\sigma(b)}$ and $\xi_{b}': D_{k} \otimes_{F} N_{\sigma(b)} \rightarrow N_{out}$ is the inclusion map.  \\

Then for $r,r_{1} \in L(k)$ and $b \in T_{k}$ we have the following equalities:

\begin{equation} \label{3.12}
\pi_{br_{1}}' \xi_{br}' = \delta_{r_{1},r} id_{N_{\sigma(b)}}
\end{equation}

\begin{equation} \label{3.13}
\pi_{br}' \xi_{br} ' = id_{N_{\sigma(b)}}
\end{equation}

Define the following $F$-linear maps:
\begin{align*}
\xi_{br}=\xi_{b}' \xi_{br}': N_{\sigma(b)} \rightarrow N_{out} \\
\pi_{br}=\pi_{br}' \pi_{b}': N_{out} \rightarrow N_{\sigma(b)}
\end{align*}

Then for $r,r_{1} \in L(k)$ and $b,b_{1} \in T_{k}$ we have
\begin{equation} \label{3.14}
\pi_{b_{1}r_{1}} \xi_{br}= \delta_{b_{1}r_{1},br}id_{N_{\sigma(b)}} 
\end{equation}

and
\begin{equation} \label{3.15}
\displaystyle \sum_{b \in T_{k}, r \in L(k)} \xi_{br} \pi_{br}=id_{N_{out}}
\end{equation}
We define a map of left $D_{k}$-modules $\alpha: N_{in} \rightarrow N_{k}$ as the map such that for all $a \in _{k}T, r \in L(k)$ we have: 
\begin{center}
$\alpha \xi_{ra}(n)=ran$
\end{center}

for each $n \in N_{\tau(a)}$. \\

Similarly, we define $\beta: N_{k} \rightarrow N_{out}$ as the $F$-linear map such that for all $b \in T_{k}, r \in L(k)$: 
\begin{center}
$\pi_{br} \beta(n)=brn$
\end{center}

for every $n \in N_{k}$. \\

Finally, the map $\gamma: N_{out} \rightarrow N_{in}$ is the morphism of left $D_{k}$-modules such that map $\gamma_{ra,bs}=\pi_{ra}\gamma \xi_{bs}: N_{\sigma(b)} \rightarrow N_{\tau(a)}$ where $r,s \in L(k)$, $a \in _{k}T$, $b \in T_{k}$, is given by: 

\begin{center}
$\gamma_{ra,bs}(n)=\displaystyle \sum_{w \in L(k)} r^{\ast}(s^{-1}w)Y_{[bwa]}(P)n$
\end{center}
for every $n \in N_{\sigma(b)}$. 

\begin{prop} The map $\beta$ is a morphism of left $D_{k}$-modules. 
\end{prop}
\begin{proof} By linearity, it suffices to show that if $c \in L(k)$ and $n \in N_{k}$ then $\beta(cn)=c\beta(n)$. Using \cite[Proposition 7.5]{2} we obtain:
\begin{align*}
\beta(cn)&=c\displaystyle \sum_{b \in T_{k}, r \in L(k)} c^{-1}\xi_{br}\pi_{br}\beta(cn) = c\displaystyle \sum_{b \in T_{k}, r \in L(k)} c^{-1}(r^{-1} \otimes brcn) \\
&=c\displaystyle \sum_{b \in T_{k}, r,r_{1},r_{2} \in L(k)} (r_{1}^{-1})^{\ast}(c^{-1}r^{-1})r_{1}^{-1} \otimes br_{2}^{\ast}(rc)r_{2}n \\
&=c\displaystyle \sum_{b \in T_{k}, r_{1},r_{2} \in L(k)} \left( \displaystyle \sum_{r \in L(k)} r_{2}^{\ast}(rc)(r_{1}^{-1})^{\ast}(c^{-1}r^{-1})\right)(r_{1}^{-1} \otimes br_{2}n) \\
&=c\left(\displaystyle \sum_{b \in T_{k}, r \in L(k)}r^{-1}\otimes brn\right) \\
&=c\beta(n)
\end{align*}
as claimed. 
\end{proof}

\begin{lemma} \label{lem1} We have $\alpha \gamma=0$ and $\gamma \beta=0$. 
\end{lemma}
\begin{proof} We first show that $\alpha \gamma=0$. It suffices to show that for all $r \in L(k)$, $b \in T_{k}$, $\alpha \gamma \xi_{br}=0$. Let $n \in N_{\sigma(b)}$, then by \ref{3.6} and \ref{3.11}:
\begin{align*}
\alpha \gamma \xi_{br}(n)&= \alpha id_{N_{in}} \gamma \xi_{br}(n) = \displaystyle \sum_{s \in L(k), a \in _{k}T} \alpha \xi_{sa} \pi_{sa} \gamma \xi_{br}(n)\\
&=\displaystyle \sum_{s \in L(k), a \in _{k}T} \alpha \xi_{sa} \gamma_{sa,br}(n)\\
&=\displaystyle \sum_{s,w \in L(k), a \in _{k}T} \alpha \xi_{sa} s^{\ast}(r^{-1}w)Y_{[bwa]}(P)(n)\\
&=\displaystyle \sum_{s,w \in L(k), a \in _{k}T} sas^{\ast}(r^{-1}w)Y_{[bwa]}(P)(n)\\
&=\displaystyle \sum_{w \in L(k), a \in _{k}T} r^{-1}waY_{[bwa]}(P)n \\
&=r^{-1}X_{b^{\ast}}(P)n\\
&=0
\end{align*}

We now show that $\gamma\beta=0$. It suffices to show that for all $r \in L(k)$, $a \in _{k}T$ we have $\pi_{ra} \gamma \beta=0$. Let $n \in N_{k}$, then by \ref{3.7} and \ref{3.15}:
\begin{align*}
\pi_{ra}\gamma \beta(n)&=\pi_{ra} \gamma id_{N_{out}} \beta = \displaystyle \sum_{b \in T_{k},s \in L(k)} \pi_{ra} \gamma \xi_{bs} \pi_{bs} \beta(n) \\
&=\displaystyle \sum_{b \in T_{k}, s \in L(k)} \gamma_{ra,bs} \pi_{bs} \beta(n) \\
&\displaystyle \sum_{b \in T_{k},s,w \in L(k)} s^{\ast}(wr^{-1})Y_{[bwa]}(P)(bsn) \\
&=\displaystyle \sum_{b \in T_{k}, w \in L(k)} Y_{[bwa]}(P)bwr^{-1}n \\
&=X_{a^{\ast}}(P)r^{-1}n \\
&=0
\end{align*}
\end{proof}

\begin{lemma} \label{lem2} For each $m \in N_{in}$, $a \in _{k}T$ we have $\pi_{e_{k}a}(r^{-1}m)=\pi_{ra}(m)$.
\end{lemma}

\begin{proof} First, for any $a_{1} \in _{k}T$ and $n \in N_{\tau(a_{1})}$ we have

\begin{center}
$r^{-1}\xi_{sa_{1}}'(n)=r^{-1}(s \otimes n)=r^{-1}s \otimes n = \displaystyle \sum_{u \in L(k)} u \otimes u^{\ast}(r^{-1}s)n = \displaystyle \sum_{u \in L(k)} \xi_{ua_{1}}'(u^{\ast}(r^{-1}s)n)$
\end{center}

Then
\begin{center}
$r^{-1}\xi_{sa_{1}}(n)=r^{-1}\xi_{a_{1}}' \xi_{sa_{1}}'(n) = \displaystyle \sum_{u \in L(k)} \xi_{a_{1}}' \xi_{ua_{1}}'(u^{\ast}(r^{-1}s)n)=\displaystyle \sum_{u \in L(k)} \xi_{ua_{1}}(u^{\ast}(r^{-1}s)n)$
\end{center}

Now let $m \in N_{in}$, then using \ref{3.10} and \ref{3.11} we obtain \\
\begin{align*}
\pi_{e_{k}a}(r^{-1}m)&=\displaystyle \sum_{s \in L(k), a_{1} \in _{k}T} \pi_{e_{k}a}r^{-1}\xi_{sa_{1}}\pi_{sa_{1}}(m)\\
&=\displaystyle \sum_{s,u \in L(k), a_{1} \in _{k}T} \pi_{e_{k}a} \xi_{ua_{1}}\left(u^{\ast}(r^{-1}s)\pi_{sa_{1}}(m)\right)\\
&=\displaystyle \sum_{s \in L(k)} e_{k}^{\ast}(r^{-1}s)\pi_{sa}(m)\\
&=\pi_{ra}(m)
\end{align*}

and the lemma follows.
\end{proof}

\section{Premutation of a decorated representation} \label{sec4}

Consider now the algebra with potential $(\mathcal{F}_{S}(\widetilde{M}),\widetilde{P})$. Recall from \cite[Definition 37]{2} that: 

\begin{align*}
\widetilde{M}&:=\bar{e_{k}}M\bar{e_{k}} \oplus Me_{k}M \oplus (e_{k}M)^{\ast} \oplus ^{\ast}(Me_{k})\\
\widetilde{P}&:=[P]+\displaystyle \sum_{sa \in _{k}\hat{T},bt \in \tilde{T}_{k}}[btsa]((sa)^{\ast})(^{\ast}(bt))
\end{align*}

To a decorated representation $\mathcal{N}=(N,V)$ of the algebra with potential $(\mathcal{F}_{S}(M),P)$ we will associate a decorated representation $\widetilde{\mu_{k}}(\mathcal{N})=(\overline{N},\overline{V})$ of $(\mathcal{F}_{S}(\widetilde{M}),\widetilde{P})$ as follows. First set: \\

\begin{center}
$\overline{N}_{i}=N_{i}$, $\overline{V_{i}}=V_{i}$ if $i \neq k$
\end{center}

Define $\overline{N}_{k}$ and $\overline{V}_{k}$ as follows:

\begin{align*}
\overline{N}_{k}&=\frac{ker(\gamma)}{im(\beta)} \oplus im(\gamma) \oplus \frac{ker(\alpha)}{im(\gamma)} \oplus V_{k} \\
\overline{V}_{k}&=\frac{ker(\beta)}{ker(\beta) \cap im(\alpha)}
\end{align*}

Let
\begin{equation} \label{4.1}
\begin{split}
&J_{1}:  \frac{ker(\gamma)}{im(\beta)} \rightarrow \overline{N}_{k} \\
&J_{2}:  im(\gamma) \rightarrow \overline{N}_{k} \\
&J_{3}:  \frac{ker(\alpha)}{im(\gamma)} \rightarrow \overline{N}_{k} \\
&J_{4}:  V_{k} \rightarrow \overline{N}_{k}
\end{split}
\end{equation}

be the corresponding inclusions and let
\begin{equation} \label{4.2}
\begin{split}
&\Pi_{1}: \overline{N}_{k}  \rightarrow \frac{ker(\gamma)}{im(\beta)}\\
&\Pi_{2}: \overline{N}_{k} \rightarrow im(\gamma)\\
&\Pi_{3}: \overline{N}_{k} \rightarrow  \frac{ker(\alpha)}{im(\gamma)}\\
&\Pi_{4}: \overline{N}_{k} \rightarrow V_{k}
\end{split}
\end{equation}

denote the canonical projections. \\

\textbf{Remark}. Suppose that $M$ is $Z$-freely generated by $M_{0}$ and let $X$ be a finite dimensional left $S$-module. To induce a structure of a $T_{S}(M)$-left module on $X$ it suffices to give a map of $S$-left modules $M \otimes_{S} X \rightarrow X$. Let $i \neq j$ be integers in $[1,n]$. Then
\begin{align*}
\operatorname{Hom}_{D_{i}}(e_{i}Me_{j} \otimes_{S} X,X)& \cong \operatorname{Hom}_{D_{i}}((D_{i} \otimes_{F} e_{i}M_{0}e_{j} \otimes_{F} D_{j}) \otimes_{D_{j}} e_{j}X,e_{i}X) \\
&\cong \operatorname{Hom}_{D_{i}}(D_{i} \otimes_{F} e_{i}M_{0}e_{j} \otimes_{F} (D_{j} \otimes_{D_{j}} e_{j}X), e_{i}X) \\
&\cong \operatorname{Hom}_{D_{i}}(D_{i} \otimes_{F} e_{i}M_{0}e_{j} \otimes_{F} e_{j}X,e_{i}X) \\
&\cong \operatorname{Hom}_{F}(e_{i}M_{0}e_{j} \otimes_{F} e_{j}X, e_{i}X)
\end{align*}

Hence $\operatorname{Hom}_{S}(_{S}(M \otimes_{S} X),_{S}X) \cong \displaystyle \bigoplus_{i,j}  \operatorname{Hom}_{F}(e_{i}M_{0}e_{j} \otimes_{F} e_{j}X, e_{i}X)$ as $F$-vector spaces. Therefore, a map of left $S$-modules $M \otimes_{S} X \rightarrow X$ is determined by a collection of $F$-linear maps $\theta_{i,j}: e_{i}M_{0}e_{j} \otimes_{F} e_{j}X \rightarrow e_{i}X$. It follows that each element $c \in e_{\sigma(c)}M_{0}e_{\tau(c)}$ gives rise to a multiplication operator $c_{X}: X_{\tau(c)} \rightarrow X_{\sigma(c)}$ given by $c_{X}(x):=\theta_{\sigma(c),\tau(c)}(c \otimes x)$.  \\

Recall from \cite[Lemma 8.7]{2} that $\widetilde{M}$ is $Z$-freely generated by the following $Z$-subbimodule: \\

\begin{center}
$(\widetilde{M})_{0}:=\bar{e_{k}}M_{0}\bar{e_{k}} \oplus M_{0}e_{k}Se_{k}M_{0} \oplus e_{k}(_{0}N) \oplus N_{0}e_{k}$
\end{center}

\vspace{0.2in}

To give $\overline{N}$ a structure of a left $T_{S}(\widetilde{M})$-module on $\overline{N}$ we will proceed by cases, by giving the action of each summand of $(\widetilde{M})_{0}$ in $\overline{N}$. \\

$\bullet$ Suppose first that $i,j \neq k$. Then 

\begin{center}
$e_{i}(\widetilde{M})_{0}e_{j}=e_{i}M_{0}e_{j} \oplus e_{i}M_{0}e_{k}Se_{k}M_{0}e_{j}$
\end{center}

\vspace{0.1in}

Assume that $c \in e_{i}M_{0}e_{j}$. By assumption $i,j \neq k$ and thus both $\sigma(c)$ and $\tau(c)$ are not equal to $k$. Then $\overline{N}_{\tau(c)}=N_{\tau(c)}$ and $\overline{N}_{\sigma(c)}=N_{\sigma(c)}$. Therefore we set $c_{\overline{N}}=c_{N}$.  
Assume now that $c$ is an element of the $Z$-local basis of $e_{i}M_{0}e_{k}Se_{k}M_{0}e_{j}$ then $c=\rho(bra)$ for some $b \in T \cap e_{i}Me_{k}$, $r \in L(k)$ and $a \in T \cap e_{k}Me_{j}$. In this case we set $\rho(bra)_{\overline{N}}:=(bra)_{N}$. \\

Recall that $\{^{\ast}b: b \in T\}$ is a $Z$-local basis of $_{0}N$ and $\{a^{\ast}: a \in T\}$ is a $Z$-local basis of $N_{0}$. Then $\{^{\ast}b: b \in T_{k}\}$ is a $Z$-free generating set of $e_{k}(^{\ast}M)$ and $\{a^{\ast}: a \in _{k}T\}$ is a $Z$-free generating set of $M^{\ast}e_{k}$. Suppose that $b=e_{\sigma(b)}be_{k}$, then $\tau(^{\ast} b)=\sigma(b)$. Therefore
\begin{align*}
\overline{N}_{in}&=\displaystyle \bigoplus_{b \in T_{k}} D_{k} \otimes_{F} N_{\tau(^{\ast}b)} \\
&=\displaystyle \bigoplus_{b \in T_{k}} D_{k} \otimes_{F} N_{\sigma(b)} \\
&=N_{out}
\end{align*}

whence $\overline{N}_{in}=N_{out}$. A similar argument shows that $\overline{N}_{out}=N_{in}$. \\

We have the inclusion maps
\begin{align*}
& j: ker(\gamma) \rightarrow N_{out} \\
& i: im(\gamma) \rightarrow N_{in} \\
& j': ker(\alpha) \rightarrow N_{in}
\end{align*}

and the canonical projections
\begin{align*}
& \pi_{1}: ker(\gamma) \rightarrow \frac{ker(\gamma)}{im(\beta)} \\
& \pi_{2}: ker(\alpha) \rightarrow \frac{ker(\alpha)}{im(\gamma)}
\end{align*}

\vspace{0.1in}

As in \cite{5} we introduce the following splitting data: \\
\begin{enumerate}[(a)]
\item Choose a $D_{k}$-linear map $p: N_{out} \rightarrow ker(\gamma)$ such that $pj=id_{ker(\gamma)}$.
\item Choose a $D_{k}$-linear map $\sigma_{2}: ker(\alpha)/im(\gamma) \rightarrow ker(\alpha)$ such that $\pi_{2} \sigma_{2}=id_{ker(\alpha)/im(\gamma)}$.
\end{enumerate}

$\bullet$ Suppose now that $i \neq k$ and that $j=k$. Then $e_{i}(\widetilde{M})_{0}e_{k}=e_{i}(N_{0})e_{k}$. Let $a \in _{k}T$, then $\tau(a^{\ast})=k$ and $\sigma(a^{\ast})=\tau(a)$.  We define an $F$-linear map: \\

\begin{center}
$\overline{N}(a^{\ast}): \overline{N}_{k} \rightarrow N_{\tau(a)}$
\end{center}

as follows
\begin{equation} \label{4.3}
\begin{split}
\overline{N}(a^{\ast})J_{1}&=0 \\
\overline{N}(a^{\ast})J_{2}&=c_{k}^{-1}\pi_{e_{k}a}i \\
\overline{N}(a^{\ast})J_{3}&=c_{k}^{-1}\pi_{e_{k}a}j'\sigma_{2} \\
\overline{N}(a^{\ast})J_{4}&=0
\end{split}
\end{equation}

where $c_{k}=[D_{k}:F]$. \\

$\bullet$ In what follows, we let $\gamma=i\gamma'$ where $\gamma': N_{out} \twoheadrightarrow im(\gamma)$. Suppose now that $i=k$ and that $j \neq k$. Then \\

\begin{center}
$e_{i}(\widetilde{M})_{0}e_{j}=e_{k}(_{0}N)e_{j}$
\end{center}

\vspace{0.1in}

Since $j \neq k$ then $\overline{N}_{j}=N_{j}=e_{j}N$. For every $b \in T_{k}$, we define an $F$-linear map: \\

\begin{center}
$\overline{N}(^{\ast}b): N_{\sigma(b)} \rightarrow \overline{N}_{k}$
\end{center}

as follows
\begin{equation} \label{4.4}
\begin{split}
\Pi_{1} \overline{N}(^{\ast}b) & = -\pi_{1}p\xi_{be_{k}} \\
\Pi_{2} \overline{N}(^{\ast}b) &= - \gamma' \xi_{be_{k}} \\
\Pi_{3} \overline{N}(^{\ast}b) &= 0 \\
\Pi_{4} \overline{N}(^{\ast}b) &=0
\end{split}
\end{equation}

The previous construction makes $\overline{N}$ a left $T_{S}(\widetilde{M})$-module. To see that $\overline{N}$ is in fact a module over the completed algebra $\mathcal{F}_{S}(\widetilde{M})$ it suffices to note that the $\mathcal{F}_{S}(M)$-module $N$ is nilpotent \cite[p. 39]{5} and thus $\overline{N}$ is annihilated by $\langle \widetilde{M} \rangle^{n}$ for large enough $n$. 

\begin{lemma} \label{lem3}  Let $\rho: \mathcal{F}_{S}(M)_{\hat{k},\hat{k}} \rightarrow \mathcal{F}_{S}((\mu_{k}M)_{\hat{k},\hat{k}})$ be the algebra isomorphism introduced on page $5$ and let $u \in \mathcal{F}_{S}(M)_{\hat{k},\hat{k}}$. Then $\rho(u)_{\overline{N}}=u_{N}$.
\end{lemma}

\begin{proof}

First note that $\mathcal{F}_{S}(M)=S \oplus M \oplus \mathcal{F}_{S}(M)^{\geq 2}$. Let $u \in \mathcal{F}_{S}(M)_{\hat{k},\hat{k}}$, then $u=s+m+x$ where $s \in \bar{e_{k}}S$, $m \in M_{\hat{k},\hat{k}}$ and $x \in \left(\mathcal{F}_{S}(M)^{\geq 2}\right)_{\hat{k},\hat{k}}$. Then
\begin{center}
$\rho(u)=s+m+\rho(x)$
\end{center}

By continuity and linearity of $\rho$, it suffices to treat the case when $x$ is of the form $s(x_{1})x_{1}s(x_{2}) \hdots s(x_{l})x_{l}$, where $s(x_{i}) \in L(\sigma(x_{i}))$ and $x_{i} \in T$. We'll use induction on $l$. Suppose first that $x=s(x_{1})x_{1}s(x_{2})x_{2}$ and we may assume that $x_{1}s(x_{2})x_{2} \in M_{0}e_{k}Se_{k}M_{0}$, then \\

\begin{center}
$\rho(x)=s(x_{1})\rho(x_{1}s(x_{2})x_{2})$
\end{center}

Therefore
\begin{center}
$\rho(x)n = s(x_{1}) \rho(x_{1}s(x_{2})x_{2})n$
\end{center}

\smallskip

Since $[b_{q}ra_{s}]_{\overline{N}}=(b_{q}ra_{s})_{N}$ then $\rho(b_{q}ra_{s})n=b_{q}ra_{s}n$. It follows that
\begin{align*}
\rho(x)n &= s(x_{1}) \rho(x_{1}s(x_{2})x_{2})n \\
&=s(x_{1}) x_{1}s(x_{2})x_{2}n \\
&=xn
\end{align*}

Suppose now that the claim holds for the length of $x$ less than $n$. We have:
\begin{center}
$x=s(x_{1})x_{1}s(x_{2})x_{2}\hdots s(x_{l-2})x_{l-2} s(x_{l-1})x_{l-1}s(x_{l})x_{l}$
\end{center}

Using the fact that $\rho$ is an algebra morphism together with the base case $l=2$ we obtain:
\begin{align*}
\rho(x)n &=\rho(s(x_{1})x_{1} \hdots s(x_{l-2})x_{l-2}) \rho(s(x_{l-1})x_{l-1}s(x_{l})x_{l})n \\
&=\rho(s(x_{1})x_{1} \hdots(x_{l-2})x_{l-2}) s(x_{l-1})\rho(x_{l-1}s(x_{l})x_{l})n \\
&=\rho(s(x_{1})x_{1} \hdots s(x_{l-2})x_{l-2})s(x_{l-1})x_{l-1}s(x_{l})x_{l}n \\
&=\rho(s(x_{1})x_{1} \hdots s(x_{l-2})x_{l-2})n' 
\end{align*}

where $n':= s(x_{l-1})x_{l-1}s(x_{l})x_{l}n$. Since $s(x_{1})x_{1}\hdots s(x_{l-2})x_{l-2}$ has length less than $l$, then:

\begin{center}
$\rho(s(x_{1})x_{1} \hdots s(x_{l-2})x_{l-2})n'=s(x_{1})x_{1} \hdots s(x_{l-2})x_{l-2}n'$
\end{center}

Therefore
\begin{align*}
\rho(x)n&=s(x_{1})x_{1} \hdots s(x_{l-2})x_{l-2}n' \\
&=s(x_{1})x_{1}\hdots s(x_{l-2})x_{l-2}s(x_{l-1})x_{l-1}s(x_{l})x_{l}n \\
&=xn
\end{align*}

it follows that $\rho(x)n=xn$ completing the proof.
\end{proof}

\begin{prop} The pair $\widetilde{\mu_{k}}(\mathcal{N})=(\overline{N},\overline{V})$ is a decorated representation of $(\mathcal{F}_{S}(\widetilde{M}),\widetilde{P})$. 
\end{prop}

\begin{proof} We have to verify that $\overline{N}$ is annihilated by $R(\widetilde{P})$. It suffices to check that $(X_{c^{\ast}}(\widetilde{P}))_{\overline{N}}=0$ for each element $c$ of the $Z$-local basis of $(\widetilde{M})_{0}$. We proceed by cases.  \\

$\bullet$ Suppose first that $c \in T \cap \bar{e_{k}}M_{0}\bar{e_{k}}$ and let $n \in N$. Then by Lemma \ref{lem3}
\begin{align*}
\biggl(X_{c^{\ast}}(\widetilde{P})_{\overline{N}}\biggr)(n)&=X_{c^{\ast}}(\widetilde{P})n \\
&=X_{c^{\ast}}(\rho(P))n \\
&=\rho(X_{c^{\ast}}(P))n \\
&=X_{c^{\ast}}(P)n
\end{align*}

Since $\mathcal{N}=(N,V)$ is a decorated representation of $(\mathcal{F}_{S}(M),P)$ then $X_{c^{\ast}}(P)n=0$. \\

$\bullet$ Suppose now that $c=\rho(bra)$ where $b \in T_{k}$, $r \in L(k)$ and $a \in _{k}T$. By \cite[p.58]{2} we have the following equality: \\

\begin{center}
$X_{[bra]^{\ast}}(\widetilde{P})=X_{[bra]^{\ast}}(\rho(P))+c_{k}a^{\ast}r^{-1}(^{\ast}b)$
\end{center}

where $c_{k}=[D_{k}:F]$. \\

We now compute the image of the operator $(c_{k}a^{\ast}r^{-1}(^{\ast}b))_{\overline{N}}$. Let $n \in N_{\sigma(b)}$, then remembering \ref{4.3}, \ref{4.4} and Lemma \ref{lem2} we obtain
\begin{align*}
c_{k}\overline{N}(a^{\ast})r^{-1}\overline{N}(^{\ast}b)(n)&=c_{k}\overline{N}(a^{\ast})r^{-1}\left(-\pi_{1}p\xi_{be_{k}}(n),-\gamma'\xi_{be_{k}}(n),0,0\right)\\
&=c_{k}\overline{N}(a^{\ast})\left(-r^{-1}\pi_{1}p\xi_{be_{k}}(n),-r^{-1}\gamma' \xi_{be_{k}}(n),0,0\right) \\
&=c_{k}c_{k}^{-1} \pi_{e_{k}a}i \left (-r^{-1}\gamma' \xi_{be_{k}}(n)\right) \\
&=-\pi_{e_{k}a}\left(r^{-1}\gamma' \xi_{be_{k}}(n)\right) \\
&=-\pi_{ra}\left(\gamma' \xi_{be_{k}}(n)\right) \\
&=-\gamma_{ra,be_{k}}(n) \\
&=-\displaystyle \sum_{w \in L(k)} r^{\ast}(e_{k}^{-1}w)Y_{[bwa]}(P)n \\
&=-\displaystyle \sum_{w \in L(k)} r^{\ast}(w)Y_{[bwa]}(P)n \\
&=-Y_{[bra]}(P)n
\end{align*}
and by Lemma \ref{lem3}
\begin{align*}
\bigl(X_{[bra]^{\ast}}(\rho(P)\bigr)_{\overline{N}}(n)&=X_{[bra]^{\ast}}(\rho(P))n \\
&=\rho(Y_{[bra]}(P))n \\
&=Y_{[bra]}(P)n
\end{align*}
Combining the above: $(X_{[bra]}(\widetilde{P}))_{\overline{N}}=\bigl(X_{[bra]^{\ast}}(\rho(P))\bigr)_{\overline{N}}+(c_{k}a^{\ast}r^{-1}(^{\ast}b))_{\overline{N}}=0$, as desired. It remains to show that $R(\widetilde{P}) \cdot \overline{N}=\{0\}$ for the remaining elements of the $Z$-local basis of $(\widetilde{M})_{0}$. 
We now show that $(X_{a^{\ast}}(\widetilde{P}))_{\overline{N}}=0$ for each $a \in _{k}T$. Using the result of page \cite[p.58]{2} we have that
\begin{align*}
(X_{a^{\ast}}(\widetilde{P}))_{\overline{N}}&=\biggl(c_{k}\displaystyle \sum_{b \in T_{k}, r \in L(k)} r^{-1}(^{\ast}b)\rho(bra)\biggr)_{\overline{N}} \\
&=c_{k} \displaystyle \sum_{b \in T_{k},r \in L(k)} (r^{-1}(^{\ast}b)\rho(bra))_{\overline{N}}
\end{align*}
Let $n \in N_{\tau(a)}$. Then remembering \ref{4.4} and \ref{3.15}, we get the following equalities 
\begin{align*}
X_{a^{\ast}}(\widetilde{P})(n)&=c_{k}\displaystyle \sum_{b \in T_{k},r \in L(k)} r^{-1}(^{\ast}b)(bran) \\
&=c_{k} \displaystyle \sum_{b \in T_{k}, r \in L(k)} r^{-1} \left( -\pi_{1}p\xi_{be_{k}}(bran), -\gamma' \xi_{be_{k}}(bran),0,0\right) \\
&=-c_{k} \left ( \displaystyle \sum_{b \in T_{k}, r \in L(k)} r^{-1} \pi_{1}p \xi_{be_{k}}(bran), \displaystyle \sum_{b \in T_{k},r \in L(k)} r^{-1} \gamma' \xi_{be_{k}}(bran),0,0\right) \\
&=-c_{k}\left ( \displaystyle \sum_{b \in T_{k}, r \in L(k)} \pi_{1}p r^{-1} \xi_{be_{k}}(bran), \displaystyle \sum_{b \in T_{k}, r \in L(k)} \gamma' r^{-1} \xi_{be_{k}}(bran),0,0\right) \\
&=-c_{k} \left( \displaystyle \sum_{b \in T_{k}, r \in L(k)} \pi_{1} p \xi_{br} \pi_{br} \beta(an), \displaystyle \sum_{b \in T_{k}, r \in L(k)} \gamma' \xi_{br} \pi_{br} \beta(an),0,0\right) \\
&=-c_{k} \left( \pi_{1}\beta(an),\gamma' \beta(an),0,0\right) \\
&=\left(0,0,0,0\right)
\end{align*}

by Lemma \ref{lem1}. This proves that $(X_{a^{\ast}}(\widetilde{P}))_{\overline{N}}=0$ for each $a \in _{k}T$. Finally, let us show that $(X_{^{\ast}b}(\widetilde{P}))_{\overline{N}}=0$ for each $b \in T_{k}$. Let us recall the following formula from \cite[p.58]{2}: 
\begin{center}
$X_{^{\ast}(b)}(\widetilde{P})=c_{k}\displaystyle \sum_{a \in _{k}T,r \in L(k)} \rho(bra)a^{\ast}r^{-1}$
\end{center}
Now let $n \in N_{k}$. Then remembering \ref{4.3} and using Lemma \ref{lem2} we get \\
\begin{align*}
X_{^{\ast}b}(\widetilde{P})&=c_{k}\displaystyle \sum_{a \in _{k}T, r \in L(k)} \rho(bra) \overline{N}(a^{\ast})(r^{-1}n)\\
&=c_{k}\displaystyle \sum_{a \in _{k}T, r \in L(k)} \rho(bra) \overline{N}(a^{\ast})\left( \displaystyle \sum_{l=1}^{4} J_{l} \Pi_{l}(r^{-1}n)\right) \\
&=c_{k}\displaystyle \sum_{a \in _{k}T, r \in L(k)} \left( \rho(bra)c_{k}^{-1}\pi_{e_{k}a}i \Pi_{2}(r^{-1}n)+\rho(bra)c_{k}^{-1}\pi_{e_{k}a}j' \sigma_{2} \Pi_{3}(r^{-1}n) \right) \\
&=\displaystyle \sum_{a \in _{k}T, r \in L(k)} \rho(bra)\pi_{e_{k}a}i \Pi_{2}(r^{-1}n) + \displaystyle \sum_{a \in T_{k}, r \in L(k)} \rho(bra) \pi_{e_{k}a}j' \sigma_{2} \Pi_{3}(r^{-1}n)  \\
&=b \displaystyle \sum_{a \in _{k}T, r \in L(k)} ra \pi_{e_{k}a}\left(r^{-1} \Pi_{2}(n)\right)  + b \displaystyle \sum_{a \in _{k}T, r \in L(k)} ra \pi_{e_{k}a} \left(r^{-1} \sigma_{2} \Pi_{3}(n)\right)\\
&=b \displaystyle \sum_{a \in _{k}T, r \in L(k)} ra \pi_{ra}(\Pi_{2}n)+ b \displaystyle \sum_{a \in _{k}T, r \in L(k)} ra \pi_{ra}\left(\sigma_{2}\Pi_{3}(n)\right)\\
&=b \displaystyle \sum_{a \in _{k}T, r \in L(k)} \alpha \xi_{ra} \pi_{ra}(\Pi_{2}n) + b \displaystyle \sum_{a \in _{k}T, r \in L(k)} \alpha \xi_{ra} \pi_{ra} \left(\sigma_{2}\Pi_{3}(n)\right) \\
&=b \alpha(\Pi_{2}n)+b \alpha \left(\sigma_{2} \Pi_{3}(n)\right) \\
&=0
\end{align*}
by Lemma \ref{lem1}. This completes the proof that $\overline{N}$ is annihilated by $R(\widetilde{P})$.
\end{proof}

\begin{definition}
We will refer to $\widetilde{\mu_{k}}(\mathcal{N})=(\overline{N},\overline{V})$ as the premutated decorated representation.
\end{definition}

As in \cite[Proposition 10.9]{6} we now show that the isoclass of the premutated decorated representation does not depend on the choice of the splitting data.

\begin{prop} The isoclass of the decorated representation $\widetilde{\mu_{k}}(\mathcal{N})=(\overline{N},\overline{V})$ does not depend on the choice of the splitting data.
\end{prop}

\begin{proof}

Suppose that we fix $p: N_{out} \rightarrow ker(\gamma)$ such that $pj =id_{ker(\gamma)}$ where $j: ker(\gamma) \rightarrow N_{out}$ is the inclusion map. Let $p': N_{out} \rightarrow ker(\gamma)$ be another map satisfying $p'j=id_{ker(\gamma)}$. 
Then the restriction of the map $p'-p$ to the subspace $ker(\gamma)$ is the zero map. Since $\gamma: N_{out} \rightarrow N_{in}$ then $N_{out}/ker(\gamma) \cong im(\gamma)$. \\

Consider the following sequence of maps:

\vspace{0.2in}

\begin{center}
$ker(\gamma) \stackrel{j}{\longrightarrow} N_{out} \stackrel{\gamma}{\longrightarrow}  N_{out}/ker(\gamma) \cong im(\gamma) $
\end{center}

\vspace{0.2in}

By the universal property of the cokernel of $j$, there exists a unique linear map $\xi: im(\gamma) \rightarrow ker(\gamma)$ making the following diagram commute
\begin{equation*}
\xymatrix{%
ker(\gamma)\ar[r]^{j} &N_{out} \ar[r]^{\gamma'} \ar[d]^{p'-p}  &im(\gamma) \ar@{-->}[ld]^\xi \\
&ker(\gamma)}
\end{equation*}

It follows that $p'=p+\xi \gamma'$ for some linear map $\xi: im(\gamma) \rightarrow ker(\gamma)$. \\

Now suppose that we fix a map $\sigma_{2}: ker(\alpha)/im(\gamma) \rightarrow ker(\alpha)$ such that $\pi_{2}\sigma_{2}=id_{ker(\alpha)/im(\gamma)}$. Let $\sigma'_{2}: ker(\alpha)/im(\gamma) \rightarrow ker(\alpha)$ be another map satisfying $\pi_{2}\sigma'_{2}=id_{ker(\alpha)/im(\gamma)}$. By the universal property of the kernel of $\pi_{2}$, there exists a unique linear map $\eta: ker(\alpha)/im(\gamma) \rightarrow im(\gamma)$ making the following diagram commute
\begin{center}
\[\xymatrixcolsep{5pc}\xymatrix{%
ker(\alpha)/im(\gamma) \ar@{-->}[d]^\eta \ar[rd]^{\sigma'_{2}-\sigma_{2}} & \\ 
im(\gamma) \ar[r] & ker(\alpha) \ar[r]^{\pi_{2}} & ker(\alpha)/im(\gamma)}
\]
\end{center}

\vspace{0.2in}

Thus $\sigma'_{2}=\sigma_{2} + \eta$ for some linear map $\eta: ker(\alpha)/im(\gamma) \rightarrow im(\gamma)$. \\

Let $\overline{N'}(a^{\ast})$ be the map in \ref{4.3} with $\sigma_{2}$ replaced by $\sigma_{2}'$. Similarly, let $\overline{N'}(^{\ast}b)$ be the map in $\ref{4.4}$ with $p$ replaced by $p'$.\\

As in \cite[Proposition 10.9]{6} we now construct a linear automorphism $\psi: \overline{N}_{k} \rightarrow \overline{N'}_{k}$ such that $\overline{N}(a^{\ast})=\overline{N'}(a^{\ast})\psi$ and $\psi \overline{N}(^{\ast}b)=\overline{N'}(^{\ast}b)$. Since  $\overline{N}_{k}=\frac{ker(\gamma)}{im(\beta)} \oplus im(\gamma) \oplus \frac{ker(\alpha)}{im(\gamma)} \oplus V_{k}$, then we may realize $\psi$ as a matrix of order $4$. Define $\psi$ as
\begin{center}
$\psi=
\begin{pmatrix}
I & \pi_{1}\xi & 0 &0 \\
0 & I & -\eta & 0 \\
0 & 0 & I & 0 \\
0 & 0 & 0 & I
\end{pmatrix}$
\end{center}
where $I$ is the identity transformation. Note that $\psi$ is invertible. We have
\begin{center}
\begin{align*}
\overline{N'}(a^{\ast})\psi&=\begin{pmatrix} 0 & c_{k}^{-1}\pi_{e_{k}a}i & c_{k}^{-1}\pi_{e_{k}a}j' \sigma_{2}' & 0 \end{pmatrix} \begin{pmatrix}
I & \pi_{1} \xi & 0 &0 \\
0 & I & -\eta & 0 \\
0 & 0 & I & 0 \\
0 & 0 & 0 & I
\end{pmatrix} \\    
&=\begin{pmatrix} 0 & c_{k}^{-1}\pi_{e_{k}a}i & -c_{k}^{-1}\pi_{e_{k}a}i\eta+c_{k}^{-1}\pi_{e_{k}a}j'\sigma_{2}' & 0 \end{pmatrix} \\
&=\begin{pmatrix} 0 & c_{k}^{-1}\pi_{e_{k}a}i & c_{k}^{-1}\pi_{e_{k}a}(-i\eta+j'\sigma_{2}') & 0 \end{pmatrix} \\
&=\begin{pmatrix} 0 & c_{k}^{-1}\pi_{e_{k}a}i & c_{k}^{-1}\pi_{e_{k}a}(-i\eta+j'\sigma_{2}+j'\eta) & 0 \end{pmatrix} \\
&=\begin{pmatrix} 0 & c_{k}^{-1}\pi_{e_{k}a}i & c_{k}^{-1}\pi_{e_{k}a}j'\sigma_{2} & 0 \end{pmatrix} \\
&=\overline{N}(a^{\ast})
\end{align*}
\end{center}

On the other hand:
\begin{align*}
\psi \overline{N}(^{\ast}b)&=\begin{pmatrix} -\pi_{1}(p+\xi \gamma')\xi_{be_{k}} \\ -\gamma' \xi_{be_{k}} \\ 0 \\ 0 \end{pmatrix} \\
&=\begin{pmatrix} -\pi_{1}p'\xi_{be_{k}} \\ -\gamma' \xi_{be_{k}} \\ 0 \\ 0 \end{pmatrix} \\
&=\overline{N'}(^{\ast}b)
\end{align*}

Now let $\varphi: \overline{N} \rightarrow \overline{N}'$ be the map defined as $\varphi_{j}=id$ if $j \neq k$ and $\varphi_{k}=\psi$. Suppose first that $a \in _{k}T$, $d_{1} \in D_{\tau(a)}$, $d_{2} \in D_{k}$ and $w \in \overline{N}_{k}$. Then
\begin{align*}
\varphi(d_{1}a^{\ast}d_{2}w)&=d_{1}a^{\ast}d_{2}w \\
&=d_{1}\overline{N}(a^{\ast})(d_{2}w) \\
&=d_{1}\overline{N'}(a^{\ast})\psi(d_{2}w) \\
&=d_{1}a^{\ast}d_{2}\varphi(w) 
\end{align*}

Now if $b \in T_{k}$, $d_{1} \in D_{k}$, $d_{2} \in D_{\sigma(b)}$ and $n \in N_{\sigma(b)}$. Then

\begin{align*}
\varphi(d_{1}(^{\ast}b)d_{2}n)&=\psi(d_{1}(^{\ast}b)d_{2}n) \\
&=\psi(d_{1}\overline{N}(^{\ast}b)(d_{2}n)) \\
&=d_{1}\psi(\overline{N}(^{\ast}b)(d_{2}n)) \\
&=d_{1}\overline{N'}(^{\ast}b)(d_{2}n) \\
&=d_{1}(^{\ast}b)d_{2}\varphi(n)
\end{align*}

Therefore for each $u \in \mathcal{F}_{S}(\mu_{k}M)$ we obtain a commutative diagram:
\begin{center}
\begin{equation*}
\xymatrix{%
\overline{N} \ar[r]^{u_{\overline{N}}} \ar[d]^{\varphi} & \overline{N} \ar[d]^{\varphi} \\
\overline{N'} \ar[r]^{u_{\overline{N'}}} & \overline{N'}}
\end{equation*}
\end{center}

This proves that the decorated representations $\widetilde{\mu_{k}}(\mathcal{N})=(\overline{N},\overline{V})$ and $(\overline{N'},\overline{V})$ are right-equivalent, as desired.

\end{proof}

Let $\mathcal{N}=(N,V)$ be a decorated representation of $(\mathcal{F}_{S}(M),P)$ and let $\mathcal{N}'=(N',V')$ be a decorated representation of $(\mathcal{F}_{S}(M'),P')$. Suppose that such representations are right-equivalent, then there exists an algebra isomorphism $\varphi: \mathcal{F}_{S}(M) \rightarrow \mathcal{F}_{S}(M')$ such that $\varphi(P)$ is cyclically equivalent to $P'$. By \cite[Theorem 5.3]{2} we have: $R(P')=R(\varphi(P))=\varphi(R(P))$. Using the representation $\mathcal{N}=(N,V)$ we construct a decorated representation $\widehat{\mathcal{N}}=(\widehat{N},V)$ of $(\mathcal{F}_{S}(M'),\varphi(P))$ as follows: let  $\widehat{N}=N$ as $F$-vector spaces and given $u \in \mathcal{F}_{S}(M')$ and $n \in N$ define $u \ast n:=\varphi^{-1}(u)n$. Clearly $R(P')\widehat{N}=0$. We will denote $\widehat{N}$ by $\widehat{N}=^{\varphi^{-1}}N$.

\begin{prop} \label{prop5} Let $\varphi: \mathcal{F}_{S}(M) \rightarrow \mathcal{F}_{S}(M)$ be a unitriangular automorphism and let $\mathcal{N}=(N,V)$ be a decorated representation of $(\mathcal{F}_{S}(M),P)$ where $P$ is a potential in $\mathcal{F}_{S}(M)$ such that $e_{k}Pe_{k}=0$. Then: 

\begin{enumerate}[(a)]
\item There exists a unitriangular automorphism $\hat{\varphi}: \mathcal{F}_{S}(\mu_{k}M) \rightarrow \mathcal{F}_{S}(\mu_{k}M)$ such that $\hat{\varphi}(\mu_{k}P)$ is cyclically equivalent to $\mu_{k}(\varphi(P))$.
\item There exists an isomorphism of decorated representations $\psi: \widetilde{\mu_{k}}(\mathcal{N}) \rightarrow \widetilde{\mu_{k}}(\widehat{\mathcal{N}})$.
\end{enumerate}
\end{prop}

\begin{proof} The fact that (a) holds is an immediate consequence of \cite[Theorem 8.12]{2}. Let us show (b). Let $\hat{\alpha},\hat{\beta}$ and $\hat{\gamma}$ be the maps associated to the representation $\widehat{N}$. Recall that $_{k}\hat{T}=\{sa: a \in _{k}T, s \in L(k)\}$ is a local basis for $(e_{k}M)_{S}$. We have that $\varphi(sa)=\displaystyle \sum_{r \in L(k), a_{1} \in _{k}T} ra_{1}C_{ra_{1},sa}$ for some $C_{ra_{1},sa} \in e_{\tau(a_{1})}\mathcal{F}_{S}(M)e_{\tau(a)}$. \\

Define $C: N_{in} \rightarrow N_{in}$ as the $F$- linear map such that for all $r,s \in L(k)$, $a,a_{1} \in _{k}T$, the map:

\begin{center}
$\pi_{ra_{1}}C\xi_{sa}: N_{\tau(a)} \rightarrow N_{\tau(a_{1})}$
\end{center}

is given by
\begin{center}
$\pi_{ra_{1}}C\xi_{sa}(n)=\varphi^{-1}(C_{ra_{1},sa})n$
\end{center}

for every $n \in N_{\tau(a)}$. Let us show that $\hat{\alpha}C=\alpha$. It suffices to show that for all $a \in _{k}T, r \in L(k)$ we have $\hat{\alpha}C\xi_{ra}=\alpha \xi_{ra}$. \\

In what follows, given $h \in \mathcal{F}_{S}(M)$ and $n \in N$ then $h \ast n=\varphi^{-1}(h)n$ denotes the product in $\widehat{N}$. \\

We have
\begin{align*}
\hat{\alpha}C\xi_{ra}(n)&=\displaystyle \sum_{s \in L(k), a_{1} \in _{k}T} \hat{\alpha}\xi_{sa_{1}}\pi_{sa_{1}}C\xi_{ra}(n) \\
&=\displaystyle \sum_{s \in L(k), a_{1} \in _{k}T} \hat{\alpha} \xi_{sa_{1}}\left(\varphi^{-1}(C_{sa_{1},ra})n\right) \\
&=\displaystyle \sum_{s \in L(k), a_{1} \in _{k}T} sa_{1} \ast \varphi^{-1}(C_{sa_{1},ra})n \\
&=\varphi^{-1} \left( \displaystyle \sum_{s \in L(k), a_{1} \in _{k}T} sa_{1}C_{sa_{1},ra}\right)n \\
&=\varphi^{-1}(\varphi(ra))n \\
&=ran \\
&=\alpha \xi_{ra}(n)
\end{align*}

and therefore $\hat{\alpha}C=\alpha$. This yields the following equalities:

\begin{equation} \label{4.5}
\begin{split}
\operatorname{ker}(\alpha)&=C^{-1}(\operatorname{ker}(\hat{\alpha})) \\
\operatorname{im}(\alpha)&=\operatorname{im}(\hat{\alpha})
\end{split}
\end{equation}

Similarly, for each $b \in T_{k}$ and $s \in L(k)$:

\begin{center}
$\varphi(bs)=\displaystyle \sum_{r \in L(k), b_{1} \in T_{k}} D_{bs,b_{1}r}b_{1}r$
\end{center}

for some $D_{bs,b_{1}r} \in e_{\sigma(b)}\mathcal{F}_{S}(M)e_{\sigma(b_{1})}$. \\

Thus there exists an $F$-linear map $D: N_{out} \rightarrow N_{out}$ such that for all $r,s \in L(k)$, $b,b_{1} \in T_{k}$ we have:

\begin{center}
$\pi_{bs}D \xi_{b_{1}r}(n)=\varphi^{-1}(D_{bs,b_{1}r})n$
\end{center}

for every $n \in N_{\sigma(b_{1})}$. We now show that $D\hat{\beta}=\beta$. It suffices to show that for all $b \in T_{k}$, $s \in L(k)$ we have $\pi_{bs}D\hat{\beta}=\pi_{bs}\beta$. Let $n \in N_{k}$, then
\begin{align*}
\pi_{bs}D\hat{\beta}(n)&=\displaystyle \sum_{r \in L(k),b_{1} \in T_{k}} \pi_{bs}D\xi_{b_{1}r}\pi_{b_{1}r} \hat{\beta}(n) \\
&=\displaystyle \sum_{r \in L(k), b_{1} \in T_{k}} \varphi^{-1}(D_{bs,b_{1}r})((b_{1}r) \ast n) \\
&=\displaystyle \sum_{r \in L(k), b_{1} \in T_{k}} \varphi^{-1}(D_{bs,b_{1}r})\varphi^{-1}(b_{1}r)n \\
&=\varphi^{-1} \left(\displaystyle \sum_{r \in L(k), b_{1} \in T_{k}} D_{bs,b_{1}r} b_{1}r\right)n \\
&=\varphi^{-1}(\varphi(bs))n \\
&=bsn \\
&=\pi_{bs}\beta(n)
\end{align*}

Therefore $D\hat{\beta}=\beta$, as claimed. Then we obtain the following equalities

\begin{equation}  \label{4.6}
\begin{split}
\operatorname{im}(\beta)&=D(\operatorname{im}(\hat{\beta})) \\
\operatorname{ker}(\beta)&=\operatorname{ker}(\hat{\beta})
\end{split}
\end{equation}

\begin{lemma} \label{lem4} We have that $\hat{\gamma}=C \gamma D$.
\end{lemma}

\begin{proof}

Using \cite[Lemma 9.2]{2} we obtain an algebra isomorphism: \\

\begin{center}
$\rho: \mathcal{F}_{S}(M)_{\hat{k},\hat{k}} \rightarrow \mathcal{F}_{S}((\mu_{k}M)_{\hat{k},\hat{k}})$
\end{center}

\vspace{0.1in}

we may view $\rho$ as a monomorphism of algebras: \\

\begin{center}
$\rho: \bar{e_{k}}\mathcal{F}_{S}(M)\bar{e_{k}} \rightarrow \bar{e_{k}} \mathcal{F}_{S}(\mu_{k}M) \bar{e_{k}}$
\end{center}

\vspace{0.1in}

By \cite[Proposition 8.11]{2} we have algebra isomorphisms:

\begin{align*}
&\phi: \mathcal{F}_{S}(\widehat{M}) \rightarrow \mathcal{F}_{S}(\widehat{M}) \\
&\hat{\varphi}: \mathcal{F}_{S}(\mu_{k}M) \rightarrow \mathcal{F}_{S}(\mu_{k}M)
\end{align*}

where $\widehat{M}:=M \oplus (e_{k}M)^{\ast} \oplus ^{\ast} (Me_{k})$ and
\begin{align*}
&i_{M}: \mathcal{F}_{S}(M) \rightarrow \mathcal{F}_{S}(\widehat{M}) \\
&i_{\mu_{k}M}: \mathcal{F}_{S}(\mu_{k}M) \rightarrow \mathcal{F}_{S}(\widehat{M})
\end{align*}

are the inclusion maps. We also have commutative diagrams

\begin{center}
\begin{equation*}
\xymatrix{%
\mathcal{F}_{S}(\mu_{k}M) \ar[r]^{\hat{\varphi}} \ar[d]^{i_{\mu_{k}M}} & \mathcal{F}_{S}(\mu_{k}M) \ar[d]^{i_{\mu_{k}M}} \\
\mathcal{F}_{S}(\widehat{M}) \ar[r]^{\phi} & \mathcal{F}_{S}(\widehat{M})}
\end{equation*}
\end{center}

\begin{center}
\begin{equation*}
\xymatrix{%
\mathcal{F}_{S}(M) \ar[r]^{\varphi} \ar[d]^{i_{M}} & \mathcal{F}_{S}(M) \ar[d]^{i_{M}} \\
\mathcal{F}_{S}(\widehat{M}) \ar[r]^{\phi} & \mathcal{F}_{S}(\widehat{M})}
\end{equation*}
\end{center}

Let us see that the previous diagram induces a commutative diagram:
\begin{center}
\begin{equation*}
\xymatrix{%
\mathcal{F}_{S}(M)_{\hat{k},\hat{k}} \ar[r]^{\rho} \ar[d]^{\varphi} & \mathcal{F}_{S}((\mu_{k}M)_{\hat{k},\hat{k}}) \ar[d]^{\hat{\varphi}} \\
\mathcal{F}_{S}(M)_{\hat{k},\hat{k}} \ar[r]^{\rho} & \mathcal{F}_{S}((\mu_{k}M)_{\hat{k},\hat{k}})}
\end{equation*}
\end{center}

Indeed, on one hand $i_{\mu_{k}M}\rho \varphi = i_{M}\varphi$ and on the other hand:
\begin{center}
$i_{\mu_{k}M}\hat{\varphi}\rho=\phi i_{\mu_{k}M} \rho = \phi i_{M}=i_{M} \varphi$
\end{center}

Since $i_{\mu_{k}M}$ is injective then $\rho \varphi=\hat{\varphi}\rho$. \\

Let $\hat{\Delta}: T_{S}(\mu_{k}M) \rightarrow T_{S}(\mu_{k}M) \otimes_{Z} T_{S}(\mu_{k}M)$ be the derivation associated to $T_{S}(\mu_{k}M)$. Define maps:
\begin{align*}
\rho^{k}: \bar{e_{k}}T_{S}(M) \rightarrow T_{S}(\mu_{k}M) \\
^{k}\rho: T_{S}(M)\bar{e_{k}} \rightarrow T_{S}(\mu_{k}M)
\end{align*}

as follows, $\rho^{k}(z)=\rho(z\bar{e_{k}})$ and $^{k}\rho(z)=\rho(\bar{e_{k}}z)$.

\begin{lemma} \label{lem5} For $z \in T_{S}(M)_{\hat{k},\hat{k}}$ we have that $\hat{\Delta}\rho(z)=(\rho^{k} \otimes (^{k}\rho))\Delta(z)$.
\end{lemma}

\begin{proof} The $T_{S}(\mu_{k}M)$-bimodule $T_{S}(\mu_{k}M) \otimes_{Z} T_{S}(\mu_{k}M)$ is a $T_{S}(M)_{\hat{k},\hat{k}}$-bimodule via the map $\rho$. We have that $\hat{\Delta}$ is a $T_{S}(M)_{\hat{k},\hat{k}}$-derivation, $\rho^{k} \otimes (^{k} \rho)$ is a map of $T_{S}(M)_{\hat{k},\hat{k}}$-bimodules and $\Delta$ is a derivation of $T_{S}(M)$. Therefore $\hat{\Delta}\rho$ and $(\rho^{k} \otimes (^{k}\rho))\Delta$ are derivations of $T_{S}(\mu_{k}M)$. Since $T_{S}(M)_{\hat{k},\hat{k}}$ is generated, as an $F$-algebra, by $\bar{e_{k}}S$, $\bar{e_{k}}M_{0}\bar{e_{k}}$ and $M_{0}D_{k}M_{0}$, then it suffices to establish the equality for $z \in \bar{e_{k}}S \cup \bar{e_{k}}M_{0}\bar{e_{k}} \cup M_{0}D_{k}M_{0}$. \\

If $z \in \bar{e_{k}}S$, then
\begin{center}
$\hat{\Delta}\rho(z)=1 \otimes \rho(z) - \rho(z) \otimes 1 = \bar{e_{k}} \otimes \rho(z) - \rho(z) \otimes \bar{e_{k}}=(\rho^{k} \otimes (^{k}\rho))\Delta(z)$
\end{center}

For $z \in \bar{e_{k}}M_{0}\bar{e_{k}}$ we have
\begin{center}
$\hat{\Delta}\rho(z)=1 \otimes \rho(z)=\bar{e_{k}} \otimes \rho(z)=(\rho^{k} \otimes (^{k}\rho))\Delta(z)$
\end{center}

If $z=m_{1}rm_{2}$ where $m_{1} \in \bar{e_{k}}M_{0}e_{k}$, $r \in D_{k}$ and $m_{2} \in e_{k}M_{0}\bar{e_{k}}$, then
\begin{center}
$\hat{\Delta}\rho(m_{1}rm_{2})=1 \otimes \rho(m_{1}rm_{2})=\bar{e_{k}} \otimes \rho(m_{1}rm_{2})$
\end{center}

and
\begin{align*}
(\rho^{k} \otimes ^{k} \rho)\Delta(m_{1}rm_{2})&=(\rho^{k} \otimes (^{k}\rho))(\Delta(m_{1})rm_{2}+m_{1}\Delta(rm_{2})) \\
&=(\rho^{k} \otimes ^{k} \rho)(1 \otimes m_{1}rm_{2}+m_{1} \otimes rm_{2}) \\
&=\bar{e_{k}} \otimes \rho(m_{1}rm_{2}) + m_{1}\bar{e_{k}} \otimes (^{k}\rho(rm_{2})) \\
&=\bar{e_{k}} \otimes \rho(m_{1}rm_{2})
\end{align*}

completing the proof of Lemma \ref{lem5}.
\end{proof}

\begin{lemma} \label{lem6} For $\alpha \in T_{S}(M)_{\hat{k},\hat{k}}$, $z \in \mathcal{F}_{S}(M)_{\hat{k},\hat{k}}$ we have

\begin{center}
$((\rho^{k} \otimes (^{k}\rho))\Delta(\alpha)) \Diamond \rho(z)=\rho^{k}(\Delta(\alpha) \Diamond z)$
\end{center}
\end{lemma}

\begin{proof} One can verify that if the equality holds for all $\alpha$ and for all $z$, and for all $\beta$ and every $z$ then it holds for all $\alpha \beta$ and every $z$. Therefore it suffices to establish the equality for $\alpha \in \bar{e_{k}}S \cup \bar{e_{k}}M_{0}\bar{e_{k}} \cup M_{0}D_{k}M_{0}$. \\
(i) Suppose first that $\alpha \in \bar{e_{k}}S$, then
\begin{align*}
(\rho^{k} \otimes (^{k}\rho))\Delta(\alpha) \Diamond \rho(z)&=(\rho^{k} \otimes (^{k}\rho))(1 \otimes \alpha - \alpha \otimes 1) \Diamond \rho(z) \\
&=(\bar{e_{k}} \otimes \rho(\alpha)-\rho(\alpha) \otimes \bar{e_{k}}) \Diamond \rho(z) \\
&=cyc(\rho(\alpha)\rho(z)-\rho(z)\rho(\alpha)) \\
&=\rho(cyc(\alpha z - z\alpha)) \\
&=\rho^{k}(\Delta(\alpha) \Diamond z)
\end{align*}

(ii) If $\alpha \in \bar{e_{k}}M_{0}\bar{e_{k}}$, then
\begin{align*}
(\rho^{k} \otimes (^{k}\rho))\Delta(\alpha) \Diamond \rho(z)&=(\rho^{k} \otimes (^{k}\rho))(1 \otimes \alpha) \Diamond \rho(z) \\
&=(\bar{e_{k}} \otimes p(\alpha)) \Diamond \rho(z) \\
&=cyc(\rho(\alpha)\rho(z)) \\
&=\rho(cyc(\alpha z)) \\
&=\rho^{k}(\Delta(\alpha) \Diamond z)
\end{align*}

(iii) Finally, if $\alpha=m_{1}rm_{2}$ where $m_{1} \in \bar{e_{k}}M_{0}e_{k}$, $r \in D_{k}$ and $m_{2} \in e_{k}M_{0}\bar{e_{k}}$, then
\begin{align*}
(\rho^{k} \otimes (^{k}\rho))\Delta(m_{1}rm_{2}) \Diamond \rho(z) &= (\rho^{k} \otimes (^{k}\rho))(1 \otimes m_{1}rm_{2}+m_{1} \otimes rm_{2}) \Diamond \rho(z) \\
&=(\bar{e_{k}} \otimes \rho(m_{1}rm_{2})) \Diamond \rho(z) \\
&=cyc(\rho(m_{1}rm_{2}z)) \\
&=\rho^{k}(\Delta(m_{1}rm_{2}) \Diamond z)
\end{align*}
\end{proof}

Lemma \ref{lem5} and Lemma \ref{lem6} imply immediately the following

\begin{lemma} \label{lem7} Let $\alpha \in T_{S}(M)_{\hat{k},\hat{k}}$, $z \in \mathcal{F}_{S}(M)_{\hat{k},\hat{k}}$, then:

\begin{center}
$\hat{\Delta}\rho(\alpha) \Diamond \rho(z)=\rho^{k}(\Delta(\alpha) \Diamond z)$
\end{center}

\end{lemma}

\vspace{0.1in}

Let $r,s \in L(k)$, $a \in _{k}T$ and $b \in T_{k}$. Let $n \in \hat{N}_{\sigma(b)}$, then
\begin{center}
$\hat{\gamma}_{sa,br}(n)=\displaystyle \sum_{w \in L(k)} s^{\ast}(r^{-1}w)\varphi^{-1} \left(Y_{[bwa]}(\varphi(P)) \right)n$.
\end{center}

We have
\begin{align*}
\varphi^{-1} \left( Y_{[bwa]}(\varphi(P)) \right)n&= \varphi^{-1} \rho^{-1} \left( X_{[bwa]^{\ast}}(\rho \varphi(P)) \right)n \\
&=\left(\rho^{-1} \hat{\varphi}^{-1} X_{[bwa]^{\ast}}(\hat{\varphi}(\rho(P))) \right)n
\end{align*}

Also $X_{[bwa]^{\ast}}(\hat{\varphi}(\rho(P))=\displaystyle \lim_{u \to \infty} Z_{u}$ where:

\begin{center}
$Z_{u}=\displaystyle \sum_{s \in L(\sigma(b))} \displaystyle \sum_{c \in \tilde{T}_{\hat{k},\hat{k}}} (s\rho(bwa))^{\ast} \left ( \hat{\Delta}(\hat{\varphi}(c)^{\leq u+1}) \Diamond \hat{\varphi}(X_{c^{\ast}}(\rho(P)) \right)s$
\end{center}

Let $v$ be an arbitrary positive integer and let $\alpha,\beta \in \mathcal{F}_{S}(M)$. We will write $\alpha \equiv \beta \ (v)$ if $\alpha-\beta \in \mathcal{F}_{S}(M)^{>v}$. \\

Clearly for any $\alpha \in \mathcal{F}_{S}(M)$ and for every positive integer $v$ we have $\alpha \equiv \alpha^{\leq v} \ (v)$.  \\

Note that if $h \in T_{S}(M)^{>v}$ and $z \in \mathcal{F}_{S}(M)$, then $\Delta(h) \Diamond z \in \mathcal{F}_{S}(M)^{>v}$. Therefore if $\alpha,\beta \in T_{S}(M)$ and $\alpha \equiv \beta \ (v)$, then $\Delta(\alpha) \Diamond z \equiv \Delta(\beta) \Diamond z \ (v)$ for every $z \in \mathcal{F}_{S}(M)$. \\

Let $h \in \mathcal{F}_{S}(M)_{\hat{k},\hat{k}}$. Let us see that:

\begin{center}
$\rho(h^{\leq 2v+3}) \equiv \rho(h)^{\leq v+1} \ (v+2)$
\end{center}

Indeed, since $h-h^{\leq 2v+3} \in \mathcal{F}_{S}(M)^{\geq 2v+4}$ then $\rho(h)-\rho(h^{\leq 2v+3}) \in \mathcal{F}_{S}(\mu_{k}M)^{\geq v+2}$, whence $\rho(h) \equiv \rho(h^{\leq 2v+3}) \ (v+2)$ and therefore $\rho(h) \equiv \rho(h)^{\leq v+1} \ (v+2)$. \\

Let $v \gg 0$ be such that $\mathcal{F}_{S}(M)^{\geq v}N=0$. \\

For $i \geq 1$ we have that $Z_{v+i}-Z_{v} \in \mathcal{F}_{S}(\mu_{k}M)^{\geq v+1}$ and thus
\begin{center}
$\varphi^{-1}\rho^{-1}(Z_{v+i}) \equiv \varphi^{-1} \rho^{-1}(Z_{v}) \ (v+1)$
\end{center}

therefore
\begin{center}
$\displaystyle \lim_{u \to \infty} \varphi^{-1} \rho^{-1}(Z_{u})n=\varphi^{-1}\rho^{-1}(Z_{v})n$
\end{center}

Define:
\begin{center}
$W(c)=\varphi^{-1}\rho^{-1} \left( \displaystyle \sum_{s \in L(\sigma(b))} (s\rho(bwa))^{\ast}\left( \hat{\Delta}(\hat{\varphi}(c)^{\leq v+1}) \Diamond \hat{\varphi}(X_{c^{\ast}}(\rho(P)))\right)s\right)n$.
\end{center}

We require the following:

\begin{lemma} \label{lem8} If $c \in \bar{e_{k}}M_{0}\bar{e_{k}} \cap T$ then $W(c)=0$. 
\end{lemma}

\begin{proof} Note that $X_{c^{\ast}}(\rho(P))=\rho(X_{c^{\ast}}(P))$, hence\\
\begin{center}
$\hat{\varphi}(c)^{\leq v+1} =\hat{\varphi}(\rho(c))^{\leq v+1}=\rho(\varphi(c))^{\leq v+1} \equiv \rho(\varphi(c)^{\leq 2v+3})  \ (v+2)$
\end{center}

Consequently
\begin{align*}
W(c)&=\varphi^{-1}\rho^{-1} \left( \displaystyle \sum_{s \in L(\sigma(b))} (s\rho(bwa))^{\ast} \left( \hat{\Delta}(\rho({\varphi}(c)^{\leq 2v+3})) \Diamond \hat{\varphi}\rho(X_{c^{\ast}}(P))\right)s\right)n \\
&=\varphi^{-1}\rho^{-1} \left( \displaystyle \sum_{s \in L(\sigma(b))} (s\rho(bwa))^{\ast}\left( \hat{\Delta}(\rho({\varphi}(c)^{\leq 2v+3})) \Diamond \rho{\varphi}(X_{c^{\ast}}(P))\right)s\right)n \\
\end{align*}

Using Lemma \ref{lem7} we obtain
\begin{center}
$W(c)=\varphi^{-1}\rho^{-1} \left( \displaystyle \sum_{s \in L(\sigma(b))} (s\rho(bwa))^{\ast} \left(\rho^{k} \left( \Delta(\varphi(c)^{\leq 2v+3}) \Diamond \varphi(X_{c^{\ast}}(P)) \right)\right)s\right)n$
\end{center}

Letting $z=\varphi(X_{c^{\ast}}(P))$ yields that $W(c)$ is a sum of elements of the form:

\begin{center}
$\varphi^{-1}\rho^{-1} \left( \displaystyle \sum_{s \in L(\sigma(b))} (s\rho(bwa))^{\ast} \left(\rho^{k} \left(\Delta(m_{1}\hdots m_{l}r) \Diamond z\right)\right)s\right)n$
\end{center}

where $m_{1},\hdots,m_{l} \in SM_{0}$ and $r \in \bar{e_{k}}S$.\\ 

Then
\begin{center}
$\displaystyle \sum_{s \in L(\sigma(b))} (s\rho(bwa))^{\ast} \left(\rho^{k} \left(\Delta(m_{1}\hdots m_{l}r) \Diamond z\right)\right)s$
\end{center}

is equal to
\begin{center}
$\displaystyle \sum_{s \in L(\sigma(b))} (s\rho(bwa))^{\ast} \left(\rho^{k} \left(\Delta(m_{1}\hdots m_{l})r \Diamond z\right)\right)s+g$
\end{center}

where
\begin{align*}
g&=\displaystyle \sum_{s \in L(\sigma(b))} (s\rho(bwa))^{\ast} \left(\rho^{k} \left(m_{1}\hdots m_{l}\Delta(r) \Diamond z\right)\right)s \\
&=\displaystyle \sum_{s \in L(\sigma(b))} (s\rho(bwa))^{\ast}\left(r\rho(cyc(\bar{e_{k}}zm_{1}\hdots m_{l}))-\rho(cyc(\bar{e_{k}}zm_{1}\hdots m_{l}))r\right)s
\end{align*}

By \cite[Lemma 5.2]{2} we obtain that $g=0$. Therefore:

\begin{center}
$W(c)=\varphi^{-1}\rho^{-1} \left( \displaystyle \sum_{s \in L(\sigma(b))} (s\rho(bwa))^{\ast} \left( \rho^{k} \left(\Delta(m_{1}\hdots m_{l}) \Diamond rz\right)\right)s\right)n$
\end{center}

then $W(c)$ is a sum of elements of the form:

\begin{center}
$\varphi^{-1}\rho^{-1}\left( (s\rho(bwa))^{\ast}(\rho(\bar{e_{k}}m_{i}\hdots m_{l}rzm_{1}\hdots m_{i-1}))s\right)n$
\end{center}

If $i=l$ then $z=\displaystyle \sum_{i} z_{i}z_{i'}$ where $z_{i} \in \bar{e_{k}}M$ and in this case: \\

\begin{center}
$(s\rho(bwa))^{\ast}(\rho(\bar{e_{k}}m_{l}r\bar{e_{k}}zm_{1}\hdots m_{l-1}))=(s\rho(bwa))^{\ast}(\rho(\bar{e_{k}}m_{l}r)\rho(\bar{e_{k}}zm_{1}\hdots m_{l-1}\bar{e_{k}}))=0$
\end{center}

Hence $W(c)$ is a sum of elements of the form:

\begin{center}
$\varphi^{-1}\rho^{-1}\left((s\rho(bwa))^{\ast}(\rho(\bar{e_{k}}m_{i}m_{i+1})\rho(\alpha z \beta))s\right)n$
\end{center}

where $m_{i}m_{i+1} \in Me_{k}M$ and thus $W(c)$ is a sum of elements of the form $\varphi^{-1}(\alpha)\varphi^{-1}(z)\varphi^{-1}(\beta)n$. Since $\varphi^{-1}(z)=X_{c^{\ast}}(P)$ then $W(c) \in R(P)N=0$, completing the proof of Lemma \ref{lem8}. \\
\end{proof}

From the above we obtain the following formula:

\begin{center}
$(\ast): \varphi^{-1}(Y_{[bwa]}(\varphi(P)))n=\varphi^{-1}\rho^{-1}(Z_{v}')n$
\end{center}

where:
\begin{align*}
Z_{v}'&=\displaystyle \sum_{s \in L(\sigma(b))} \displaystyle \sum_{b_{1} \in T_{k}, r \in L(k), a_{1} \in _{k}T} (s\rho(bwa))^{\ast} \left( \hat{\Delta}(\hat{\varphi}(\rho(b_{1}ra_{1})))^{\leq v+1} \Diamond \hat{\varphi}(X_{\rho(b_{1}ra_{1})}(\rho(P)))\right)s \\
&=\displaystyle \sum_{s \in L(\sigma(b))} \displaystyle \sum_{b_{1}ra_{1}} (s\rho(bwa))^{\ast} \left( \hat{\Delta}(\hat{\varphi}(\rho(b_{1}ra_{1})))^{\leq v+1} \Diamond \hat{\varphi}\rho \left(Y_{[b_{1}ra_{1}]}(P)\right)\right)s  \\
&=\displaystyle \sum_{s \in L(\sigma(b))} \displaystyle \sum_{b_{1}ra_{1}} (s\rho(bwa))^{\ast} \left( \hat{\Delta}(\rho\varphi(b_{1}ra_{1}))^{\leq v+1} \Diamond \rho\varphi \left(Y_{[b_{1}ra_{1}]}(P)\right)\right)s 
\end{align*}

Thus, in $(\ast)$ we may replace $Z_{v}'$ by the term:
\begin{center}
$\displaystyle \sum_{b_{1}ra_{1}} \displaystyle \sum_{s \in L(\sigma(b))} (s\rho(bwa))^{\ast}\left( \hat{\Delta}(\rho({\varphi}(b_{1}ra_{1})^{\leq 2v+3})) \Diamond \rho{\varphi}(Y_{[b_{1}ra_{1}]}(P))\right)s$
\end{center}

which by Lemma \ref{lem7} is equal to:

\begin{center}
$\displaystyle \sum_{b_{1}ra_{1}} \displaystyle \sum_{s \in L(\sigma(b))} (s\rho(bwa))^{\ast}\left( \rho^{k} \left(\Delta(\varphi(b_{1}ra_{1})^{\leq 2v+3}) \Diamond \varphi(Y_{[b_{1}ra_{1}]}(P))\right)\right)s$
\end{center}

the latter term can be replaced by:

\begin{center}
$\displaystyle \sum_{b_{1}ra_{1}} \displaystyle \sum_{s \in L(\sigma(b))} (s\rho(bwa))^{\ast}\left( \rho^{k} \left(\Delta(\varphi(b_{1}r)^{\leq 2v+3}\varphi(a_{1})^{\leq 2v+3}) \Diamond \varphi(Y_{[b_{1}ra_{1}]}(P))\right)\right)s$
\end{center}

which in turn can be replaced by $S=S_{1}+S_{2}$, where:

\begin{align*}
S_{1}&=\displaystyle \sum_{s \in L(\sigma(b))} \displaystyle \sum_{b_{1}ra_{1}} (s\rho(bwa))^{\ast} \left( \rho^{k} \left(\Delta(\varphi(b_{1})^{\leq 2v+3}) \Diamond \varphi(ra_{1}Y_{[b_{1}ra_{1}]}(P))\right)\right)s \\
S_{2}&= \displaystyle \sum_{s \in L(\sigma(b))}  \displaystyle \sum_{b_{1}ra_{1}} (s\rho(bwa))^{\ast} \left( \rho^{k} \left(\Delta(\varphi(a_{1})^{\leq 2v+3}) \Diamond \varphi(Y_{[b_{1}ra_{1}]}(P)b_{1}r)\right)\right)s
\end{align*}

Using \ref{3.7} gives that:

\begin{center}
$S_{2}= \displaystyle \sum_{s \in L(\sigma(b))}  \displaystyle \sum_{a_{1}} (s\rho(bwa))^{\ast} \left( \rho^{k} \left(\Delta(\varphi(a_{1})^{\leq 2v+3}) \Diamond \varphi(X_{a_{1}^{\ast}}(P))\right)\right)s$
\end{center}

whence:

\begin{center}
$\varphi^{-1}(Y_{[bwa]}(\varphi(P)))n=\varphi^{-1}\rho^{-1}(S_{1})n+\varphi^{-1}\rho^{-1}(S_{2})n$
\end{center}

(i) Let us see that $\varphi^{-1}\rho^{-1}(S_{2}) \subseteq R(P)$. \\

Let $z=\varphi(X_{a_{1}^{\ast}}(P))$, then $\varphi(a_{1})^{\leq 2v+3}$ is a sum of elements of the form $m_{1} \hdots m_{l}t$ where $m_{j} \in SM_{0}$, $t \in \bar{e_{k}}S$. Therefore $S_{2}$ is a sum of elements of the form: 
\begin{center}
$(s\rho(bwa))^{\ast}(\rho(\bar{e_{k}}m_{i}\hdots m_{l}tzm_{1} \hdots m_{i-1})\bar{e_{k}})s$
\end{center}

If $i=l$, then $(s\rho(bwa))^{\ast}(\rho(\bar{e_{k}}m_{i}\hdots m_{l}\hdots m_{i-1})\bar{e_{k}})s$ is equal to:

\begin{center}
$(s\rho(bwa))^{\ast}(\rho(\bar{e_{k}}m_{l}\bar{e_{k}})\rho(\bar{e_{k}}tzm_{1}\hdots m_{l-1})\bar{e_{k}})s=0$
\end{center}

If $i<l$ then: 

\begin{center}
$(s\rho(bwa))^{\ast}(\rho(\bar{e_{k}}m_{i}\hdots m_{l}tzm_{1}\hdots m_{i-1})\bar{e_{k}})s=(s\rho(bwa))^{\ast}(\rho(\bar{e_{k}}m_{i}\hdots m_{l})\rho(tzm_{1}\hdots m_{i-1}\bar{e_{k}}))s$
\end{center}

Therefore $S_{2}$ is a sum of elements of the form $\rho(\alpha z \beta)$, so $\varphi^{-1} \rho^{-1}(S_{2})$ is a sum of elements of the form $\varphi^{-1}(\alpha)\varphi^{-1}(z)\varphi^{-1}(\beta)$. Since $\varphi^{-1}(z)=X_{a_{1}^{\ast}}(P)$ then $\varphi^{-1}\rho^{-1}(S_{2}) \subseteq R(P)$ which establishes (i). \\

From the above it follows that $\varphi^{-1}(Y_{[bwa]}(\varphi(P)))n=\varphi^{-1}\rho^{-1}(S_{1})n$. \\

(ii) Let us show that $\varphi^{-1}\rho^{-1}(S_{1})=\nu_{1}+\nu_{2}$, where $\nu_{1} \in R(P)$ and $\nu_{2}$ is a sum of elements of the form: \\
\begin{center}
$\displaystyle \sum_{s \in L(\sigma(b))} \displaystyle \sum_{ra_{1}} (s\rho(bwa))^{\ast} \rho(\bar{e_{k}}m_{i} \hdots m_{l}z_{r_{a_{1}}}m_{1} \hdots m_{i-1}\bar{e_{k}})s$
\end{center}

where $z_{r_{a_{1}}}=\varphi(ra_{1}Y_{[b_{1}ra_{1}]}(P))$.

Note that $\varphi(b_{1})^{\leq 2v+3}$ is a sum of elements of the form $m_{1}m_{2} \hdots m_{l}$ where $m_{1},\hdots,m_{l-1} \in SM_{0}$ and $m_{l} \in \bar{e_{k}}Me_{k}$. The element $\varphi(b_{1})$ is a sum of elements of the form $m_{1} \hdots m_{l}$ where $m_{1},\hdots,m_{l-1} \in SM_{0}$ and $m_{l} \in \bar{e_{k}}Me_{k}$. Then $\varphi^{-1}\rho^{-1}(S_{1})$ lies in the $F$-vector space generated by $\varphi^{-1}\rho^{-1}(T_{i})$ where $T_{i}$ is the $F$-vector space generated by elements of the form:
\begin{center}
$u_{i}=\displaystyle \sum_{s \in L(\sigma(b))} \displaystyle \sum_{ra_{1}} (s\rho(bwa))^{\ast} \rho(\bar{e_{k}}m_{i} \hdots m_{l}z_{r_{a_{1}}}m_{1} \hdots m_{i-1}\bar{e_{k}})s$
\end{center}

Let us show that if $i<l$ then $\varphi^{-1}\rho^{-1}(T_{i}) \subseteq R(P)$. We have that
\begin{center}
$u_{i}=\displaystyle \sum_{s \in L(\sigma(b))} (s\rho(bwa))^{\ast} \left(\rho(\bar{e_{k}}m_{i} \hdots m_{l-1})\right)\rho(m_{l}w_{b_{1}}m_{1} \hdots m_{i-1}\bar{e_{k}})s$
\end{center}

where $w_{b_{1}}=\varphi(X_{b_{1}^{\ast}}(P))$. \\

It follows that $u_{i}$ is a sum of elements of the form $\rho(\alpha w_{b_{1}}\beta)$ and thus $\varphi^{-1}\rho^{-1}(u_{i})$ is a sum of elements of the form $\varphi^{-1}(\alpha)X_{b_{1}^{\ast}}(P)\varphi^{-1}(\beta)$, as was to be shown. This completes the proof of (ii). \\

We have that
\begin{center}
$\varphi(b_{1})=\displaystyle \sum_{b'r'}D_{b_{1},b'r'}b'r'$
\end{center}

then
\begin{center}
$\varphi(b_{1})^{\leq 2v+3}=\displaystyle \sum_{b'r'} (D_{b_{1},b'r'})^{\leq 2v+2}b'r'$
\end{center}

Also
\begin{align*}
\varphi(ra_{1})&=\displaystyle \sum_{r''a'} r''a'C_{r''a',ra_{1}} \\
z_{r_{a_{1}}}&=\displaystyle \sum_{r''a'} r''a'C_{r''a',ra_{1}}\varphi(Y_{[b_{1}ra_{1}]}(P))
\end{align*}

On the other hand, $(D_{b_{1},b'r'})^{\leq 2v+2}$ is a sum of elements of the form $m_{1}m_{2}\hdots m_{l-1}s''$ where each $m_{i} \in SM_{0}$ and $s'' \in L(\sigma(b'))$. Therefore $\varphi(b_{1})^{\leq 2v+3}$ is a sum of elements of the form $m_{1}m_{2}\hdots m_{l-1}s''b'r'$. \\

In what follows, $z(b_{1}ra_{1})=\varphi(Y_{[b_{1}ra_{1}]})$. \\

We obtain that $\varphi^{-1}(Y_{[bwa]}(\varphi(P)))n$ is a sum of elements of the form:
\begin{center}
$(\ast): \displaystyle \sum_{b',s',s'',r',r''} \varphi^{-1}\rho^{-1}\left((s'\rho(bwa))^{\ast}(\rho(\bar{e_{k}}H))\right)$
\end{center}
where:
\begin{align*}
H&=s''b'r'r''a'C_{r''a',ra_{1}}z(b_{1}ra_{1})m_{1}\hdots m_{l-1}s'n \\
&=\displaystyle \sum_{w_{1} \in L(k)} s''b'w_{1}^{\ast}(r'r'')w_{1}a'C_{r''a',ra_{1}}z(b_{1}ra_{1})m_{1}\hdots m_{l-1}s'n
\end{align*}

The non-zero terms of $(\ast)$ are those with $s'=s''$, $b=b'$, $a'=a$, $w_{1}=w$. Thus: 
\begin{center}
$\varphi^{-1}(Y_{[bwa]}(\varphi(P)))n=\displaystyle \sum_{h,r',r'',a_{1},b_{1}}w^{\ast}(r'r'')\varphi^{-1}(C_{r''a,ha_{1}})Y_{[b_{1}ha_{1}]}(P)\varphi^{-1}(D_{b_{1},br'})n$
\end{center}
Therefore: 
\begin{align*}
\hat{\gamma}_{sa,br}(n)&=\displaystyle \sum_{w \in L(k)} s^{\ast}(r^{-1}w)\varphi^{-1}(Y_{[bwa]}(\varphi(P)))n\\
&=\displaystyle \sum_{w,h,r',r'',a_{1},b_{1}}s^{\ast}(r^{-1}w)w^{\ast}(r'r'')\varphi^{-1}(C_{r''a,ha_{1}})Y_{[b_{1}ha_{1}]}(P)\varphi^{-1}(D_{b_{1},br'})n\\
&=\displaystyle \sum_{w,h,r',r'',a_{1},b_{1}}s^{\ast}(r^{-1}ww^{\ast}(r'r''))\varphi^{-1}(C_{r''a,ha_{1}})Y_{[b_{1}ha_{1}]}(P)\varphi^{-1}(D_{b_{1},br'})n\\
&=\displaystyle \sum_{h,r',r'',a_{1},b_{1}}s^{\ast}\left(r^{-1}\displaystyle \sum_{w \in L(k)}w^{\ast}(r'r'')w\right)\varphi^{-1}(C_{r''a,ha_{1}})Y_{[b_{1}ha_{1}]}(P)\varphi^{-1}(D_{b_{1},br'})n \\
&=\displaystyle \sum_{h,r',r'',a_{1},b_{1}}s^{\ast}(r^{-1}r'r'')\varphi^{-1}(C_{r''a,ha_{1}})Y_{[b_{1}ha_{1}]}(P)\varphi^{-1}(D_{b_{1},br'})n
\end{align*}

By \cite[Proposition 8.1 (iii)]{2} and \cite[Proposition 8.2 (iii)]{2} the following equalities hold: 
\begin{align*}
\displaystyle \varphi^{-1}(C_{r''a,ha_{1}})&=\displaystyle \sum_{u \in L(\sigma(a))} (r'')^{\ast}(hu)\varphi^{-1}(C_{ua,a_{1}})
\\
\varphi^{-1}(D_{b_{1},br'})&=\displaystyle \sum_{v \in L(\tau(b_{1}))}(r')^{\ast}(v)\varphi^{-1}(D_{b_{1},bv})
\end{align*}

whence
\begin{align*}
\hat{\gamma}_{sa,br}(n)&=\displaystyle \sum_{h,r',r'',a_{1},b_{1},u,v}s^{\ast}(r^{-1}r'r'')(r'')^{\ast}(hu)\varphi^{-1}(C_{ua,a_{1}})Y_{[b_{1}ha_{1}]}(P)(r')^{\ast}(v)\varphi^{-1}(D_{b_{1},bv})n\\
&=\displaystyle \sum_{h,r',a_{1},b_{1},u,v} s^{\ast}(r^{-1}r'hu)(r')^{\ast}(v)\varphi^{-1}(C_{ua,a_{1}})Y_{[b_{1}ha_{1}]}(P)\varphi^{-1}(D_{b_{1},bv})n\\
&=\displaystyle \sum_{h,a_{1},b_{1},u,v} s^{\ast}(r^{-1}vhu)\varphi^{-1}(C_{ua,a_{1}})Y_{[b_{1}ha_{1}]}(P)\varphi^{-1}(D_{b_{1},bv})n
\end{align*}

On the other hand, using again \cite[Proposition 8.1 (iii)]{2} and \cite[Proposition 8.2 (iii)]{2} one gets that the $(sa,br)$-entry of $(C\gamma D)(n)$ is given by
\begin{center}
$\displaystyle \sum_{s',t',h,a_{1},b_{1},u,v} (s')^{\ast}((t')^{-1}h)s^{\ast}(s'u)r^{\ast}(vt')\varphi^{-1}(C_{ua,a_{1}})Y_{[b_{1}ha_{1}]}(P)\varphi^{-1}(D_{b_{1},bv})n$
\end{center}

Using \ref{2.4} one obtains the following:

\begin{align*}
\displaystyle \sum_{s',t',h,u,v} (s')^{\ast}((t')^{-1}h)s^{\ast}(s'u)r^{\ast}(vt')&=\displaystyle \sum_{s',t',h,u,v} s^{\ast} \left( (s')^{\ast}((t')^{-1}h)s'u\right)r^{\ast}(vt') \\
&=\displaystyle \sum_{t',h,u,v}s^{\ast}\left((t')^{-1}hu\right)r^{\ast}(vt') \\
&=\displaystyle \sum_{h,u,v}s^{\ast} \left(\displaystyle \sum_{t'} r^{\ast}(vt')(t')^{-1}hu\right)\\
&=\displaystyle \sum_{h,u,v} s^{\ast}(r^{-1}vhu)
\end{align*}

which implies that $\hat{\gamma}=C\gamma D$ and the proof of Lemma \ref{lem4} is now complete.
\end{proof}

Note that Lemma \ref{lem4} implies the following equalities:

\begin{equation} \label{4.7}
\begin{split}
\operatorname{ker}(\gamma)&=D(\operatorname{ker}(\hat{\gamma})) \\
\operatorname{im}(\gamma)&=C^{-1}(\operatorname{im}(\hat{\gamma}))
\end{split}
\end{equation}

We now complete the proof of Proposition \ref{prop5}. Let us establish a right-equivalence $\left( \hat{\varphi}, \psi, \eta\right)$ between the representations $\widetilde{\mu_{k}}(\mathcal{N})$ and $\widetilde{\mu_{k}}\left(\widehat{\mathcal{N}}\right)$. First, we define $\widehat{\varphi}: \mathcal{F}_{S}(\mu_{k}M) \rightarrow \mathcal{F}_{S}(\mu_{k}M)$ as the right-equivalence between the algebras $(\mathcal{F}_{S}(\mu_{k}M),\mu_{k}P)$ and $\left(\mathcal{F}_{S}(\mu_{k}M), \mu_{k}\varphi(P)\right)$ given by \cite[Theorem 8.12]{2}. Let $\widetilde{\mu_{k}}(\widehat{\mathcal{N}})=(\overline{\widehat{N}},\overline{\widehat{V}})$. If $i \neq k$, then  $\widehat{N}_{i}=N_{i}$ and: \\

\begin{center}
$\overline{\widehat{N}}_{k}=\frac{\operatorname{ker}(\hat{\gamma})}{\operatorname{im}(\hat{\beta})} \oplus \operatorname{im}(\hat{\gamma}) \oplus \frac{\operatorname{ker}(\hat{\alpha})}{\operatorname{im}(\hat{\gamma})} \oplus V_{k}$
\end{center}

\vspace{0.1in}

For each $i \in \{1,2,3,4\}$, let $\overline{J}_{i}$ and $\overline{\Pi}_{i}$ be the corresponding inclusions and projections associated to $\overline{\widehat{N}}_{k}$, analogous to those given in \ref{4.1} and \ref{4.2}. Then we have inclusion maps:

\begin{align*}
&\overline{j}: \operatorname{ker}(\hat{\gamma}) \rightarrow N_{out} \\
&\overline{i}: \operatorname{im}(\hat{\gamma}) \rightarrow N_{in}
\end{align*}

and projections:

\begin{align*}
&\overline{\pi}_{1}: \operatorname{ker}(\hat{\gamma}) \rightarrow \frac{\operatorname{ker}(\hat{\gamma})}{\operatorname{im}(\hat{\beta})} \\
&\overline{\pi}_{2}: \operatorname{ker}(\hat{\alpha})\rightarrow \frac{\operatorname{ker}(\hat{\alpha})}{\operatorname{im}(\hat{\gamma})}
\end{align*}

By Lemma \ref{lem4} we have $\hat{\gamma}=C\gamma D$ and thus $\hat{\gamma}D^{-1}=C\gamma$. It follows that $D^{-1}$ induces an isomorphism
\begin{center}
$D_{1}^{-1}: \operatorname{ker}(\gamma) \rightarrow \operatorname{ker}(\hat{\gamma})$
\end{center}

such that $\overline{j}D_{1}^{-1}=D^{-1}j$. Also, $D^{-1}$ maps $\operatorname{im}(\beta)$ to $\operatorname{im}(\hat{\beta})$. Therefore, $D^{-1}$ also induces an isomorphism
\begin{center}
$\underline{D}^{-1}:  \frac{\operatorname{ker}(\gamma)}{\operatorname{im}(\beta)} \rightarrow \frac{\operatorname{ker}(\hat{\gamma})}{\operatorname{im}(\hat{\beta})}$
\end{center}

\vspace{0.1in}

such that $\underline{D}^{-1}\pi_{1}=\overline{\pi}_{1}D_{1}^{-1}$. The isomorphism $C$ induces an isomorphism
\begin{center}
$C_{1}: \operatorname{im}(\gamma) \rightarrow \operatorname{im}(\hat{\gamma})$
\end{center}

such that $\overline{i}C_{1}=Ci$. The equality $\hat{\alpha}C=\alpha$ implies that $C$ also induces an isomorphism $C_{2}: \operatorname{ker}(\alpha) \rightarrow \operatorname{ker}(\hat{\alpha})$; thus there exists an isomorphism
\begin{center}
$\underline{C}: \frac{\operatorname{ker}(\alpha)}{\operatorname{im}(\gamma)} \rightarrow \frac{ \operatorname{ker}(\hat{\alpha})}{\operatorname{im}(\hat{\gamma})}$
\end{center}

such that $\underline{C}\pi_{2}=\overline{\pi}_{2}C_{2}$. \\

To construct $\widetilde{\mu_{k}}(\widehat{N})$ we choose splitting data as follows:

\begin{align*}
\overline{p}&=D_{1}^{-1}pD: N_{out} \rightarrow \operatorname{ker}(\hat{\gamma}) \\
\overline{\sigma}_{2}&=C_{2}\sigma_{2}\underline{C}^{-1}: \frac{\operatorname{ker}(\hat{\alpha})}{\operatorname{im}(\hat{\gamma})} \rightarrow \operatorname{ker}(\hat{\alpha})
\end{align*}

Note that $\overline{p}\overline{j}=id_{\operatorname{ker}(\hat{\gamma})}$, $\overline{\pi}_{2}\overline{\sigma}_{2}=id_{\operatorname{ker}(\hat{\alpha})/\operatorname{im}(\hat{\gamma})}$. Define: \\

\begin{center}
$\psi: \overline{N} \rightarrow \overline{\widehat{N}}$
\end{center}

as follows. If $i \neq k$ then $\psi_{i}: \overline{N}_{i}=N_{i} \rightarrow \overline{\widehat{N}}_{i}=N_{i}$ is the identity map and
\begin{center}
$\psi_{k}: \overline{N}_{k} \rightarrow \overline{\widehat{N}}_{k}$
\end{center}

\vspace{0.1in}

is the map such that for every $i \neq j$, $\overline{\Pi}_{i} \psi_{k}J_{j}=0$ and
\begin{align*}
\overline{\Pi}_{1}\psi_{k}J_{1}&=\underline{D}^{-1} \\
\overline{\Pi}_{2}\psi_{k}J_{2}&=C_{1} \\
\overline{\Pi}_{3}\psi_{k}J_{3}&=\underline{C} \\
\overline{\Pi}_{4}\psi_{k}J_{4}&=id_{V_{k}}
\end{align*}

Let us show that for every $z \in \widetilde{T}$ and $n \in N_{\tau(z)}$: 
\begin{center}
$\psi_{\sigma(z)}(zn)=\widehat{\varphi}(z)\psi_{\tau(z)}(n)$
\end{center}

\vspace{0.1in}

Suppose first that $z=a^{\ast}$ where $a \in _{k}T$. In this case $\tau(z)=\sigma(a)=k$ and $\sigma(z)=\tau(a) \neq k$. By \cite[Proposition 8.11]{2}: 
\begin{center}
$\widehat{\varphi}(a^{\ast})=\displaystyle \sum_{t \in L(k), a_{1} \in _{k}T} (C^{-1})_{a,ta_{1}} a_{1}^{\ast}t^{-1}$
\end{center}

whence
\begin{center}
$\overline{\widehat{N}}(\widehat{\varphi}(a^{\ast}))=\displaystyle \sum_{t \in L(k), a_{1} \in _{k}T} (C^{-1})_{a,ta_{1}} \ast \overline{\widehat{N}}(a_{1}^{\ast})t^{-1}$
\end{center}
where $\ast$ denotes the action of $\mathcal{F}_{S}(M)$ in $\widehat{N}$. In this case we have to verify the following equality:

\begin{equation} \label{4.8}
\overline{N}(a^{\ast})=\overline{\widehat{N}}(\widehat{\varphi}(a^{\ast}))\psi_{k}
\end{equation}
On one hand, $\overline{N}(a^{\ast})J_{1}=0$. On the other hand: 
\begin{align*}
\overline{\widehat{N}}(\widehat{\varphi}(a^{\ast}))\psi_{k}J_{1}&= \displaystyle \sum_{a_{1} \in _{k}T, t \in L(k)} (C^{-1})_{a,ta_{1}} \ast \overline{\widehat{N}}(a_{1}^{\ast})t^{-1}\psi_{k}J_{1} \\
&=\displaystyle \sum_{a_{1} \in _{k}T, t \in L(k)} (C^{-1})_{a,ta_{1}} \ast \overline{ \widehat{N}}(a_{1}^{\ast})t^{-1} \overline{J}_{1} \overline{\Pi}_{1} \psi_{k} J_{1} \\
&=\displaystyle \sum_{a_{1} \in _{k}T, t \in L(k)} (C^{-1})_{a,ta_{1}} \ast \overline{ \widehat{N}}(a_{1}^{\ast})\overline{J}_{1} t^{-1} \overline{\Pi}_{1} \psi_{k} J_{1} =0
\end{align*}

and thus $\overline{N}(a^{\ast})J_{1}=\overline{\widehat{N}}(\widehat{\varphi}(a^{\ast}))\psi_{k}J_{1}$. Now let us consider $\overline{N}(a^{\ast})J_{2}$. By \ref{4.3} we have $\overline{N}(a^{\ast})J_{2}=c_{k}^{-1}\pi_{e_{k}a}i$. On the other hand: 
\begin{align*}
\overline{\widehat{N}}(\widehat{\varphi}(a^{\ast}))\psi_{k}J_{2}&= \displaystyle \sum_{a_{1} \in _{k}T, t \in L(k)} (C^{-1})_{a,ta_{1}} \ast \overline{\widehat{N}}(a_{1}^{\ast})t^{-1}\overline{J}_{2} \overline{\Pi}_{2} \psi_{k} J_{2} \\
&= \displaystyle \sum_{a_{1} \in _{k}T, t \in L(k)} (C^{-1})_{a,ta_{1}} \ast \overline{ \widehat{N}}(a_{1}^{\ast}) \overline{J}_{2} t^{-1} \overline{\Pi}_{2} \psi_{k}J_{2} \\
&= \displaystyle \sum_{a_{1} \in _{k}T, t \in L(k)} (C^{-1})_{a,ta_{1}} \ast c_{k}^{-1} \pi_{e_{k}a_{1}} \overline{i} t^{-1} \overline{\Pi}_{2} \psi_{k} J_{2} \\
&=\displaystyle \sum_{a_{1} \in _{k}T, t \in L(k)} (C^{-1})_{a,ta_{1}} \ast c_{k}^{-1} \pi_{e_{k}a_{1}} \overline{i} t^{-1} C_{1} \\
&=c_{k}^{-1} \displaystyle \sum_{a_{1} \in _{k}T, t \in L(k)} (C^{-1})_{a,ta_{1}} \ast \pi_{e_{k}a_{1}} t^{-1} Ci \\
&=c_{k}^{-1} \displaystyle \sum_{a_{1} \in _{k}T, t \in L(k)} (C^{-1})_{a,ta_{1}} \ast \pi_{ta_{1}} Ci \\
&=c_{k}^{-1} \displaystyle \sum_{a_{1},a_{2} \in _{k}T, r,t \in L(k)} \varphi^{-1}((C^{-1})_{a,ta_{1}})\pi_{ta_{1}}C\xi_{ra_{2}}\pi_{ra_{2}}i \\
&=c_{k}^{-1} \displaystyle \sum_{a_{1},a_{2} \in _{k}T, r,t \in L(k)} \varphi^{-1}((C^{-1})_{a,ta_{1}}) \varphi^{-1}(C_{ta_{1},ra_{2}})\pi_{ra_{2}}i \\
&=c_{k}^{-1} \displaystyle \sum_{a_{1}, a_{2} \in _{k}T, r,t \in L(k)} \varphi^{-1} \left( (C^{-1})_{a,ta_{1}}C_{ta_{1},ra_{2}}\right) \pi_{ra_{2}}i \\
&=c_{k}^{-1} \pi_{e_{k}a}i
\end{align*}

Therefore $\overline{N}(a^{\ast})J_{2}=\overline{\widehat{N}}(\widehat{\varphi}(a^{\ast}))\psi_{k}J_{2}$. For $J_{3}$ we have $\overline{N}(a^{\ast})J_{3}=c_{k}^{-1}\pi_{e_{k}a}j' \sigma_{2}$ and
\begin{align*}
\overline{\widehat{N}}(\widehat{\varphi}(a^{\ast}))\psi_{k}J_{3}&= \displaystyle \sum_{a_{1} \in _{k}T, t \in L(k)} (C^{-1})_{a,ta_{1}} \ast \overline{\widehat{N}}(a_{1}^{\ast}) \overline{J}_{3}t^{-1} \overline{\Pi}_{3} \psi_{k} J_{3} \\
&=c_{k}^{-1} \displaystyle \sum_{a_{1} \in _{k}T, t \in L(k)} \varphi^{-1} \left( (C^{-1})_{a,ta_{1}}\right) \pi_{ta_{1}} \overline{j} \overline{\sigma}_{2}\underline{C} \\
&=c_{k}^{-1} \displaystyle \sum_{a_{1} \in _{k}T, t \in L(k)} \varphi^{-1} \left( (C^{-1})_{a,ta_{1}}\right) \pi_{ta_{1}} Cj' \sigma_{2} \\
&=c_{k}^{-1} \displaystyle \sum_{a_{1}, a_{2} \in _{k}T, t,r \in L(k)} \varphi^{-1} \left( (C^{-1})_{a,ta_{1}}\right)\pi_{ta_{1}}C \xi_{ra_{2}}\pi_{ra_{2}}j ' \sigma_{2} \\
&=c_{k}^{-1} \displaystyle \sum_{a_{1},a_{2} \in _{k}T, t,r \in L(k)} \varphi^{-1} \left((C^{-1})_{a,ta_{1}}\right) \varphi^{-1}(C_{ta_{1},ra_{2}})\pi_{ra_{2}}j' \sigma_{2} \\
&=c_{k}^{-1}\pi_{e_{k}a}j'\sigma_{2}
\end{align*}
Thus $\overline{N}(a^{\ast})J_{3}=\overline{\widehat{N}}(\widehat{\varphi}(a^{\ast}))\psi_{k}J_{3}$. \\

Finally, $\overline{N}(a^{\ast})J_{4}=0$ and: 

\begin{center}
$\overline{\widehat{N}}(\varphi(a^{\ast}))\psi_{k}J_{4}=\displaystyle \sum_{a_{1} \in _{k}T, t \in L(k)} (C^{-1})_{a,ta_{1}} \ast \overline{\widehat{N}}(a_{1}^{\ast})\overline{J}_{4}t^{-1} \overline{\Pi}_{4}\psi_{k}J_{4}=0$.
\end{center}

and \ref{4.8} holds. \\

Suppose now that $z=^{\ast}b$ where $b \in T_{k}$. In this case $\tau(z)=\sigma(b) \neq k$ and $\sigma(z)=\tau(b)=k$. By \cite[Proposition 8.11]{2}: 

\begin{center}
$\widehat{\varphi}(^{\ast}b)=\displaystyle \sum_{r \in L(k), b_{1} \in T_{k}} r^{-1}(^{\ast}b_{1})(D^{-1})_{b_{1}r,b}$
\end{center}

We have to show that
\begin{equation} \label{4.9}
\psi_{k}\overline{N}(^{\ast}b)=\overline{\widehat{N}}(\widehat{\varphi}(^{\ast}b))
\end{equation}

On one hand, $\overline{\Pi}_{1} \psi_{k} \overline{N}(^{\ast}b)=\overline{\Pi}_{1} \psi_{k}J_{1}\Pi_{1} \overline{N}(^{\ast}b)=- \underline{D}^{-1} \pi_{1}p \xi_{be_{k}}$. On the other hand:

\begin{align*}
\overline{\Pi}_{1} \overline{\widehat{N}}(\widehat{\varphi}(^{\ast}b))&= \displaystyle \sum_{b_{1} \in T_{k}, r \in L(k)} r^{-1} \overline{\Pi}_{1} \overline{\widehat{N}}(^{\ast}b_{1})\varphi^{-1}((D^{-1})_{b_{1}r,b}) \\
&=-\displaystyle \sum_{r \in L(k), b_{1} \in T_{k}} r^{-1} \overline{\pi}_{1} \overline{p} \xi_{b_{1}e_{k}} \varphi^{-1}((D^{-1})_{b_{1}r,b}) \\
&=-\displaystyle \sum_{r \in L(k), b_{1} \in T_{k}} \overline{\pi}_{1} \overline{p} \xi_{b_{1}r} \pi_{b_{1}r}D^{-1} \xi_{be_{k}} \\
&=-\overline{\pi}_{1} \overline{p} D^{-1} \xi_{be_{k}}  \\
&=-\overline{\pi}_{1} D_{1}^{-1}p D D^{-1} \xi_{be_{k}} \\
&=-\overline{\pi}_{1}D_{1}^{-1}p \xi_{be_{k}} \\
&= -\underline{D}^{-1} \pi_{1}p \xi_{be_{k}}
\end{align*}

Consequently
\begin{center}
$\overline{\Pi}_{1} \psi_{k} \overline{N}(^{\ast}b)=\overline{\Pi}_{1} \overline{\widehat{N}}(\widehat{\varphi}(^{\ast}b))$
\end{center}

\vspace{0.1in}

Now consider the map
\begin{center}
$\hat{\gamma'}: N_{out} \rightarrow \operatorname{im}(\widehat{\gamma})$
\end{center}

where $\hat{\gamma}=\overline{i}\widehat{\gamma'}$. Let us consider $\overline{\Pi}_{2}$. We have: \\

\begin{center}
$\overline{\Pi}_{2}\psi_{k}\overline{N}(^{\ast}b)=\overline{\Pi}_{2} \psi_{k} J_{2} \Pi_{2} \overline{N}(^{\ast}b)=-C_{1}\gamma' \xi_{be_{k}}$
\end{center}

and
\begin{align*}
\overline{\Pi}_{2} \overline{\widehat{N}}(\widehat{\varphi}(^{\ast}b))&=\displaystyle \sum_{r \in L(k), b_{1} \in T_{k}} r^{-1} \overline{\Pi}_{2} \overline{\widehat{N}}(^{\ast}b_{1})\varphi^{-1}((D^{-1})_{b_{1}r,b}) \\
&=-\displaystyle \sum_{r \in L(k), b_{1} \in T_{k}} r^{-1} \widehat{\gamma'} \xi_{b_{1}e_{k}} \varphi^{-1}((D^{-1})_{b_{1}r,b}) \\
&=-\displaystyle \sum_{r \in L(k), b_{1} \in T_{k}} \widehat{\gamma'} r^{-1} \xi_{b_{1}e_{k}}\varphi^{-1}((D^{-1})_{b_{1}r,b}) \\
&=-\displaystyle \sum_{r \in L(k), b_{1} \in T_{k}} \widehat{\gamma'} \xi_{b_{1}r} \pi_{b_{1}r}D^{-1} \xi_{be_{k}} \\
&=-\widehat{\gamma'}D^{-1}\xi_{be_{k}}
\end{align*}

Also:

\begin{center}
$\overline{i}C_{1}\gamma' \xi_{be_{k}}=Ci\gamma' \xi_{be_{k}}=C\gamma \xi_{be_{k}} = \hat{\gamma}D^{-1}\xi_{be_{k}}=\overline{i}\widehat{\gamma'}D^{-1}\xi_{be_{k}}$
\end{center}

Since $\overline{i}$ is injective, then $C_{1}\gamma' \xi_{be_{k}}=\widehat{\gamma'}D^{-1}\xi_{be_{k}}$. Therefore

\begin{center}
$\overline{\Pi}_{2}\psi_{k}\overline{N}(^{\ast}b)=\overline{\Pi}_{2} \overline{\widehat{N}}(\widehat{\varphi}(^{\ast}b))$
\end{center}

Finally:

\begin{align*}
\overline{\Pi}_{3} \psi_{k} \overline{N}(^{\ast}b)&=\overline{\Pi}_{3} \psi_{k}J_{3}\Pi_{3}\overline{N}(^{\ast}b)=0 \\
\overline{\Pi}_{3} \overline{\widehat{N}}(\widehat{\varphi}(^{\ast}b))&=\displaystyle \sum_{r \in L(k), b_{1} \in T_{k}} r^{-1}\overline{\Pi}_{3} \overline{ \widehat{N}}(^{\ast}b_{1})\varphi^{-1}((D^{-1})_{b_{1}r,b})=0 \\
\overline{\Pi}_{4}\psi_{k} \overline{N}(^{\ast}b)&=\overline{\Pi}_{4} \psi_{k}J_{4}\Pi_{4}\overline{N}(^{\ast}b)=0 \\
\overline{\Pi}_{4} \overline{ \widehat{N}}(\widehat{\varphi}(^{\ast}b))&=\displaystyle \sum_{r \in L(k), b_{1} \in T_{k}} r^{-1} \overline{\Pi}_{4} \overline {\widehat{N}}(^{\ast}b_{1})\varphi^{-1}((D^{-1})_{b_{1}r,b})=0
\end{align*}

and the proof of Proposition \ref{prop5} is now complete.
\end{proof}

\begin{theorem} \label{teo1} Let $\varphi: \mathcal{F}_{S}(M_{1}) \rightarrow \mathcal{F}_{S}(M)$ be an algebra isomorphism with $\varphi_{|S}=id_{S}$ and let $\mathcal{N}=(N,V)$ be a decorated representation of $(\mathcal{F}_{S}(M_{1}),P)$ where $P$ is a potential such that $e_{k}Pe_{k}=0$. Then there exists an algebra isomorphism:

\begin{center}
$\hat{\varphi}: \mathcal{F}_{S}(\mu_{k}M_{1}) \rightarrow \mathcal{F}_{S}(\mu_{k}M)$
\end{center}

satisfying the following conditions:

\begin{enumerate}[(a)]
\item $\hat{\varphi}(\mu_{k}P)$ is cyclically equivalent to $\mu_{k}(\varphi(P))$.
\item There exists an isomorphism of decorated representations $\psi: \widetilde{\mu_{k}}(\mathcal{N}) \rightarrow \widetilde{\mu_{k}}(\widehat{\mathcal{N}})$.
\end{enumerate} 
\end{theorem}

\begin{proof} Part (a) follows by \cite[Theorem 8.14]{2}. Note that $\varphi$ is equal to the composition of an algebra isomorphism $\mathcal{F}_{S}(M_{1}) \rightarrow \mathcal{F}_{S}(M)$, induced by an isomorphism of $S$-bimodules $M_{1} \rightarrow M$, with a unitriangular automorphism of $\mathcal{F}_{S}(M)$. In view of Proposition \ref{prop5}, it suffices to prove the statement when $\varphi$ is induced by an isomorphism of $S$-bimodules $\phi: M_{1} \rightarrow M$. Let $T_{1}$ be a $Z$-free generating set of $M_{1}$ and let $T$ be a $Z$-free generating set of $M$. Associated to the representation $\mathcal{N}$ we have the following left $D_{k}$-spaces:

\begin{align*}
N_{in}^{1}&=\displaystyle \bigoplus_{a \in _{k}T_{1}} D_{k} \otimes_{F} N_{\tau(a)} \\
N_{out}^{1}&= \displaystyle \bigoplus_{b \in (T_{1})_{k}} D_{k} \otimes_{F} N_{\sigma(b)}
\end{align*}

and associated to the representation $\widehat{\mathcal{N}}$ we also have the following left $D_{k}$-spaces:

\begin{align*}
N_{in} &= \displaystyle \bigoplus_{a \in _{k}T} D_{k} \otimes_{F} N_{\tau(a)} \\
N_{out}&=\displaystyle \bigoplus_{b \in T_{k}} D_{k} \otimes_{F} N_{\sigma(b)}
\end{align*}

Let $\gamma_{1}: N_{out}^{1} \rightarrow N_{in}^{1}$, $\alpha_{1}: N_{in}^{1} \rightarrow N_{k}$, $\beta_{1}: N_{k} \rightarrow N_{out}^{1}$ be the maps associated to $\mathcal{N}$ and $\gamma: N_{out} \rightarrow N_{in}$, $\alpha: N_{in} \rightarrow N_{k}$ and $\beta: N_{k} \rightarrow N_{out}$ be the maps associated to $\widehat{\mathcal{N}}$. \\

For $a_{1} \in _{k}T_{1}, r \in L(k)$ we have:

\begin{center}
$\phi(ra_{1})=\displaystyle \sum_{r' \in L(k), a \in _{k}T} r'aC_{r'a,ra_{1}}$
\end{center}

For $b_{1} \in (T_{1})_{k}, t \in L(k)$ we have:

\begin{center}
$\phi(b_{1}t)=\displaystyle \sum_{b \in T_{k}, t' \in L(k)} D_{b_{1}t,bt'}bt'$
\end{center}

where $C_{r'a,ra_{1}}, D_{b_{1}t,bt'} \in S$. Define $C: N_{in}^{1} \rightarrow N_{in}$ as the linear map such that for all $r,r' \in L(k)$, $a,a_{1} \in _{k}T_{1}$, the map:

\begin{center}
$\pi_{r'a}C\xi_{ra_{1}}: N_{\tau(a_{1})} \rightarrow N_{\tau(a)}$
\end{center}

is given by $\pi_{r'a}C\xi_{ra_{1}}(n)=C_{r'a,ra_{1}}n$. Similarly, we define $D: N_{out} \rightarrow N_{out}^{1}$ as the linear map such that for all $t,t' \in L(k)$, $b,b_{1} \in T_{k}$, the map:

\begin{center}
$\pi_{b_{1}t}D\xi_{bt'}: N_{\sigma(b)} \rightarrow N_{\sigma(b_{1})}$
\end{center}

is given by $\pi_{b_{1}t}D\xi_{bt'}(n)=D_{b_{1}t,bt'}n$. \\

Let us show that $\alpha C= \alpha_{1}$, $D\beta=\beta_{1}$. If $n \in N_{\tau(a_{1})}$, then \\
\begin{align*}
\alpha C \xi_{ra_{1}}(n)&=\displaystyle \sum_{r'a' \in _{k}\hat{T}_{1}} \alpha \xi_{r'a'}\pi_{r'a'}C\xi_{ra_{1}}(n) \\
&=\displaystyle \sum_{r'a' \in _{k}\hat{T}_{1}} \phi^{-1}\left(r'a'C_{r'a',ra_{1}}\right)n = \phi^{-1}\phi(ra_{1})n \\
&=ra_{1}n=\alpha_{1}\xi_{ra_{1}}(n)
\end{align*}

whence $\alpha C=\alpha_{1}$. If $n \in N_{\sigma(b)}$, $b_{1} \in (T_{1})_{k}$ and $t \in L(k)$, then \\
\begin{align*}
\pi_{b_{1}t}D\beta(n)&=\displaystyle \sum_{bt' \in \hat{T}_{k}} \pi_{b_{1}t}D\xi_{bt'}\pi_{bt'}\beta(n) \\
&=\displaystyle \sum_{bt' \in \hat{T}_{k}} D_{b_{1}t,bt'}(\phi^{-1}(bt')n) \\
&=\phi^{-1}\left( \displaystyle \sum_{bt' \in \hat{T}_{k}} D_{b_{1}t,bt'}bt'n\right) \\
&=\phi^{-1}\phi(b_{1}t)n \\
&=b_{1}tn=\pi_{b_{1}t}\beta_{1}(n)
\end{align*}

hence $D\beta=\beta_{1}$. \\

We now show that $\hat{\gamma}=C\gamma_{1} D$. Let $br \in \hat{T}_{k}$, $sa \in _{k}\hat{T}$, $n \in N_{\sigma(b)}$. We have:

\begin{center}
$\hat{\gamma}_{sa,br}(n)=\displaystyle \sum_{w \in L(k)} s^{\ast}(r^{-1}w)\varphi^{-1}(Y_{[bwa]}(\varphi(P)))n$ \\
\end{center}

Also:
\begin{center}
$\varphi^{-1}\left(Y_{[bwa]}(\varphi(P))\right)=\varphi^{-1} \rho^{-1} \left(X_{[bwa]}(\rho\varphi(P))\right)=\rho^{-1}\hat{\varphi}^{-1}\left(X_{[bwa]}(\hat{\varphi}\rho(P))\right)$
\end{center}

Let
\begin{center}
$T_{1}^{\mu}=\{\rho(c_{1}): c_{1} \in T_{1} \cap \bar{e_{k}}M_{1}\bar{e_{k}}\} \cup \{a_{1}^{\ast}: a_{1} \in _{k}T_{1}\} \cup \{^{\ast}b_{1}: b_{1} \in (T_{1})_{k} \} \cup \{\rho(b_{1}ra_{1}): b_{1} \in (T_{1})_{k}, a_{1} \in _{k}T_{1}, r \in L(k)\}$
\end{center}

denote the $Z$-free generating set of $\mu_{k}M_{1}$ associated to $T_{1}$. Likewise, let $T^{\mu}$ denote the $Z$-free generating set of $\mu_{k}M$ associated to $T$. Then:

\begin{center}
$X_{[bwa]}(\hat{\varphi}(\rho(P)))=\displaystyle \sum_{s \in L(\sigma(b))} \displaystyle \sum_{c \in (\hat{T}_{1})^{\mu}_{\hat{k},\hat{k}}} (s\rho(bwa))^{\ast}(\hat{\Delta}(\hat{\varphi}(c))\Diamond z(c))s$
\end{center}

where $z(c)=\hat{\varphi}(X_{c}(\rho(P)))$. \\

The elements of the set $(T_{1})^{\mu}_{\hat{k},\hat{k}}$ are either elements of the form $\rho(c_{1})$ where $c_{1} \in T_{1} \cap \bar{e_{k}}M\bar{e_{k}}$ or elements of the form $\rho(b_{1}ra_{1})$ where $b_{1} \in (T_{1})_{k},a_{1} \in _{k}T_{1}, r \in L(k)$. We now compute $W(c_{1})$ when $c_{1} \in T_{1} \cap \bar{e_{k}}M\bar{e_{k}}$. 

\begin{center}
$W(c_{1})=\displaystyle \sum_{s \in L(\sigma(b))} (s\rho(bwa))^{\ast}( \hat{\Delta}(\hat{\varphi}\rho(c_{1})) \Diamond z(\rho(c_{1}))s$
\end{center}

Then $\hat{\varphi}\rho(c_{1})=\rho(\varphi(c_{1}))=\rho(\phi(c_{1}))=\displaystyle \sum_{i} r_{i}\rho(c_{i})t_{i}$ where $r_{i},t_{i} \in S$, $c_{i} \in T \cap \bar{e_{k}}M\bar{e_{k}}$.

Using \cite[Lemma 5.2]{2} one gets:
\begin{align*}
W(c_{1})&=\displaystyle \sum_{i} \displaystyle \sum_{s \in L(\sigma(b))} (s\rho(bwa))^{\ast}\left(\hat{\Delta}(\rho(c_{i})) \Diamond t_{i}z(\rho(c_{1}))r_{i}\right)s \\
&=\displaystyle \sum_{i} \displaystyle \sum_{s \in L(\sigma(b))} (s\rho(bwa))^{\ast}\left(\rho(c_{i})t_{i}z(\rho(c_{1}))r_{i}\right)s=0.
\end{align*}

Therefore
\begin{center}
$X_{[bwa]}(\hat{\varphi}(\rho(P)))=\displaystyle \sum_{s \in L(\sigma(b))} \displaystyle \sum_{b_{1}ha_{1}} (s\rho(bwa))^{\ast}(\hat{\Delta}(\hat{\varphi}\rho(b_{1}ha_{1})) \Diamond z(\rho(b_{1}ha_{1})))s$
\end{center}

and \\

\begin{center}
$\hat{\varphi}(\rho(b_{1}ha_{1}))=\rho\varphi(b_{1}ha_{1})=\rho(\phi(b_{1})\phi(ha_{1}))=\displaystyle \sum_{b',r',r'',a'',w'} (w')^{\ast}(r'r'')D_{b_{1},b'r'}\rho(b'w'a'')C_{r''a'',ha_{1}}$
\end{center}

Once again, \cite[Lemma 5.2]{2} yields: \\

\begin{center}
$X_{[bwa]}(\hat{\varphi}(\rho(P)))=\displaystyle \sum_{s \in L(\sigma(b)),b_{1}ha_{1},b',r',r'',a'',w'} (s\rho(bwa))^{\ast}\left((w')^{\ast}(r'r'')\rho(b'w'a'')C_{r''a'',ha_{1}}z(\rho(b_{1}ha_{1}))D_{b_{1},b'r'}\right)s$
\end{center}

Therefore:

\begin{center}
$X_{[bwa]}(\hat{\varphi}(\rho(P)))=\displaystyle \sum_{b_{1}ha_{1},r',r''} w^{\ast}(r'r'')C_{r''a,ha_{1}}z(\rho(b_{1}ha_{1}))D_{b_{1},br'}$
\end{center}

Note that
\begin{center}
$z(\rho(b_{1}ha_{1}))=\hat{\varphi}(X_{[b_{1}ha_{1}]}(\rho(P)))=\hat{\varphi}\rho(Y_{[b_{1}ha_{1}]}(P))=\rho\varphi(Y_{[b_{1}ha_{1}]}(P))$
\end{center}

This yields the equality:
\begin{center}
$X_{[bwa]}(\hat{\varphi}(\rho(P)))=\rho \varphi \left( \displaystyle \sum_{b_{1}ha_{1},r',r''} w^{\ast}(r'r'')C_{r''a,ha_{1}}Y_{[b_{1}ha_{1}]}(P)D_{b_{1},br'}\right)$
\end{center}

Consequently:

\begin{align*}
\hat{\gamma}_{sa,br}(n)&=\displaystyle \sum_{w \in L(k)} s^{\ast}(r^{-1}w) \displaystyle \sum_{b_{1}ha_{1},r',r''} w^{\ast}(r'r'')C_{r''a,ha_{1}}Y_{[b_{1}ha_{1}]}(P)D_{b_{1},br'} n\\
&=\displaystyle \sum_{b_{1}ha_{1},r',r''} s^{\ast}(r^{-1}r'r'')C_{r''a,ha_{1}}Y_{[b_{1}ha_{1}]}(P)D_{b_{1},br'}n
\end{align*}

A similar argument to the proof of Lemma \ref{lem4} yields that $\hat{\gamma}_{sa,br}(n)$ is equal to $(C\gamma_{1}D)_{sa,br}(n)$. The desired right-equivalence is analogous to the one constructed in Proposition \ref{prop5}.
\end{proof}
\begin{prop} \label{prop6} The right-equivalence class of the representation $\widetilde{\mu_{k}}(\mathcal{N})$ is determined by the right-equivalence class of the representation $\mathcal{N}$.
\end{prop}

\begin{proof}

Let $\mathcal{N}=(N,V)$ be a decorated representation of $(\mathcal{F}_{S}(M),P)$ and let $\mathcal{N}'=(N',V')$ be a decorated representation of $(\mathcal{F}_{S}(M'),P')$. Suppose that such representations are right-equivalent, then there exists an algebra isomorphism $\varphi: \mathcal{F}_{S}(M) \rightarrow \mathcal{F}_{S}(M')$, with $\varphi_{|S}=id_{S}$, such that $\varphi(P)$ is cyclically equivalent to $P'$. Recall that associated to the representation $\mathcal{N}=(N,V)$ we have a decorated representation $\widehat{\mathcal{N}}=(\widehat{N},V)$ of $(\mathcal{F}_{S}(M'),\varphi(P))$ where $\widehat{N}=N$ as $F$-vector spaces, and for $u \in \mathcal{F}_{S}(M')$ and $n \in N$ we let $u \ast n:=\varphi^{-1}(u)n$. Let $\psi: \widehat{N} \rightarrow N'$ be the isomorphism of $F$-vector spaces induced by the right-equivalence between $\mathcal{N}$ and $\mathcal{N'}$. Note that $\psi$ is an isomorphism of left $\mathcal{F}_{S}(M')$-modules because if $u \in \mathcal{F}_{S}(M')$ and $n \in \widehat{N}$ then $\psi(un)=\psi(\varphi^{-1}(u)n)=\varphi(\varphi^{-1}(u))\psi(n)=u\psi(n)$. Now, let us show that:

\begin{enumerate}[(a)]
\item $\widetilde{\mu_{k}}(\mathcal{N})$ is right-equivalent to $\widetilde{\mu_{k}}(\mathcal{\widehat{N}})$
\item $\widetilde{\mu_{k}}(\mathcal{\widehat{N}})$ is right-equivalent to $\widetilde{\mu_{k}}(\mathcal{N'})$
\end{enumerate}

\vspace{0.2in}

Note that $(a)$ follows immediately by Theorem \ref{teo1}. We now prove $(b)$. Consider the cyclically equivalent potentials $\varphi(P)$ and $P'$ in $\mathcal{F}_{S}(M')$. By \cite[Proposition 8.15]{2} we have that $\mu_{k}(\varphi(P))$ is cyclically equivalent to $\mu_{k}P'$, in particular $\mu_{k}(\varphi(P))$ is right-equivalent to $\mu_{k}P'$ via the identity map $id_{\mathcal{F}_{S}(\mu_{k}M')}$. \\

Note that $\psi$ induces isomorphisms of $F$-vector spaces: 
\begin{align*}
& \hat{\psi}_{i}: N_{in} \rightarrow N_{in}' \\
& \hat{\psi}_{o}: N_{out} \rightarrow N_{out}'
\end{align*}

Let $\rho=\psi^{-1}: N' \rightarrow \widehat{N}$, then $\rho$ also induces isomorphisms of $F$-vector spaces:
\begin{align*}
& \rho_{i}: N_{in}' \rightarrow N_{in} \\
& \rho_{o}: N_{out}' \rightarrow N_{out}
\end{align*}

Let $\psi_{k}: N_{k} \rightarrow N_{k}'$ and $\rho_{k}: N_{k}' \rightarrow N_{k}$ be the maps induced by $\psi$ and $\rho$. Then we have the following equalities:
\begin{equation} \label{4.10}
\begin{split}
\psi_{k}\alpha&=\alpha' \hat{\psi}_{i}\\
\beta' \psi_{k}&=\hat{\psi}_{o} \beta\\
\hat{\psi}_{i} \gamma &= \gamma' \hat{\psi}_{o}\\
\rho_{k}\alpha'&=\alpha\rho_{i}\\
\beta\rho_{k} &= \rho_{o} \beta'\\
\rho_{i} \gamma' &= \gamma \rho_{o}
\end{split}
\end{equation}

Observe that $\hat{\psi}_{i}$ induces a map $ker(\alpha) \rightarrow ker(\alpha')$ and $\rho_{i}$ induces a map $ker(\alpha') \rightarrow ker(\alpha)$ such that $\rho_{i}=(\hat{\psi}_{i})^{-1}$. Hence $\hat{\psi}_{i}$ induces an isomorphism between $ker(\alpha)$ and $ker(\alpha')$. Similarly,  $\hat{\psi}_{i}$ induces an isomorphism $im(\gamma) \rightarrow im(\gamma')$, $\hat{\psi}_{o}$ induces an isomorphism $ker(\gamma) \rightarrow ker(\gamma')$ and $\hat{\psi}_{o}$ induces an isomorphism $im(\beta) \rightarrow im(\beta')$.

To construct $\overline{N'}$ we choose the splitting data $(p',\sigma')$ in terms of the splitting data $(p,\sigma)$ of $\overline{N}$ as follows: 

\begin{equation} \label{4.11}
\begin{split}
p'&=\psi_{1}p(\hat{\psi}_{o})^{-1} \\
\sigma'&=\psi_{2}\sigma(\overline{\psi}_{2})^{-1}
\end{split}
\end{equation}

where $\psi_{1}$ is the restriction of $\hat{\psi}_{o}$ to $\operatorname{ker}(\gamma)$, $\psi_{2}$ is the restriction of the isomorphism $\hat{\psi}_{i}: N_{in} \rightarrow N_{in}'$ to $ker(\alpha)$ and $\overline{\psi_{2}}: \frac{ker(\alpha)}{im(\gamma)} \rightarrow \frac{ker(\alpha')}{im(\gamma')}$ is the isomorphism induced by $\hat{\psi}_{i}$.

Define $\widetilde{\psi}: \overline{\widehat{N}} \rightarrow \overline{N'}$ as the map $\psi: \widehat{N}_{i} \rightarrow N'_{i}$ for all $i \neq k$ and if $i=k$ then $\widetilde{\psi}$ is the map given in diagonal matrix form by the maps previously defined. Let $d_{1} \in D_{k}$,  $b \in T_{k}$ and $d_{2} \in D_{\sigma(b)}$. Let us show the following equality holds for every $n \in N_{\sigma(b)}$: \\

\begin{center}
$\widetilde{\psi}(d_{1}(^{\ast}b)d_{2} \cdot n) = d_{1}(^{\ast}b)d_{2} \cdot \widetilde{\psi}(n)$
\end{center}

\vspace{0.2in}

First consider the term on the left-hand side. We have \\
\begin{align*}
\widetilde{\psi}(d_{1}(^{\ast}b)d_{2} \cdot n)&=\widetilde{\psi}(d_{1}\overline{N}(^{\ast}b)(d_{2}n))\\
&=\widetilde{\psi}
\begin{pmatrix}
-d_{1} \pi_{1} p \xi_{be_{k}}(d_{2}n) \\
-d_{1} \gamma' \xi_{be_{k}}(d_{2}n) \\
0 \\
0
\end{pmatrix}
\end{align*}

On the other hand, 
\begin{align*}
d_{1}(^{\ast}b)d_{2} \cdot \widetilde{\psi}(n)&=d_{1}(^{\ast}b)d_{2} \cdot \psi(n) \\
&=\begin{pmatrix}
-d_{1}\pi_{1}' p' \xi_{be_{k}}(d_{2}\psi(n)) \\
-d_{1} \gamma' \xi_{be_{k}}(d_{2}\psi(n)) \\
0 \\
0
\end{pmatrix}
\end{align*}

\vspace{0.1in}

In what follows, we will use the notation of \ref{4.2} for the projection maps associated to $\overline{N'}_{k}$. \\

We now show that $\Pi_{1}'\left(\widetilde{\psi}(d_{1}\overline{N}(^{\ast}b)(d_{2}n))\right)=\Pi_{1}' \left(d_{1}(^{\ast}b)d_{2} \cdot \psi(n)\right)$. \\

On one hand
\begin{center}
$\Pi_{1}'\left(\widetilde{\psi}(d_{1}\overline{N}(^{\ast}b)(d_{2}n))\right)=\pi_{1}' \left(-d_{1}\psi_{1}p\xi_{be_{k}}(d_{2}n)\right)$
\end{center}

and on the other hand
\begin{align*}
\Pi_{1}' \left(d_{1}(^{\ast}b)d_{2} \cdot \psi(n)\right)&=-d_{1}\pi_{1}' p' \xi_{be_{k}}(d_{2}\psi(n)) \\
&=-d_{1}\pi_{1}' p' \hat{\psi}_{o} \xi_{be_{k}}(d_{2}n)
\end{align*}

By \ref{4.11} we have $p' \hat{\psi}_{0}=\psi_{1}p$ and thus \\
\begin{center}
$-d_{1}\pi_{1}' p' \hat{\psi}_{o} \xi_{be_{k}}(d_{2}n)= -d_{1} \pi_{1}' \psi_{1}p \xi_{be_{k}}(d_{2}n)=\pi_{1}'\left(-d_{1}\psi_{1}p\xi_{be_{k}}(d_{2}n)\right)$
\end{center}

Therefore $\Pi_{1}'\left(\widetilde{\psi}(d_{1}\overline{N}(^{\ast}b)(d_{2}n))\right)=\Pi_{1}' \left(d_{1}(^{\ast}b)d_{2} \cdot \psi(n)\right)$, as was to be shown. \\

We now verify that $\Pi_{2}'\left(\widetilde{\psi}(d_{1}\overline{N}(^{\ast}b)(d_{2}n))\right)=\Pi_{2}' \left(d_{1}(^{\ast}b)d_{2} \cdot \psi(n)\right)$. We have
\begin{align*}
\Pi_{2}'\left(\widetilde{\psi}(d_{1}\overline{N}(^{\ast}b)(d_{2}n))\right)&=-d_{1}\hat{\psi}_{i}\gamma' \xi_{be_{k}}(d_{2}n) \\
&=-d_{1}\gamma' \hat{\psi}_{o} \xi_{be_{k}}(d_{2}n) \\
&=-d_{1} \gamma' \xi_{be_{k}} (d_{2}\widetilde{\psi}(n))\\
&=\Pi_{2}' \left(d_{1}(^{\ast}b)d_{2} \cdot \psi(n)\right)
\end{align*}

Let $a \in _{k}T$, $d_{1} \in D_{\tau(a)}$, $d_{2} \in D_{k}$ and $w \in \overline{N}_{k}$. We now prove the following equality: \\

\begin{center}
$\widetilde{\psi}(d_{1}a^{\ast}d_{2} \cdot w)=d_{1}a^{\ast}d_{2} \widetilde{\psi}(w)$
\end{center}

Using \ref{4.3} we obtain
\begin{align*}
\widetilde{\psi}(d_{1}a^{\ast} d_{2} w)&=\widetilde{\psi}\left(d_{1}c_{k}^{-1}\pi_{e_{k}a}i \Pi_{2}(d_{2}w)+d_{1}c_{k}^{-1}\pi_{e_{k}a}j'\sigma_{2}\Pi_{3}(d_{2}w)\right) \\
&=d_{1}c_{k}^{-1}\psi\left(\pi_{e_{k}a}i\Pi_{2}(d_{2}w)\right)+d_{1}c_{k}^{-1}\psi \left(\pi_{e_{k}a}j'\sigma_{2}\Pi_{3}(d_{2}w)\right) \\
\end{align*}

On the other hand
\begin{align*}
d_{1}a^{\ast}d_{2}\widetilde{\psi}(w)&=d_{1}\overline{N}'(a^{\ast})(d_{2}\widetilde{\psi}(w)) \\
&=d_{1}c_{k}^{-1}\pi_{e_{k}a}'i \hat{\psi}_{i} \Pi_{2}(d_{2}w)+d_{1}c_{k}^{-1}\pi_{e_{k}a}j'\sigma_{2}'\overline{\psi}_{2}(\Pi_{3}(d_{2}w)))
\end{align*}

and thus (b) follows. This completes the proof of Proposition \ref{prop6}. \\
\end{proof}

\section{Mutation of decorated representations} \label{sec5}

Let $(\mathcal{F}_{S}(M),P)$ be an algebra with potential such that $P^{(2)}$ is splittable. By \cite[Theorem 7.15]{2} there exists a decomposition of $S$-bimodules $M=M_{1} \oplus \Xi_{2}(P)$ and a unitriangular automorphism $\varphi: \mathcal{F}_{S}(M) \rightarrow \mathcal{F}_{S}(M)$ such that $\varphi(P)$ is cyclically equivalent to $Q^{\geq 3} \oplus Q^{(2)}$ where $Q^{\geq 3}$ is a reduced potential in $\mathcal{F}_{S}(M_{1})$ and $Q^{(2)}$ is a trivial potential in $\mathcal{F}_{S}(\Xi_{2}(P))$. \\

We have that $\widehat{\mathcal{N}}$ is a decorated representation of  $(\mathcal{F}_{S}(M),Q^{\geq 3} \oplus Q^{(2)})$. Therefore, $\widehat{\mathcal{N}}|_{\mathcal{F}_{S}(M_{1})}$ is a decorated representation of $(\mathcal{F}_{S}(M_{1}),Q^{\geq 3})$.

\begin{definition} We will refer to the decorated representation $\widehat{\mathcal{N}}|_{\mathcal{F}_{S}(M_{1})}$ as the reduced part of $\mathcal{N}$ and will be denoted by $\mathcal{N}_{red}$.
\end{definition}

Let $k$ be a fixed integer in $[1,n]$ and let $M=M_{1} \oplus M_{2}$ be a $Z$-freely generated $S$-bimodule such that for every $i$, $e_{i}Me_{k} \neq 0$ implies $e_{k}Me_{i}=0$ and likewise $e_{k}Me_{i} \neq 0$ implies $e_{i}Me_{k}=0$. 

Let $T_{1}$ be a $Z$-free generating set of $M_{1}$ and let $T_{2}$ be a $Z$-free generating set of $M_{2}$. Let $P_{1}$ be a potential in $\mathcal{F}_{S}(M_{1})$ and  let $W=\displaystyle \sum_{i=1}^{l} c_{i}d_{i}$ be a trivial potential in $\mathcal{F}_{S}(M_{2})$, where $\{c_{1},\hdots,c_{l},d_{1}\hdots, d_{l}\}=T_{2}$.

Suppose that $e_{k}P_{1}e_{k}=0$ and consider the potential $P=P_{1}+W$ in $\mathcal{F}_{S}(M_{1} \oplus M_{2})$. Since $e_{k}Pe_{k}=0$, then we can consider

\begin{center}
$\mu_{k}(P)=\mu_{k}(P_{1})+W$
\end{center}

Note that
\begin{center}
$\mu_{k}M=\mu_{k}M_{1} \oplus M_{2}$
\end{center}

Let $\mathcal{N}=(N,V)$ be a decorated representation of $(\mathcal{F}_{S}(M),P)$. Recall that $\widetilde{\mu}_{k}(\mathcal{N})_{|\mathcal{F}_{S}(\mu_{k}M_{1})}$ is a decorated representation of $(\mathcal{F}_{S}(\mu_{k}M_{1}),\mu_{k}P_{1})$. Recall also that $\mathcal{N}_{| \mathcal{F}_{S}(M_{1})}$ is a decorated representation of $(\mathcal{F}_{S}(M_{1}),P_{1})$.

\begin{lemma} \label{lem9} The following equality holds: $\widetilde{\mu}_{k}(\mathcal{N}_{| \mathcal{F}_{S}(M_{1})})=(\widetilde{\mu}_{k}(\mathcal{N}))_{| \mathcal{F}_{S}(\mu_{k}M_{1})}$.
\end{lemma}

\begin{proof} Let $T_{1}$ be a $Z$-free generating set of $(M_{1})_{0}$ and let $T_{2}$ be a $Z$-free generating set of $(M_{2})_{0}$. Then

\begin{center}
$W=\displaystyle \sum_{s=1}^{l} c_{s}d_{s}$
\end{center}

where $T_{2}=\{c_{1}, \hdots, c_{l}, d_{1} \hdots, d_{l}\}$. Note that $e_{k}T_{2}=T_{2}e_{k}=0$. We have \\

\begin{center}
$\mathcal{N}_{| \mathcal{F}_{S}(M_{1})}=(N_{| \mathcal{F}_{S}(M_{1})},V)$
\end{center}

Define $N'=N_{\mathcal{F}_{S}(M_{1})}$. For $i \neq k$, $\overline{N'}_{i}=N'_{i}=N_{i}=\overline{N}_{i}$. On the other hand \\

\begin{center}
$N'_{out}=\displaystyle \bigoplus_{b \in (T_{1})_{k}} D_{k} \otimes_{F} N_{\sigma(b)}= \displaystyle \bigoplus_{b \in T_{k}} D_{k} \otimes_{F} N_{\sigma(b)}=N_{out}$
\end{center}

Similarly, $N'_{in}=N_{in}$. Now let $\alpha': N'_{in} \rightarrow N'_{k}$, $\beta': N'_{k} \rightarrow N'_{out}$ and $\gamma': N'_{out} \rightarrow N'_{in}$ be the maps associated to the representation $\overline{N'}$. Clearly $\alpha=\alpha'$, $\beta=\beta'$. Also:

\begin{center}
$Y_{[bra]}(P_{1}+W)=Y_{[bra]}(P_{1})$
\end{center}

and thus $\gamma=\gamma'$. Then $\overline{N'}_{k}=\overline{N}_{k}$ and it follows that $\overline{N'}=\overline{N}$ as left $S$-modules. Since the action of $\mathcal{F}_{S}(\mu_{k}M_{1})$ in $\overline{N'}$ and $\overline{N}$ is the same, then the claim follows.

\end{proof}

\begin{prop} \label{prop7} Let $(\mathcal{F}_{S}(M),P)$ be an algebra with potential where $P^{(2)}$ is splittable. Suppose that $e_{k}Pe_{k}=0$ and that $(M \otimes_{S} e_{k}M)_{cyc}=0$. For every decorated representation $\mathcal{N}$ of $(\mathcal{F}_{S}(M),P)$, the decorated representation $\widetilde{\mu_{k}}(\mathcal{N})_{red}$ is right-equivalent to $\widetilde{\mu_{k}}(\mathcal{N}_{red})_{red}$.
\end{prop}

\begin{proof}
By \cite[Theorem 7.15]{2} there exists a unitriangular automorphism: 

\begin{center}
$\varphi_{1}: \mathcal{F}_{S}(M) \rightarrow \mathcal{F}_{S}(M)$
\end{center}

\vspace{0.1in}
and a decomposition of $S$-bimodules $M=M_{1} \oplus \Xi_{2}(P)$ such that $\varphi_{1}(P)$ is cyclically equivalent to $Q+W_{1}$ where $Q$ is a reduced potential in $\mathcal{F}_{S}(M_{1})$ and $W_{1}$ is a trivial potential in $\mathcal{F}_{S}(\Xi_{2}(P))$. Then $\mu_{k}(\varphi_{1}(P))$ is cyclically equivalent to $\mu_{k}(Q)+W_{1}$. By \cite[Theorem 8.12]{2} there exists a unitriangular automorphism:  
\begin{center}
$\varphi_{2}: \mathcal{F}_{S}(\mu_{k}M_{1}) \rightarrow \mathcal{F}_{S}(\mu_{k}M_{1})$
\end{center}

and a decomposition of $S$-bimodules:

\begin{center}
$\mu_{k}M_{1}=M_{2} \oplus \Xi_{2}(\mu_{k}(Q))$
\end{center}

such that $\varphi_{2}(\mu_{k}Q)=Q_{1}+W_{2}$ where $Q_{1}$ is a reduced potential in $\mathcal{F}_{S}(M_{2})$ and $W_{2}$ is a trivial potential in $\mathcal{F}_{S}(\Xi_{2}(\mu_{k}(Q)))$. Let $\hat{\varphi}_{2}$ be the unitriangular automorphism of $\mathcal{F}_{S}(\mu_{k}M)$ extending $\varphi_{2}$ and such that $\hat{\varphi}_{2}$ is equal to the identity in $\mathcal{F}_{S}(\Xi_{2}(P))$. 
Now consider the algebra automorphism $\hat{\varphi}_{1}$ of $\mathcal{F}_{S}(\mu_{k}M)$ given by Theorem \ref{teo1}. Then \\

\begin{enumerate}[(a)]
\item $\widehat{\varphi}_{1}(\mu_{k}P)$ is cyclically equivalent to $\mu_{k}(\varphi_{1}(P))$.
\item There exists an isomorphism of decorated representations $\widetilde{\mu_{k}}(\mathcal{N}) \rightarrow \widetilde{\mu_{k}}(^{\varphi_{1}}{\mathcal{N}})$
\end{enumerate}

where $^{\varphi_{1}}{\mathcal{N}}$ denotes the representation $\widehat{\mathcal{N}}$ (to emphasize the dependance of the action on $\varphi_{1}$). We have that $\widehat{\varphi_{2}}\widehat{\varphi}_{1}(\mu_{k}P)$ is cyclically equivalent to $\widehat{\varphi}_{2}\mu_{k}(\varphi_{1}(P))$ and the latter is cyclically equivalent to $\widehat{\varphi}_{2}(\mu_{k}Q+W_{1})=\varphi_{2}(\mu_{k}(Q))+W_{1}$; also, the latter potential is cyclically equivalent to $Q_{1}+W_{1}+W_{2}$ and

\begin{center}
$\mu_{k}M=M_{2} \oplus \Xi_{2}(\mu_{k}Q) \oplus \Xi_{2}(P)$
\end{center}

with $Q_{1}$ a reduced potential in $\mathcal{F}_{S}(M_{2})$ and $W_{1}+W_{2}$ a trivial potential in $\mathcal{F}_{S}(\Xi_{2}(\mu_{k}Q) \oplus \Xi_{2}(P))$. Then $(\widetilde{\mu_{k}}(\mathcal{N}))_{red}$ is isomorphic to $^{\hat{\varphi}_{2}}(\widetilde{\mu_{k}}(^{\varphi_{1}}\mathcal{N}))_{| \mathcal{F}_{S}(M_{2})}$. Note that, by Lemma \ref{lem9}, $^{\widehat{\varphi}_{2}} \widetilde{\mu_{k}}(^{\varphi_{1}}\mathcal{N})_{|\mathcal{F}_{S}(\mu_{k}M_{1})}$ is isomorphic to $^{\widehat{\varphi}_{2}} \widetilde{\mu_{k}}(^{\varphi_{1}}\mathcal{N}_{|\mathcal{F}_{S}(M_{1})})$. The claim follows.
\end{proof}
\begin{definition} \label{def13} Let $(N,V)$ be a decorated representation of the algebra with potential $(\mathcal{F}_{S}(M),P)$ where $P$ is reduced. We define the mutation of the decorated representation $\mathcal{N}$ in $k$ as: \\

\begin{center}
$\mu_{k}(\mathcal{N}):=\widetilde{\mu_{k}}(\mathcal{N})_{red}$
\end{center}
\end{definition}

\begin{cor} The correspondence $\mathcal{N} \mapsto \mu_{k}(\mathcal{N})$ is a well-defined transformation on the set of right-equivalence classes of decorated representations of reduced algebras with potentials.
\end{cor}

\begin{theorem} \label{theo2} The mutation $\mu_{k}$ of decorated representations is an involution; that is, for every decorated representation $\mathcal{N}$ of a reduced algebra with potential $(\mathcal{F}_{S}(M),P)$, the decorated representation $\mu_{k}^{2}(\mathcal{N})$ is right-equivalent to $\mathcal{N}$. 
\end{theorem}

\begin{proof} Let $\mathcal{N}=(N,V)$ be a decorated representation of $(\mathcal{F}_{S}(M),P)$. First note that
\begin{align*}
\mu_{k}^{2}(\mathcal{N})&=\mu_{k}\left(\mu_{k}(\mathcal{N})\right) \\
&=\widetilde{\mu_{k}}(\mu_{k}\left(\mathcal{N}\right))_{red} \\
&=\widetilde{\mu_{k}}\left( \widetilde{\mu_{k}}(\mathcal{N}\right)_{red})_{red}
\end{align*}

In view of Proposition \ref{prop7}, $\widetilde{\mu_{k}}\left( \widetilde{\mu_{k}}(\mathcal{N}\right)_{red})_{red}$ is right-equivalent to $\widetilde{\mu_{k}} \left( \widetilde{\mu_{k}} (\mathcal{N})\right)_{red}=\widetilde{\mu_{k}}^{2}(\mathcal{N})_{red}$. Thus, it suffices to show that $\widetilde{\mu_{k}}^{2}\left(\mathcal{N}\right)_{red}$ is right-equivalent to $\mathcal{N}$. 
We write the representation $\widetilde{\mu_{k}}^{2}\left(\mathcal{N}\right)$ as $\left(\overline{\overline{N}}, \overline{\overline{V}}\right)$ and this representation is associated to the algebra with potential $\left(\mathcal{F}_{S}(\mu_{k}^{2}M),\mu_{k}^{2}P\right)$. By \cite[Proposition 8.8]{2} and \cite[Theorem 8.17]{2} we have:
\begin{align*}
\mu_{k}^{2}M&=M \oplus Me_{k}M \oplus M^{\ast}e_{k}(^{\ast}M) \\
\mu_{k}^{2}P&=\rho(P)+\displaystyle \sum_{bt,sa} \left( [btsa][(sa)^{\ast}(^{\ast}bt)]+[(sa)^{\ast}(^{\ast}(bt))](bt)(sa)\right)
\end{align*}

By \cite[Theorem 8.17]{2} there exists an algebra automorphism $\psi: \mathcal{F}_{S}(\mu_{k}^{2}M) \rightarrow \mathcal{F}_{S}(\mu_{k}^{2}M)$ such that $\psi(\mu_{k}^{2}P)$ is cyclically equivalent to $P \oplus W$ where $W$ is a trivial potential in $\mathcal{F}_{S}\left(Me_{k}M \oplus M^{\ast}e_{k}(^{\ast}M)\right)$. Such automorphism $\psi$ restricts to an automorphism $\psi_{0}: \mathcal{F}_{S}(M) \rightarrow \mathcal{F}_{S}(M)$ such that $\psi_{0}(b)=-b$ for every $b \in T_{k}$ and $\psi_{0}(x)=x$ for every $x \in T \setminus T_{k}$. In view of Definition \ref{def13}, the representation $\widetilde{\mu_{k}}(\mathcal{N})_{red}$ can be realized as $(\overline{\overline{N}},\overline{\overline{V}})$, the latter being associated to the algebra with potential $(\mathcal{F}_{S}(M),P)$, and the action of $\mathcal{F}_{S}(M)$ in $\overline{\overline{N}}$ is given by $u \cdot n = \psi_{0}^{-1}(u)n$. \\

Let us show that $\overline{\alpha}: N_{out}=\overline{N}_{in} \longrightarrow \overline{N}_{k}$ is the $F$-linear map such that for each $b \in T_{k}$ and $r \in L(k)$, $\overline{\alpha} \xi_{br}=r^{-1}\overline{N}(^{\ast}b)$. We have

\begin{align*}
\overline{\alpha}\xi_{br}(n)&=\displaystyle \sum_{w \in L(k)} \overline{\alpha}\xi_{w(^{\ast}b)}\pi_{w(^{\ast}b)}\xi_{br}(n) \\
&=\displaystyle \sum_{w \in L(k)} \overline{\alpha}\xi_{w(^{\ast}b)}\pi'_{w(^{\ast}b)}(r^{-1} \otimes n) \\
&=\displaystyle \sum_{w \in L(k)} \overline{\alpha}\xi_{w(^{\ast}b)}w^{\ast}(r^{-1})n \\
&=\displaystyle \sum_{w \in L(k)} \overline{N}(w(^{\ast}b))(w^{\ast}(r^{-1})n) \\
&=\overline{N} \left( \displaystyle \sum_{w \in L(k)} w^{\ast}(r^{-1})w(^{\ast}b)\right)(n) \\
&=\overline{N}(r^{-1}(^{\ast}b))(n) \\
&=r^{-1}\overline{N}(^{\ast}b)(n)
\end{align*}

We now show that $ker(\overline{\alpha})=\operatorname{im}(\beta)$. \\

Let $x \in N_{out}$, using \ref{3.15} we have $x=\displaystyle \sum_{b \in T_{k},r \in L(k)} \xi_{br} \pi_{br}(x)$. Therefore
\begin{align*}
\overline{\alpha}(x)&= \displaystyle \sum_{b \in T_{k}, r \in L(k)} \overline{\alpha} \xi_{br} \pi_{br}(x) \\
&=\displaystyle \sum_{b \in T_{k}, r \in L(k)}r^{-1} \overline{N}(^{\ast} b)\left(\pi_{br}(x)\right)
\end{align*}

\vspace{0.1in}

thus $\overline{\alpha}(x)=0$ if and only if $\Pi_{i} \left(\displaystyle \sum_{b \in T_{k}, r \in L(k)}r^{-1} \overline{N}(^{\ast} b)\left(\pi_{br}(x)\right)\right)=0$ for every $i \in \{1,2,3,4\}$. Remembering \ref{4.4} yields
\begin{equation} \label{5.1}
\Pi_{1}\left(\displaystyle \sum_{b \in T_{k}, r \in L(k)}r^{-1} \overline{N}(^{\ast} b)\left(\pi_{br}(x)\right)\right)= -\displaystyle \sum_{b \in T_{k}, r \in L(k)} r^{-1}\pi_{1} p \xi_{be_{k}}\pi_{br}(x) 
\end{equation}

\begin{equation} \label{5.2}
\Pi_{2}\left(\displaystyle \sum_{b \in T_{k}, r \in L(k)}r^{-1} \overline{N}(^{\ast} b)\left(\pi_{br}(x)\right)\right)= -\displaystyle \sum_{b \in T_{k}, r \in L(k)} r^{-1}\gamma' \xi_{be_{k}}\pi_{br}(x) 
\end{equation}

\vspace{0.1in}

so if $\overline{\alpha}(x)=0$ then \ref{5.2} implies that \\
\begin{equation} \label{5.3}
\displaystyle \sum_{b \in T_{k}, r \in L(k)} r^{-1} \xi_{be_{k}} \pi_{br}(x) \in \operatorname{ker}(\gamma)
\end{equation}

On the other hand, if $\overline{\alpha}(x)=0$ then by \ref{5.1}, $\displaystyle \sum_{b \in T_{k}, r \in L(k)} r^{-1}p \xi_{be_{k}} \pi_{br}(x) \in \operatorname{im}(\beta)$. Note that
\begin{center}
$\displaystyle \sum_{b \in T_{k}, r \in L(k)} r^{-1}p \xi_{be_{k}} \pi_{br}(x) =p \left( \displaystyle \sum_{b \in T_{k}, r \in L(k)} r^{-1}\xi_{be_{k}}\pi_{br}(x)\right)$
\end{center}

\vspace{0.2in}

Using \ref{5.3} and the fact that $p$ is a left inverse of the inclusion map $j: ker(\gamma) \rightarrow N_{out}$ yields $p \left( \displaystyle \sum_{b \in T_{k}, r \in L(k)} r^{-1}\xi_{be_{k}}\pi_{br}(x)\right)=\displaystyle \sum_{b \in T_{k}, r \in L(k)} r^{-1}\xi_{be_{k}}\pi_{br}(x)$. Finally, note that

\begin{center}
$\displaystyle \sum_{b \in T_{k}, r \in L(k)} r^{-1}\xi_{be_{k}}\pi_{br}(x)=\displaystyle \sum_{b \in T_{k}, r \in L(k)} \xi_{br} \pi_{br}(x)$
\end{center}

\vspace{0.1in}

and by \ref{3.15} $\displaystyle \sum_{b \in T_{k}, r \in L(k)} \xi_{br} \pi_{br}(x)=x$. Therefore, $\overline{\alpha}(x)=0$ if and only if $\displaystyle \sum_{b \in T_{k}, r \in L(k)} r^{-1}p \xi_{be_{k}} \pi_{br}(x)=x \in \operatorname{im}(\beta)$. This proves that: 
\begin{equation} \label{5.4}
ker(\overline{\alpha})=\operatorname{im}(\beta)
\end{equation}

\vspace{0.2in}

As a consequence, $\Pi_{1} \overline{\alpha}(x)=-\pi_{1}p$ y $\Pi_{2} \overline{\alpha} = - \gamma'$. Therefore

\begin{equation} \label{5.5}
\operatorname{im}(\overline{\alpha})=\frac{\operatorname{ker(\gamma)}}{\operatorname{im}(\beta)} \oplus \operatorname{im}(\gamma) \oplus \{0\} \oplus \{0\}
\end{equation}

\vspace{0.2in}

We now show that $\overline{\beta}: \overline{N}_{k} \longrightarrow \overline{N}_{out}=N_{in}$ is the $F$-linear map such that for each $r \in L(k)$ and $a \in _{k}T$, $\pi_{ra} \overline{\beta}=\overline{N}(a^{\ast})r^{-1}$. We have:
\begin{align*}
\pi_{ra}\overline{\beta}(n)&=\displaystyle \sum_{w \in L(k)} \pi_{ra} \xi_{a^{\ast}w} \pi_{a^{\ast}w} \overline{\beta}(n) \\
&=\displaystyle \sum_{w \in L(k)} \pi_{ra}\xi_{a^{\ast}w}\overline{N}(a^{\ast}w)(n) \\
&=\displaystyle \sum_{w \in L(k)} \pi'_{ra}\left(w^{-1} \otimes \overline{N}(a^{\ast}w)(n)\right) \\
&=\displaystyle \overline{N} \left(a^{\ast} \displaystyle \sum_{w \in L(k)} r^{\ast}(w^{-1})w\right)(n)
\end{align*}

By Remark \ref{rem3}, $\displaystyle \sum_{w \in L(k)} r^{\ast}(w^{-1})w=r^{-1}$. It follows that

\begin{center}
$\pi_{ra} \overline{\beta}(n)=\overline{N}(a^{\ast}r^{-1})(n)=\overline{N}(a^{\ast})(r^{-1}n)$
\end{center}

We now prove that
\begin{equation} \label{5.6}
\operatorname{ker}(\overline{\beta})=\frac{\operatorname{ker}(\gamma)}{\operatorname{im}(\beta)} \oplus \{0\} \oplus \{0\} \oplus V_{k} 
\end{equation}

Let $w \in \overline{N}_{k}$, then $w=\displaystyle \sum_{l=1}^{4} J_{l} \Pi_{l}(w)$. Suppose that $\overline{\beta}(w)=0$, then $\pi_{ra}\overline{\beta}(w)=0$ for every $r \in L(k)$ and $a \in _{k}T$. Using Lemma \ref{lem2} and \ref{4.3} we obtain the following equalities:

\begin{align*}
0=\pi_{ra} \overline{\beta}(w)&=\displaystyle \sum_{l=1}^{4} \pi_{ra} \overline{\beta} \left(J_{l} \Pi_{l}(w) \right) \\
&=\displaystyle \sum_{l=1}^{4} \overline{N}(a^{\ast})\left(J_{l}r^{-1}\Pi_{l}(w)\right) \\
&=\overline{N}(a^{\ast})J_{2}\left(r^{-1}\Pi_{2}(w)\right)+\overline{N}(a^{\ast})J_{3}\left(r^{-1}\Pi_{3}(w)\right) \\
&=c_{k}^{-1}\pi_{e_{k}a}i\left(r^{-1}\Pi_{2}(w)\right)+c_{k}^{-1}\pi_{e_{k}a}j'\sigma_{2}\left(r^{-1}\Pi_{3}(w)\right) \\
&=c_{k}^{-1}\pi_{ra}\left(i\Pi_{2}(w)\right)+c_{k}^{-1}\pi_{ra}\left(j'\sigma_{2}\Pi_{3}(w)\right) \\
&=c_{k}^{-1}\pi_{ra}\left(i\Pi_{2}(w)+j'\sigma_{2}\Pi_{3}(w)\right)
\end{align*}

Since this is true for all projections $\pi_{ra}$, then \\
\begin{equation} \label{5.7}
0=i \Pi_{2}(w) + j' \sigma_{2}\Pi_{3}(w)=\Pi_{2}(w)+\sigma_{2}\Pi_{3}(w)
\end{equation}

By \ref{4.2}, $\Pi_{2}(w) \in \operatorname{im}(\gamma)$. Applying $\pi_{2}$ to \ref{5.7}, where $\pi_{2}: \operatorname{ker}(\alpha) \rightarrow \frac{ \operatorname{ker}(\alpha)}{\operatorname{im}(\gamma)}$ is the projection map, yields $\pi_{2} \sigma_{2} \Pi_{3}(w)=0$. Since $\pi_{2}\sigma_{2}=id_{ker(\alpha)/im(\gamma)}$, then $\Pi_{3}(w)=0$. Substituting $\Pi_{3}(w)=0$ into \ref{5.7} gives $\Pi_{2}(w)=0$. Consequently:

\begin{center}
$\operatorname{ker}(\overline{\beta})=\frac{\operatorname{ker}(\gamma)}{\operatorname{im}(\beta)} \oplus \{0\} \oplus \{0\} \oplus V_{k}$ 
\end{center}

and the proof of \ref{5.6} is now complete. We also have
\begin{equation} \label{5.8}
\operatorname{im}(\overline{\beta})=\operatorname{ker}(\alpha)
\end{equation}

We now prove the following formula
\begin{equation} \label{5.9}
\operatorname{\overline{\gamma}}=c_{k}\beta \alpha
\end{equation}

First, we compute $Y_{[a^{\ast}w(^{\ast}b)]}(\mu_{k}P)$ where $a \in _{k}T, b \in T_{k}$, $w \in L(k)$ and $\Delta_{k}$ is the following potential 
\begin{center}
$\Delta_{k}=\displaystyle \sum_{sa_{1} \in _{k}\hat{T}, b_{1}t \in \tilde{T}_{k}} [b_{1}tsa_{1}]((sa_{1})^{\ast})(^{\ast}(b_{1}t))$
\end{center}

Note that $\Delta_{k}$ is cyclically equivalent to the following potential 
\begin{center}
$\Delta_{k}'=\displaystyle \sum_{sa_{1},b_{1}t} a_{1}^{\ast}s^{-1}t^{-1}(^{\ast}b_{1})[b_{1}tsa_{1}]$.
\end{center}

Therefore
\begin{align*}
\Delta_{k}' &= \displaystyle \sum_{sa_{1},b_{1}t} \displaystyle \sum_{v,v_{1} \in L(k)} (v_{1}^{-1})^{\ast}(s^{-1}t^{-1})v^{\ast}(ts)a_{1}^{\ast}v_{1}^{-1}(^{\ast}b_{1})[b_{1}va_{1}] \\
&=\displaystyle \sum_{sa_{1},b_{1}} \displaystyle \sum_{v,v_{1},t \in L(k)}  (v_{1}^{-1})^{\ast}(s^{-1}t^{-1})v^{\ast}(ts)a_{1}^{\ast}v_{1}^{-1}(^{\ast}b_{1})[b_{1}va_{1}] 
\end{align*}

By \cite[Proposition 7.5]{2} we have: 
\begin{center}
$\displaystyle \sum_{t \in L(k)} v^{\ast}(ts)(v_{1}^{-1})^{\ast}(s^{-1}t^{-1})=\delta_{v,v_{1}}$
\end{center}

thus
\begin{align*}
\Delta_{k}' &= \displaystyle \sum_{sa_{1},b_{1}} \displaystyle \sum_{v \in L(k)} a_{1}^{\ast} v^{-1} (^{\ast}b_{1}) [b_{1}va_{1}] \\
&= \displaystyle \sum_{sa_{1},b_{1}} \displaystyle \sum_{v,r \in L(k)} r^{\ast}(v^{-1}) a_{1}^{\ast} r (^{\ast} b_{1}) [b_{1}va_{1}]
\end{align*}

Therefore
\begin{center}
$\rho \left(\Delta_{k}'\right)= \displaystyle \sum_{sa_{1},b_{1}} \displaystyle \sum_{v,r \in L(k)} r^{\ast}(v^{-1})[a_{1}^{\ast}r(^{\ast}b_{1})][b_{1}va_{1}]$
\end{center}

\vspace{0.1in}

Let $a \in _{k}T$, $w \in L(k)$ and $b \in T_{k}$ be fixed. Then: \\

\begin{center}
$X_{[a^{\ast}w(^{\ast}b)]}\left(\rho(\Delta_{k}')\right)=\displaystyle \sum_{s,v \in L(k)} w^{\ast}(v^{-1})[bva]$
\end{center}

Note that
\begin{align*}
Y_{[a^{\ast}w(^{\ast}b)]}(\mu_{k}P)&=\rho^{-1}\left(X_{[a^{\ast}w(^{\ast}b)]}(\rho(\mu_{k}P))\right) \\
&=\rho^{-1}\left(X_{[a^{\ast}w(^{\ast}b)]}(\rho(\Delta_{k}'))\right) \\
&=\rho^{-1}\left(X_{[a^{\ast}w(^{\ast}b)]}(\rho(\Delta_{k}'))\right)
\end{align*}

Therefore
\begin{center}
$Y_{[a^{\ast}w(^{\ast}b)]}\left(\mu_{k}P\right)=\displaystyle \sum_{s,v \in L(k)} w^{\ast}(v^{-1})[bva]=c_{k}\displaystyle \sum_{v \in L(k)} w^{\ast}(v^{-1})[bva]$
\end{center}

Consequently
\begin{align*}
\pi_{bt}\overline{\gamma} \xi_{sa}(n) &= \displaystyle \sum_{w \in L(k)} (t^{-1})^{\ast}(sw)Y_{[a^{\ast}w(^{\ast}b)]}\left(\mu_{k}P\right)n \\
&=c_{k} \displaystyle \sum_{v,w \in L(k)} (t^{-1})^{\ast}(sw) w^{\ast}(v^{-1})[bva]n \\
&=c_{k} \displaystyle \sum_{v \in L(k)}(t^{-1})^{\ast}\left(s\displaystyle \sum_{w \in L(k)} w^{\ast}(v^{-1})w\right)[bva]n \\
&=c_{k}\displaystyle \sum_{v \in L(k)} (t^{-1})^{\ast}(sv^{-1})[bva]n
\end{align*}

By \ref{2.3}, $(t^{-1})^{\ast}(sv^{-1})=v^{\ast}(ts)$. Then

\begin{align*}
\pi_{bt}\overline{\gamma} \xi_{sa}(n) &= c_{k} \displaystyle \sum_{v \in L(k)} v^{\ast} (ts) [bva]n \\
&= c_{k} \left[b \displaystyle \sum_{v \in L(k)} v^{\ast}(ts)v a \right]n \\
&=c_{k} [btsa]n \\
&=c_{k} btsan \\
&=c_{k}\pi_{bt} \beta \alpha \xi_{sa}(n)
\end{align*}

As a direct consequence of \ref{5.4}, \ref{5.5}, \ref{5.6}, \ref{5.8} and \ref{5.9} we conclude that
\begin{equation} \label{5.10}
\begin{aligned}
& \operatorname{ker}(\overline{\alpha})=\operatorname{im}(\beta), \operatorname{im}(\overline{\alpha})=\frac{\operatorname{ker}(\gamma)}{\operatorname{im}(\beta)} \oplus \operatorname{im}(\gamma) \oplus \{0\} \oplus \{0\}, \\
& \operatorname{ker}(\overline{\beta})=\frac{ \operatorname{ker}(\gamma)}{\operatorname{im}(\beta)} \oplus \{0\} \oplus \{0\} \oplus V_{k},    \operatorname{im}(\overline{\beta})=\operatorname{ker}(\alpha), \\
& \operatorname{ker}(\overline{\gamma})=\operatorname{ker}(\beta \alpha),  \operatorname{im}(\overline{\gamma})=\operatorname{im}(\beta \alpha).
\end{aligned}
\end{equation}

Therefore
\begin{center}
$\overline{ \overline{V}}_{k} = \frac{ \operatorname{ker}(\overline{\beta})}{\operatorname{ker}(\overline{\beta}) \cap \operatorname{im}(\overline{\alpha})}=V_{k}$
\end{center}

and hence $\overline{\overline{V}}=V$. \vspace{0.2in}

On the other hand

\begin{center}
$\overline{\overline{N}}_{k}= \frac{ \operatorname{ker}(\overline{\gamma})}{\operatorname{im}(\overline{\beta})} \oplus \operatorname{im}(\overline{\gamma}) \oplus \frac{ \operatorname{ker}(\overline{\alpha})}{\operatorname{im}(\overline{\gamma})} \oplus \overline{V}_{k}$
\end{center}

and by \ref{5.10}

\begin{center}
$\overline{\overline{N}}_{k} = \frac{ \operatorname{ker}(\beta \alpha)}{\operatorname{ker}(\alpha)} \oplus \operatorname{im}(\beta \alpha) \oplus \frac{\operatorname{im}(\beta)}{\operatorname{im}(\beta \alpha)} \oplus \frac{ \operatorname{ker}(\beta)}{\operatorname{ker}(\beta) \cap  \operatorname{im}(\alpha)}$
\end{center}

\vspace{0.1in}

We now make the following observations:
\begin{enumerate}[(a)]
\item the map $\alpha$ induces an isomorphism $\operatorname{ker}(\beta \alpha)/ \operatorname{ker}(\alpha) \stackrel{\theta_{1}}{\longrightarrow}\operatorname{ker}(\beta) \cap \operatorname{im}(\alpha)$;
\item the map $\beta$ induces an isomorphism $\operatorname{im}(\alpha)/( \operatorname{ker}(\beta) \cap \operatorname{im}(\alpha)) \stackrel{\theta_{2}}{\longrightarrow} \operatorname{im}(\beta \alpha)$;
\item the map $\beta$ induces an isomorphism $N_{k}/(\operatorname{ker}(\beta) + \operatorname{im}(\alpha)) \stackrel{\theta_{3}}{\longrightarrow} \operatorname{im}(\beta)/ \operatorname{im}(\beta \alpha)$.
\item there exists an isomorphism $\operatorname{ker}(\beta)/( \operatorname{ker}(\beta) \cap \operatorname{im}(\alpha)) \stackrel{\theta_{4}}{\longrightarrow} (\operatorname{ker}(\beta) + \operatorname{im}(\alpha))/ \operatorname{im}(\alpha)$.
\end{enumerate}

Using these isomorphisms, we can identify $\overline{\overline{N}}_{k}$ with the space \\

\begin{center}
$\overline{\overline{N}}_{k} \stackrel{f}{\longrightarrow} (\operatorname{ker}(\beta) \cap \operatorname{im}(\alpha)) \oplus \frac{\operatorname{im}(\alpha)}{\operatorname{ker}(\beta) \cap \operatorname{im}(\alpha)} \oplus \frac{N_{k}}{\operatorname{ker}(\beta) + \operatorname{im}(\alpha)} \oplus \frac{\operatorname{ker}(\beta) + \operatorname{im}(\alpha)}{\operatorname{im}(\alpha)}$
\end{center}

via the map
\[f=
\begin{bmatrix}
\theta_{1} & 0 & 0 & 0 \\
0 & \theta_{2}^{-1} & 0 & 0 \\
0 & 0 & \theta_{3}^{-1} & 0 \\
0 & 0 & 0 & \theta_{4}
\end{bmatrix}
\]

We have canonical projections
\begin{align*}
\overline{\pi}_{1}&: \operatorname{ker}(\beta \alpha) \rightarrow \frac{ \operatorname{ker}(\beta \alpha)}{ \operatorname{ker}(\alpha)} \\
\overline{\pi}_{2}&: \operatorname{im}(\beta) \rightarrow \frac{ \operatorname{im}(\beta)}{ \operatorname{im}(\beta \alpha)}\\
\overline{\pi}_{3}&: \operatorname{im}(\alpha) \rightarrow \frac{ \operatorname{im}(\alpha)}{\operatorname{ker}(\beta) \cap \operatorname{im}
(\alpha)} \\
\overline{\pi}_{4}&: N_{k} \rightarrow \frac{N_{k}}{\operatorname{ker}(\beta) + \operatorname{im}(\alpha)} \\
\overline{\pi}_{5}&: \operatorname{ker}(\beta)+\operatorname{im}(\alpha) \rightarrow \frac{\operatorname{ker}(\beta) + \operatorname{im}(\alpha)}{\operatorname{im}(\alpha)}
\end{align*}

and inclusion maps
\begin{align*}
&\overline{i}_{1}: \operatorname{im}(\overline{\gamma}) \rightarrow N_{out} \\
&\overline{i}_{2}: \operatorname{ker}(\beta \alpha) \rightarrow N_{in} \\
&\overline{i}_{3}: \operatorname{im}(\beta \alpha) \rightarrow N_{k}\\
&\overline{j}: \operatorname{ker}(\overline{\alpha}) \rightarrow N_{out} \\
\end{align*}

We now choose splitting data as follows: \\

\begin{enumerate}[(a)]
\item Let $\overline{p}: N_{in} \rightarrow \operatorname{ker}(\beta \alpha)$ be a map of left $D_{k}$-modules such that $\overline{p}\overline{i}_{2} = id_{ \operatorname{ker}(\beta \alpha)}$. \\
\item Let $\overline{\sigma}: \operatorname{im}(\beta)/\operatorname{im}(\beta \alpha) \rightarrow \operatorname{im}(\beta)$ be a map of left $D_{k}$-modules such that $\overline{\pi}_{2} \overline{\sigma}=id_{\operatorname{im}(\beta)/\operatorname{im}(\beta \alpha)}$. 
\end{enumerate}

\vspace{0.1in}

For each $a \in _{k}T$, there exists an $F$-linear map \\

\begin{center}
$\overline{\overline{N}}(a): N_{\tau(a)} \rightarrow (\operatorname{ker}(\beta) \cap \operatorname{im}(\alpha)) \oplus \frac{\operatorname{im}(\alpha)}{\operatorname{ker}(\beta) \cap \operatorname{im}(\alpha)} \oplus \frac{N_{k}}{\operatorname{ker}(\beta) + \operatorname{im}(\alpha)} \oplus \frac{\operatorname{ker}(\beta) + \operatorname{im}(\alpha)}{\operatorname{im}(\alpha)}$
\end{center}

such that
\begin{align*}
&\overline{\Pi}_{1} \overline{\overline{N}}(a)=-\theta_{1}\overline{\pi}_{1} \overline{p} \xi_{e_{k}a} \\
&\overline{\Pi}_{2} \overline{\overline{N}}(a)=-\theta_{2}^{-1} \overline{\gamma}' \xi_{e_{k}}a\\
&\overline{\Pi}_{3} \overline{\overline{N}}(a) = \overline{\Pi}_{4} \overline{\overline{N}}(a) = 0
\end{align*}

Let $n \in N_{\tau(a)}$. Then
\begin{align*}
-\theta_{1} \overline{\pi}_{1} \overline{p} \xi_{e_{k}a}(n) &= -\theta_{1}\left(\overline{p}\xi_{e_{k}a}(n)+\operatorname{ker}(\alpha)\right)=-\alpha \overline{p} \xi_{e_{k}a}(n) \\
-\theta_{2}^{-1}(\overline{\gamma}'(\xi_{e_{k}a}(n)))&=-\theta_{2}^{-1}(c_{k}\beta \alpha \xi_{e_{k}a}(n))=-c_{k}\overline{\pi}_{3} \alpha \xi_{e_{k}a}(n)
\end{align*}

\vspace{0.2in}

Thus the map $\overline{\overline{N}}(a)$ takes the following form:
\begin{equation} \label{5.11}
\begin{split}
\overline{\Pi}_{1} \overline{\overline{N}}(a)&=-\alpha \overline{p} \xi_{e_{k}a} \\
\overline{\Pi}_{2} \overline{\overline{N}}(a)&= -c_{k}\overline{\pi}_{3} \alpha \xi_{e_{k}a} \\
\overline{\Pi}_{3} \overline{\overline{N}}(a)&=\overline{\Pi}_{4} \overline{\overline{N}}(a)=0
\end{split}
\end{equation}

\vspace{0.2in}

Similarly, for each $b \in T_{k}$, there exists an $F$-linear map
\begin{center}
$\overline{\overline{N}}(b): (\operatorname{ker}(\beta) \cap \operatorname{im}(\alpha)) \oplus \frac{\operatorname{im}(\alpha)}{\operatorname{ker}(\beta) \cap \operatorname{im}(\alpha)} \oplus \frac{N_{k}}{\operatorname{ker}(\beta) + \operatorname{im}(\alpha)} \oplus \frac{\operatorname{ker}(\beta) + \operatorname{im}(\alpha)}{\operatorname{im}(\alpha)}  \longrightarrow N_{\sigma(b)}$
\end{center}

\vspace{0.2in}

given by
\begin{align*}
\overline{\overline{N}}(b)\overline{J}_{1}&=\overline{\overline{N}}(b)\overline{J}_{4}=0 \\
\overline{\overline{N}}(b)\overline{J}_{2}&=c_{k}^{-1} \pi_{be_{k}}\overline{i}_{3} \theta_{2} \\
\overline{\overline{N}}(b)\overline{J}_{3}&=c_{k}^{-1} \pi_{be_{k}} \overline{j} \overline{\sigma} \theta_{3}
\end{align*}

To complete the proof of Theorem \ref{theo2}, it remains to construct an isomorphism of $F$-vector spaces $\psi: \overline{\overline{N}}_{k} \rightarrow N_{k}$ such that
\begin{align*}
\psi \overline{\overline{N}}(a)&=\alpha \xi_{e_{k}a} \\
\pi_{be_{k}}\beta \psi &= \overline{\overline{N}}(b)
\end{align*}

for every $a \in _{k}T$ and $b \in T_{k}$. Choose some sections
\begin{align*}
\sigma_{1}&: \operatorname{im} \alpha/( \operatorname{ker}\beta \cap \operatorname{im} \alpha) \rightarrow \operatorname{im}\alpha \\
\sigma_{2}&: (\operatorname{ker} \beta + \operatorname{im} \alpha)/ \operatorname{im}(\alpha) \rightarrow \operatorname{ker} \beta + \operatorname{im} \alpha \\
\sigma_{3}&: N_{k}/(\operatorname{ker} \beta + \operatorname{im} \alpha) \rightarrow N_{k}
\end{align*}

so that they satisfy
\begin{center}
$\operatorname{im}(\sigma_{1})=\alpha(\operatorname{ker} \overline{p})$, $\operatorname{im}(\sigma_{2}) \subseteq \operatorname{ker}(\beta)$, $\operatorname{im}(\beta \sigma_{3})=\operatorname{im}(\overline{\sigma})$
\end{center}

\vspace{0.1in} 

Define $\psi: \overline{\overline{N}}_{k} \rightarrow N_{k}$ as

\[\psi=
\begin{bmatrix}
-i_{1}  & -c_{k}^{-1}i_{2}\sigma_{1} & -c_{k}^{-1}\sigma_{3} & -i_{3}\sigma_{2}
\end{bmatrix}
\]
where
\begin{align*}
&i_{1}: \operatorname{ker}(\beta) \cap \operatorname{im}(\alpha) \rightarrow N_{k} \\
&i_{2}: \operatorname{im}(\alpha) \rightarrow N_{k} \\
&i_{3}: \operatorname{ker}(\beta)+\operatorname{im}(\alpha) \rightarrow N_{k}
\end{align*}

are the inclusion maps. We now show that

\begin{equation} \label{5.12}
\psi \overline{\overline{N}}(a)=\alpha \xi_{e_{k}a}
\end{equation}

Since $\overline{p}: N_{in} \rightarrow \operatorname{ker}(\beta \alpha)$ is a retraction, then $N_{in}=\operatorname{ker}(\overline{p}) \oplus \operatorname{ker}(\beta \alpha).$ On the other hand, since $\operatorname{ker}(\alpha) \subseteq \operatorname{ker}(\beta \alpha)$ then there exists a submodule $L \subseteq \operatorname{ker}(\beta \alpha)$ such that $\operatorname{ker}(\alpha) \oplus L = \operatorname{ker}(\beta \alpha)$. Therefore $N_{in}=\operatorname{ker}(\overline{p}) \oplus \operatorname{ker}(\alpha) \oplus L$.

Let $n \in N_{\tau(a)}$, then $\xi_{e_{k}a}(n) \in N_{in}$. Thus, $\xi_{e_{k}a}(n)=n_{1}+n_{2}+n_{3}$ for some uniquely determined $n_{1} \operatorname{ker}(\overline{p})$, $n_{2} \in \operatorname{ker}(\alpha)$ and $n_{3} \in L$. Therefore:

\begin{center}
$\alpha \xi_{e_{k}a}(n)=\alpha(n_{1}+n_{2}+n_{3})=\alpha(n_{1}+n_{3})$
\end{center}

On the other hand:
\begin{align*}
\left(\psi \overline{\overline{N}}(a)\right)(n)&= \psi \left( - \alpha \overline{p} \xi_{e_{k}a}(n), -c_{k}\overline{\pi}_{3} \alpha \xi_{e_{k}a}(n),0,0\right) \\
&=\alpha \overline{p}(n_{1}+n_{2}+n_{3})+c_{k}^{-1}c_{k}\sigma_{1} \overline{\pi}_{3} \alpha(n_{1}+n_{3}) \\
&=\alpha \overline{p}(n_{2}+n_{3})+\alpha(n_{1}) \\
&=\alpha(n_{2}+n_{3})+\alpha(n_{1}) \\
&=\alpha(n_{1}+n_{3}) \\
&=\alpha \xi_{e_{k}a}(n)
\end{align*}

and the proof of \ref{5.12} is now complete. \\

We now verify that \\
\begin{equation} \label{5.13}
\pi_{be_{k}}\beta \psi = \overline{\overline{N}}(b)
\end{equation}

\vspace{0.1in}

Let $n_{1} \in \operatorname{ker}(\beta) \cap \operatorname{im}(\alpha), n_{2} \in \operatorname{im}(\alpha), n_{3} \in N_{k}, n_{4} \in \operatorname{ker}(\beta)+\operatorname{im}(\alpha)$. Then

\begin{align*}
\overline{\overline{N}}(b)\left(n_{1},\overline{\pi}_{3}(n_{2}), \overline{\pi}_{4}(n_{3}), \overline{\pi}_{5}(n_{4})\right)&=c_{k}^{-1}\pi_{be_{k}}\overline{i}_{3}\theta_{2}\overline{\pi}_{3}(n_{2})+c_{k}^{-1}\pi_{be_{k}}\overline{j}\overline{\sigma}\theta_{3}\overline{\pi}_{4}(n_{3}) \\
&=c_{k}^{-1}b \cdot n_{2} + c_{k}^{-1}b \cdot n_{3} \\
&=c_{k}^{-1}\psi_{0}^{-1}(b)n_{2}+c_{k}^{-1}\psi_{0}^{-1}(b)n_{3} \\
&=-c_{k}^{-1}bn_{2}-c_{k}^{-1}bn_{3}
\end{align*}

On the other hand \\
\begin{align*}
\pi_{be_{k}} \beta \psi(n) &= \pi_{be_{k}} \beta \left(-n_{1}-c_{k}^{-1}i_{2}\sigma_{1} \overline{\pi}_{3}(n_{2})-c_{k}^{-1}\sigma_{3} \overline{\pi}_{4}(n_{3})-i_{3}\sigma_{2}\overline{\pi}_{5}(n_{4})\right)
\end{align*}

\vspace{0.1in}

Since $n_{1} \in \operatorname{ker}(\beta)$ and $\operatorname{im}(\sigma_{2}) \subseteq \operatorname{ker}(\beta)$ we see that the above expression is equal to \\
\begin{align*}
\pi_{be_{k}} \beta(-c_{k}^{-1}i_{2}\sigma_{1}\overline{\pi}_{3}(n_{2})-c_{k}^{-1}\sigma_{3}\overline{\pi}_{4}(n_{3})) &=\pi_{be_{k}}\beta(-c_{k}^{-1}n_{2}-c_{k}^{-1}n_{3}) \\
&=-c_{k}^{-1}bn_{2}-c_{k}^{-1}bn_{3}
\end{align*}

and \ref{5.13} follows. Finally, we extend $\psi$ to an isomorphism of $F$-vector spaces $\psi': \overline{\overline{N}} \rightarrow N$ defined as the identity map on $\displaystyle \bigoplus_{i \neq k} \overline{\overline{N}}_{i}$.

\end{proof}

\section{Nearly Morita equivalence} \label{sec6}

Throughout this section, $\mathcal{F}_{S}(M)/R(P)$-$\operatorname{mod}_{k}$ denotes the category of finite dimensional left $\mathcal{F}_{S}(M)/R(P)$-modules modulo the ideal of morphisms which factor through direct sums of the left $\mathcal{F}_{S}(M)$-simple module at $k$. \\

In this section we prove that the categories $\mathcal{F}_{S}(M)/R(P)-\operatorname{mod}_{k}$ and  $\mathcal{F}_{S}(\overline{\mu}_{k}M)/R(\overline{\mu}_{k}P)-\operatorname{mod}_{k}$ are equivalent, where $(\mathcal{F}_{S}(\overline{\mu}_{k}M),\mathcal{F}_{S}(\overline{\mu}_{k}P))$ denotes the mutation at $k$ in the sense of \cite[p.56]{2}. \\

For each finite dimensional left $\mathcal{F}_{S}(M)/R(P)$-module, we choose splitting data $(p_{N},(\sigma_{2})_{N})$. Let $u: N \rightarrow N^{1}$ be a morphism of left $\mathcal{F}_{S}(M)/R(P)$-modules. Then $u$ induces $D_{k}$-linear maps:

\begin{align*}
u_{out}&:N_{out} \rightarrow N_{out}^{1} \\
u_{in}&: N_{in} \rightarrow N_{in}^{1}
\end{align*}

such that for each $a,a_{1} \in _{k}T, r,r_{1} \in L(k)$

\begin{center}
$\pi_{r_{1}a_{1}}^{1}u_{in}\xi_{ra}=\delta_{r_{1}a_{1},ra}u$
\end{center}

and for each $b,b_{1} \in T_{k}, r,r_{1} \in L(k)$

\begin{center}
$\pi_{b_{1}r_{1}}^{1}u_{out}\xi_{br}=\delta_{b_{1}r_{1},br}u$
\end{center}

We also have $D_{k}$-linear maps
\begin{align*}
& \alpha: N_{in} \rightarrow N_{k}; \alpha_{1}: N_{in}^{1} \rightarrow N_{k}^{1} \\
& \beta: N_{k} \rightarrow N_{out}; \beta_{1}: N_{k}^{1} \rightarrow N_{out}^{1} \\
& \gamma: N_{out} \rightarrow N_{in}; \gamma_{1}: N_{out}^{1} \rightarrow N_{in}^{1}
\end{align*}

Then we have the following equalities
\begin{align*}
& u_{k}\alpha=\alpha_{1}u_{in}; u_{out}\beta=\beta_{1}u_{k} \\
& u_{in}\gamma=\gamma_{1}u_{out}
\end{align*}

The map $u_{out}$ induces $D_{k}$-linear maps
\begin{align*}
&u_{1}: \operatorname{ker}(\gamma) \rightarrow \operatorname{ker}(\gamma_{1}) \\
&u_{4}: \operatorname{im}(\beta) \rightarrow \operatorname{im}(\beta_{1})
\end{align*}

The map $u_{in}$ induces $D_{k}$-linear maps
\begin{align*}
&u_{2}: \operatorname{im}(\gamma) \rightarrow \operatorname{im}(\gamma_{1}) \\
&u_{3}: \operatorname{ker}(\alpha) \rightarrow \operatorname{ker}(\alpha_{1})
\end{align*}

The map $u_{1}$ induces a $D_{k}$-linear map
\begin{center}
$\underline{u}_{1}: \operatorname{ker}(\gamma)/\operatorname{im}(\beta) \rightarrow \operatorname{ker}(\gamma_{1})/\operatorname{im}(\beta_{1})$
\end{center}

\vspace{0.1in}

likewise, the map $u_{3}$ induces a $D_{k}$-linear map
\begin{center}
$\underline{u}_{3}: \operatorname{ker}(\alpha)/\operatorname{im}(\gamma) \rightarrow \operatorname{ker}(\alpha_{1})/\operatorname{im}(\gamma_{1})$
\end{center}

so we have a $D_{k}$-linear map
\begin{center}
$h=\underline{u}_{1} \oplus u_{2} \oplus \underline{u}_{3}: \overline{N}_{k} \rightarrow \overline{N^{1}}_{k}$
\end{center}

Then we define a linear map of left $S$-modules:

\begin{center}
$\overline{u}: \overline{N} \rightarrow \overline{N^{1}}$
\end{center}

as $\overline{u}_{i}=u_{i}$ for every $i \neq k$ and $\overline{u}_{k}=h$.

\begin{definition} Following \cite{7} we say that $u \in \operatorname{Hom}_{S}(_{S}L_{1},_{S}L_{2})$ is confined to $k$ if $u_{i}: e_{i}L_{1} \rightarrow e_{i}L_{2}$ is the zero map for all $i \neq k$. Note that a map of left $\mathcal{F}_{S}(M)$-modules $u: L_{1} \rightarrow L_{2}$ factors through a direct sum of the left $\mathcal{F}_{S}(M)$-simple module at $k$ if and only if $u$ is confined to $k$.
\end{definition}

\begin{lemma} Let $u: N \rightarrow N^{1}$ be a map of finite dimensional left $\mathcal{F}_{S}(M)/R(P)$-modules. Then there exists a map of left $S$-modules $\overline{\rho(u)}: \overline{N} \rightarrow \overline{N^{1}}$, confined to $k$, and such that $\overline{u}+\overline{\rho(u)}$ is a map of left $\mathcal{F}_{S}(\mu_{k}M)$-modules.
\end{lemma}

\begin{proof} Let $(p=p_{N},\sigma_{2}=(\sigma_{2})_{N})$ and $(p_{1},\sigma_{2,1})$ be the splitting data for $N$ and $N^{1}$, respectively. Then we have the following commutative diagram with exact rows:

\begin{center}
 \begin{tabular}{c}
\xymatrix{
0 \ar[r] & \operatorname{im}(\gamma) \ar[r] \ar[d]^{u_{2}}& \operatorname{ker}(\alpha) \ar[r]^(.35){\pi_{2}} \ar[d]^{u_{3}} & \operatorname{ker}(\alpha)/ \operatorname{im}(\gamma) \ar[r] \ar[d]^{\underline{u}_{3}} & 0 \\
0 \ar[r] & \operatorname{im}(\gamma_{1}) \ar[r] & \operatorname{ker}(\alpha_{1}) \ar[r]^(.35){\pi^{1}_{2}} & \operatorname{ker}(\alpha_{1})/ \operatorname{im}(\gamma_{1}) \ar[r] & 0 
}
\end{tabular}
\end{center}

thus $\pi_{2}^{1}(u_{3}\sigma_{2}-\sigma_{2,1}\underline{u}_{3})=\underline{u}_{3}\pi_{2}\sigma_{2}-\underline{u}_{3}=\underline{u}_{3}-\underline{u}_{3}=0$. Hence there exists a linear map of left $S$-modules:

\begin{center}
$\rho_{1}=u_{3}\sigma_{2}-\sigma_{2,1}\underline{u}_{3}: \operatorname{ker}(\alpha)/\operatorname{im}(\gamma) \rightarrow \operatorname{im}(\gamma_{1})$.
\end{center}

Similarly, we have a commutative diagram with exact rows:
\begin{center}
 \begin{tabular}{c}
\xymatrix{
0 \ar[r] & \operatorname{ker}(\gamma) \ar[r]^{j} \ar[d]^{u_{1}}& N_{out} \ar[r]^{\gamma} \ar[d]^{u_{out}} & \operatorname{im}(\gamma) \ar[r] \ar[d]^{u_{2}} & 0 \\
0 \ar[r] & \operatorname{ker}(\gamma_{1}) \ar[r]^{j_{1}} & N_{out^{1}} \ar[r]^{\gamma_{1}} & \operatorname{im}(\gamma_{1}) \ar[r] & 0
}
\end{tabular}
\end{center}

hence $(u_{1}p-p_{1}u_{out})j=u_{1}pj-p_{1}u_{out}j=u_{1}-p_{1}j_{1}u_{1}=u_{1}-u_{1}=0$. \\ 

Therefore, there exists a linear map of left $S$-modules:
\begin{center}
$\rho_{2}: \operatorname{im}(\gamma) \rightarrow \operatorname{ker}(\gamma_{1})$
\end{center}

such that $\rho_{2}\gamma=p_{1}u_{out}-u_{1}p$. \\

Define $\overline{\rho(u)}: \overline{N} \rightarrow \overline{N^{1}}$ as follows: $\overline{\rho(u)}_{i}=0$ for all $i \neq k$ and $\overline{\rho(u)}_{k}: \overline{N}_{k} \rightarrow \overline{N^{1}}_{k}$ is the following linear map of left $S$-modules:
\[
\begin{bmatrix}
0 & \pi_{1}^{1}\rho_{2} & 0 \\
0 & 0 & \rho_{1} \\
0 & 0 & 0
\end{bmatrix}
\]

Let us show that the map $\overline{u}+\overline{\rho(u)}$ is in fact a map of left $\mathcal{F}_{S}(\mu_{k}M)$-modules. \\

Let $n \in N_{\sigma(b)}$. On one hand \\
\begin{align*}
\left(\overline{u}+\overline{\rho(u)}\right)\left(\overline{N}(^{\ast}b)(n)\right)&=\left(-\underline{u}_{1}\pi_{1}p\xi_{be_{k}}(n)-\pi_{1}^{1}\rho_{2}\gamma \xi_{be_{k}}(n),-u_{2}\gamma \xi_{be_{k}}(n),0\right) \\
&=(-\pi_{1}^{1}u_{1}p\xi_{be_{k}}(n)-\pi_{1}^{1}\rho_{2}\gamma \xi_{be_{k}}(n), -u_{2}\gamma \xi_{be_{k}}(n),0) \\
\end{align*}

On the other hand \\
\begin{align*}
\overline{N^{1}}(^{\ast}b)(u_{\sigma(b)}(n))&=(-\pi_{1}^{1}p_{1}\xi_{be_{k}}(u_{\sigma(b)}(n)),-\gamma_{1}\xi_{be_{k}}(u_{\sigma(b)}(n)),0) \\
&=(-\pi_{1}^{1}p_{1}u_{out}(\xi_{be_{k}}(n)),-\gamma_{1}u_{out}(\xi_{be_{k}}(n)),0) \\
&=(-\pi_{1}^{1}u_{1}p\xi_{be_{k}}(n)-\pi_{1}^{1}\rho_{2}\gamma\xi_{be_{k}}(n),-u_{2}\gamma\xi_{be_{k}}(n),0) 
\end{align*}

Therefore

\begin{center}
$\left(\overline{u}+\overline{\rho(u)}\right)\overline{N}(^{\ast}b)=\overline{N^{1}}(^{\ast}b)u_{\sigma(b)}$
\end{center}

Now for each $a \in T_{k}$, $x \in \operatorname{ker}(\gamma)/\operatorname{im}(\gamma)$, $y \in \operatorname{im}(\gamma)$ and $z \in \operatorname{ker}(\alpha)/\operatorname{im}(\gamma)$ we have \\
\begin{align*}
u_{\tau(a)}\overline{N}(a^{\ast})(x,y,z)&=c_{k}^{-1}u_{\tau(a)}(\pi_{e_{k}a}i(y)+\pi_{e_{k}a}j'\sigma_{2}(z)) \\
&=c_{k}^{-1}(\pi^{1}_{e_{k}a}i_{1}u_{2}(y)+\pi^{1}_{e_{k}a}j_{1}'u_{3}\sigma_{2}(z))
\end{align*}

On the other hand

\begin{align*}
\overline{N^{1}}(a^{\ast})\left(\overline{u}+\overline{\rho(u)}\right)(x,y,z)&=\overline{N^{1}}(a^{\ast})\left((\overline{u}_{1}(x),u_{2}(y),\overline{u}_{3}(z))+(\pi_{1}^{1}\rho_{2}(y),\rho_{1}(z),0)\right) \\
&=c_{k}^{-1}\left(\pi^{1}_{e_{k}a}i_{1}u_{2}(y)+\pi^{1}_{e_{k}a}j_{1}'\sigma_{2,1}\underline{u}_{3}(z)+\pi^{1}_{e_{k}a}i_{1}\rho_{1}(z)\right) \\
&=c_{k}^{-1}\left(\pi^{1}_{e_{k}a}i_{1}u_{2}(y)+\pi^{1}_{e_{k}a}j_{1}'(\sigma_{2,1}\underline{u}_{3}(z)+\rho_{1}(z))\right) \\
&=c_{k}^{-1}\left(\pi^{1}_{e_{k}a}i_{1}u_{2}(y)+\pi^{1}_{e_{k}a}j_{1}'u_{3}\sigma_{2}(z)\right)
\end{align*}

thus $u_{\tau(a)}\overline{N}(a^{\ast})=\overline{N^{1}}(a^{\ast})\left(\overline{u}+\overline{\rho(u)}\right)$, as was to be shown.
\end{proof}

\begin{prop} \label{prop8} There exists a faithful functor $\widetilde{\mu}_{k}: \mathcal{F}_{S}(M)/R(P)-\operatorname{mod}_{k} \rightarrow  \mathcal{F}_{S}(\mu_{k}M)/R(\mu_{k}P)-\operatorname{mod}_{k}$
\end{prop}

\begin{proof} First we define a functor $G: \mathcal{F}_{S}(M)/R(P)-\operatorname{mod} \rightarrow  \mathcal{F}_{S}(\mu_{k}M)/R(\mu_{k}P)-\operatorname{mod}_{k}$ as $G(N)=\overline{N}$ and given a map of left $\mathcal{F}_{S}(M)/R(P)$-modules $u: N \rightarrow N^{1}$, we define:

\begin{center}
$G(u)=\underline{\overline{u}+\overline{\rho(u)}}: \overline{N} \rightarrow \overline{N^{1}}$
\end{center}

On the other hand, if $v: N^{1} \rightarrow N^{2}$ is a map of left $\mathcal{F}_{S}(M)/R(P)$-modules then $G(vu)=\underline{\overline{vu}+\overline{\rho(vu)}}$. Since $\rho$ is concentrated in $k$ and $\overline{vu}=\overline{v} \ \overline{u}$ one sees that $G(vu)=G(v)G(u)$ so that $G$ preserves composition. Since $\overline{\rho(id_{N})}=0$ then $G(id_{N})=id_{\overline{N}}$ so that $G$ is indeed a covariant (additive) functor.  \\

Finally, note that $G(u)=0$ if and only if $\overline{u}+\overline{\rho(u)}$ is confined to $k$, which happens if and only if $\overline{u}$ is confined to $k$ and the latter happens if only if $u$ is confined to $k$, as was to be shown.
\end{proof}

Let $\varphi: \mathcal{F}_{S}(M) \rightarrow \mathcal{F}_{S}(M_{1})$ be an algebra isomorphism such that $\varphi_{|S}=id_{S}$. Let $P$ be a potential in $\mathcal{F}_{S}(M)$. \\ 

Throughout the rest of this section, $J(M,P)$ will denote the quotient algebra $\mathcal{F}_{S}(M)/R(P)$. \\

The isomorphism $\varphi$ induces a functor \\
 
\begin{center}
$H_{\varphi}: J(M,P)-\operatorname{mod} \rightarrow J(M_{1},\varphi(P))-\operatorname{mod}$
\end{center}

\vspace{0.1in}

given as follows. In objects, $H_{\varphi}(N)=^{\varphi}N$; that is, $H_{\varphi}(N)=N$ as $S$-left modules, and given $n \in N$ and $z \in \mathcal{F}_{S}(M_{1})$, $z \cdot n=\varphi^{-1}(z)n$. Clearly, $\operatorname{Hom}_{J(M,P)}(N,N^{1})=\operatorname{Hom}_{J(M_{1},\varphi(P))}(^{\varphi}N,^{\varphi}N^{1})$. Therefore, we let $H_{\varphi}(u)=u$. This gives an equivalence of categories 

\begin{equation} \label{6.1}
H_{\varphi}: J(M,P)-\operatorname{mod} \rightarrow J(M_{1},\varphi(P))-\operatorname{mod}
\end{equation}

Now suppose that $M=M_{1} \oplus M_{2}$, and $Q=Q_{1}+W$ where $Q_{1}$ is a reduced potential in $\mathcal{F}_{S}(M_{1})$ and $W$ is a trivial potential in $\mathcal{F}_{S}(M_{2})$. Then the restriction functor \\

\begin{equation} \label{6.2}
\operatorname{res}: J(M,Q)-\operatorname{mod} \rightarrow J(M_{1},Q_{1})-\operatorname{mod}
\end{equation}

yields also an equivalence of categories. \\

On the other hand, by \cite[Theorem 8.17]{2} there exists a right-equivalence \\

\begin{center}
$\psi: \mathcal{F}_{S}(\mu_{k}^2M) \rightarrow \mathcal{F}_{S}(M \oplus M')$
\end{center}

such that $\psi(\mu_{k}^2P)$ is cyclically equivalent to $P+W$ where $W$ is a trivial potential in $\mathcal{F}_{S}(M')$.

Thus, using \ref{6.1} and \ref{6.2} we obtain equivalence of categories:

\begin{align}
H_{\psi}: J(\mu_{k}^2M,\mu_{k}^2P)-\operatorname{mod} \rightarrow J(M \oplus M',P+W)-\operatorname{mod} \label{6.3} \\
res: J(M \oplus M',P+W)-\operatorname{mod} \rightarrow J(M,P)-\operatorname{mod} \label{6.4}
\end{align}

\vspace{0.1in}

composing the  above functors yields an equivalence of categories \\
\begin{equation} \label{6.5}
G(\psi)=resH_{\psi}: J(\mu_{k}^2M,\mu_{k}^2P)-\operatorname{mod} \rightarrow J(M,P)-\operatorname{mod}
\end{equation}

In what follows, given $N \in J(M,P)-\operatorname{mod}$, we will denote by $_{S}N$ the $S$-left module underlying $N$. In particular, $_{S}G(\psi)(N)=_{S}N$.

\begin{lemma} \label{lem11} Let $\mathcal{A},\mathcal{B}$ be $F$-categories and let $C: \mathcal{A} \rightarrow \mathcal{B}$, $D: \mathcal{B} \rightarrow \mathcal{A}$ be $F$-functors such that $D$ is faithful and there exists a natural isomorphism $id_{\mathcal{A}} \cong DC$. Then $C$ is fully faithful. Moreover, if $D$ is full, then $id_{\mathcal{B}} \cong CD$.
\end{lemma}

\begin{proof} For each $X \in \operatorname{Ob}(\mathcal{A})$ there exists an isomorphism
\begin{center}
$\phi_{X}: X \rightarrow DC(X)$
\end{center}

Now let $u \in \operatorname{Hom}_{\mathcal{A}}(X,X_{1})$. By naturality, we have a commutative diagram
\begin{center}
\begin{equation*}
\xymatrix{%
X \ar[r]^{\phi_{X}} \ar[d]^{u} & DC(X) \ar[d]^{DC(u)} \\
X_{1} \ar[r]^{\phi_{X_{1}}} & DC(X_{1})}
\end{equation*}
\end{center}
Thus $\phi_{X_{1}}u=DC(u)\phi_{X}$. Therefore if $C(u)=0$, then $u=0$. This shows that $C$ is faithful. 
Now let $h \in \operatorname{Hom}_{\mathcal{B}}(C(X),C(X_{1}))$. Let
\begin{center}
$\lambda=\phi_{X_{1}}^{-1}D(h)\phi_{X}: X \rightarrow X_{1}$
\end{center}

Since the following diagram commutes

\begin{center}
\begin{equation*}
\xymatrix{%
X \ar[r]^{\phi_{X}} \ar[d]^{\lambda} & DC(X) \ar[d]^{DC(\lambda)} \\
X_{1} \ar[r]^{\phi_{X_{1}}} & DC(X_{1})}
\end{equation*}
\end{center}

then $\lambda=\phi_{X_{1}}^{-1}DC(\lambda)\phi_{X}$. Therefore $D(h)=DC(\lambda)$ and thus $h=C(\lambda)$. It follows that $C$ is full.  
Now suppose that $D$ is full, then for each $Y \in \operatorname{Ob}(\mathcal{B})$ there exists an isomorphism $\phi_{D(Y)}: D(Y) \rightarrow DCD(Y)$. Since $D$ is full, then $\phi_{D(Y)}=D(\psi_{Y})$ for some $\psi_{Y} \in \operatorname{Hom}_{\mathcal{B}}(Y,CD(Y))$. Let us show that $\psi_{Y}$ is natural. Let $f \in \operatorname{Hom}_{\mathcal{B}}(Y_{1},Y_{2})$. We have to prove the following diagram is commutative

\begin{center}
\begin{equation} \label{6.6}
\xymatrix{%
Y \ar[r]^{\psi_{Y}} \ar[d]^{f} & CD(Y) \ar[d]^{CD(f)} \\
Y_{1} \ar[r]^{\psi_{Y_{1}}} & CD(Y_{1})}
\end{equation}
\end{center}
 
Consider the following commutative diagram
\begin{center}
\begin{equation*}
\xymatrix{%
D(Y) \ar[r]^{\phi_{D(Y)}} \ar[d]^{D(f)} & DCD(Y) \ar[d]^{DCD(f)} \\
D(Y_{1}) \ar[r]^{\phi_{D(Y_{1})}} & DCD(Y_{1})}
\end{equation*}
\end{center}

thus $DCD(f)\phi_{D(Y)}=\phi_{D(Y_{1})}D(f)$. This implies that $DCD(f)D(\psi_{Y})=D(\psi_{Y_{1}})D(f)$. Because $D$ is faithful it follows that $CD(f)\psi_{Y}=\psi_{Y_{1}}f$ and \ref{6.6} commutes, as was to be shown.

\end{proof}

By \ref{6.5} there exists an equivalence of categories \\
\begin{center}
$G(\psi)=resH_{\psi}: J(\mu_{k}^2M,\mu_{k}^2P)-\operatorname{mod} \rightarrow J(M,P)-\operatorname{mod}$
\end{center}

and this functor descends to a functor $G(\psi)_{k}$ in the quotient category $J(M,P)-\operatorname{mod}_{k}$ which is the category $J(M,P)-\operatorname{mod}$, modulo the ideal of morphisms which factor through direct sums of the simple module at $k$. Thus, we have a functor \\

\begin{center}
$G(\psi)_{k}: J(\mu_{k}^{2}M,\mu_{k}^{2}P)-\operatorname{mod}_{k} \rightarrow J(M,P)-\operatorname{mod}_{k}$
\end{center}

\begin{prop} \label{prop9} There exists a natural isomorphism of functors $id_{J(M,P)-\operatorname{mod}_{k}} \cong G(\psi)_{k}\tilde{\mu}_{k}^{2}$.
\end{prop}

\begin{proof} Let $u \in \operatorname{Hom}_{J(M,P)-\operatorname{mod}_{k}}(N,N^{1})$. Remembering the proof of Proposition \ref{prop8} we have
\begin{align*}
\tilde{\mu}_{k}(u)&=\underline{\overline{u}+\overline{\rho(u)}}=u_{1}: \overline{N} \rightarrow \overline{N^{1}} \\
\tilde{\mu}_{k}^{2}(u)=\tilde{\mu}_{k}(u_{1})&=\underline{\overline{u_{1}}+\overline{\rho(u_{1})}}=u_{2}: \overline{\overline{N}} \rightarrow \overline{\overline{N^{1}}}
\end{align*}
Using the notation introduced in the proof of Theorem \ref{theo2}, we have isomorphisms of $J(\mu_{k}^{2}M,\mu_{k}^{2}P)$-left modules
\begin{align*}
&\psi': G(\psi)_{k}\tilde{\mu}_{k}^{2}N \rightarrow N \\
&\psi_{1}': G(\psi)_{k}\tilde{\mu}_{k}^{2} N^{1} \rightarrow N^{1}
\end{align*}

It remains to show that the following diagram commutes in $J(M,P)-\operatorname{mod}_{k}$.
\begin{center}
\begin{equation} \label{6.7}
\xymatrix{%
G(\psi)_{k}\tilde{\mu}_{k}^{2}N \ar[r]^(.65){\psi'} \ar[d]^{u_{2}} & N \ar[d]^{u} \\
G(\psi)_{k}\tilde{\mu}_{k}^{2}N \ar[r]^(.65){\psi_{1}'} & N^{1}}
\end{equation}
\end{center}

but this is true since $u\psi'-\psi_{1}'u_{2}$ is confined to $k$. This completes the proof.
\end{proof}

\begin{prop} \label{prop10} There exists a natural isomorphism of functors \\

\begin{center}
$\tilde{\mu}_{k}G(\psi)_{k}\tilde{\mu}_{k} \cong id_{J(\mu_{k}M,\mu_{k}P)-\operatorname{mod}_{k}}$.
\end{center}
\end{prop}

\begin{proof} By Proposition \ref{prop8}, the functor

\begin{center}
$\tilde{\mu}_{k}: J(\mu_{k}M,\mu_{k}P)-\operatorname{mod}_{k} \rightarrow J(\mu_{k}^{2}M,\mu_{k}^{2}P)-\operatorname{mod}_{k}$
\end{center}

is faithful, hence $G(\psi)_{k}\tilde{\mu}_{k}$ is faithful as well. By Lemma \ref{lem11} and Proposition \ref{prop8}, the functor $\tilde{\mu}_{k}$ is fully faithful. Therefore, $G(\psi)_{k}\tilde{\mu}_{k}$ is full. The result now follows by applying Lemma \ref{lem11}. 
\end{proof}

\begin{theorem} Let $P$ be a potential in $\mathcal{F}_{S}(M)$. If $\mu_{k}P$ is splittable, then there exists an equivalence of categories:

\begin{center}
$\mu_{k}: J(M,P)-\operatorname{mod}_{k} \rightarrow J(\overline{\mu}_{k}M,\overline{\mu}_{k}P)-\operatorname{mod}_{k}$
\end{center}

\end{theorem}

\begin{proof}
Since $\mu_{k}P$ is splittable, then by \cite[Theorem 7.15]{2} there exists an algebra isomorphism $\varphi: \mathcal{F}_{S}(\mu_{k}M) \rightarrow \mathcal{F}_{S}(\overline{\mu}_{k}M \oplus M')$, with $\varphi_{|S}=id_{S}$, and such that $\varphi(\mu_{k}P)$ is cyclically equivalent to $\overline{\mu}_{k}P + W$ where $W$ is a trivial potential in $\mathcal{F}_{S}(M')$. By \ref{6.1} and \ref{6.2} there exists an equivalence of categories
\begin{center}
$G(\varphi): J(\mu_{k}M,\mu_{k}P)-\operatorname{mod} \rightarrow J(\overline{\mu}_{k}M,\overline{\mu}_{k}P)-\operatorname{mod}$
\end{center}

which induces an equivalence of categories
\begin{center}
$G(\varphi)_{k}: J(\mu_{k}M,\mu_{k}P)-\operatorname{mod}_{k} \rightarrow J(\overline{\mu}_{k}M,\overline{\mu}_{k}P)-\operatorname{mod}_{k}$
\end{center}

By Propositions \ref{prop9} and \ref{prop10}, the categories $J(M,P)-\operatorname{mod}_{k}$ and $J(\mu_{k}M,\mu_{k}P)-\operatorname{mod}_{k}$ are equivalent; hence, we get an equivalence of categories

\begin{center}
$\mu_{k}: J(M,P)-\operatorname{mod}_{k} \rightarrow J(\overline{\mu}_{k}M,\overline{\mu}_{k}P)-\operatorname{mod}_{k}$
\end{center}

as desired.

\end{proof}


\begin{thebibliography}{8}

\bibitem{1} M. Auslander, I. Reiten and S. Smal\o{}, \textit{Representation theory of Artin algebras}. Cambridge Studies in Advanced Mathematics, 36. Cambridge University Press, Cambridge, (1997).

\bibitem{2} R. Bautista and D. L\'{o}pez-Aguayo, \textit{Potentials for some tensor algebras}. \href{https://arxiv.org/abs/1506.05880}{arXiv:1506.05880}.

\bibitem{3} R. Bautista, L. Salmer\'{o}n and R. Zuazua, \textit{Differential tensor algebras and their module categories}. London Mathematical Society Lecture Note Series, 362. Cambridge University Press, Cambridge, (2009).

\bibitem{4} A.B. Buan, O. Iyama, I. Reiten and D. Smith, \textit{Mutation of cluster-tilting objects and potentials}. American Journal of Mathematics, 133 (2011), no. 4, 835-887. \href{http://arxiv.org/abs/0804.3813}{arXiv:0804.3813}.

\bibitem{5} L. Demonet, \textit{Mutations of group species with potentials and their representations. Applications to cluster algebras}. \href{http://arxiv.org/abs/1003.5078}{arXiv:1003.5078}.

\bibitem{6} H. Derksen, J. Weyman, A. Zelevinsky. \textit{Quivers with potentials and their representations I: Mutations}. Selecta Math. \textbf{14} (2008), no. 1, 59-119. \href{https://arxiv.org/abs/0704.0649}{arXiv:0704.0649}.

\bibitem{7} H. Derksen, J. Weyman, A. Zelevinsky. \textit{Quivers with potentials and their representations II: Applications to cluster algebras}. J. Amer. Math. Soc.  \textbf{23} (2010), no. 3, 749-790. \href{https://arxiv.org/abs/0904.0676}{arXiv:0904.0676}.

\bibitem{8} D. Labardini-Fragoso, A. Zelevinsky. \textit{Strongly primitive species with potentials I: Mutations}. Bolet\'{i}n de la Sociedad Matem\'{a}tica Mexicana (Third series), Vol. \textbf{22}, (2016), Issue 1, 47-115. \url{http://arxiv.org/pdf/1306.3495v1.pdf}

\end{thebibliography}
\end{document}